\documentclass{article}
\usepackage[utf8]{inputenc}
\usepackage{comment}
\usepackage{tikz}
\usepackage{amsmath}
\usepackage{amssymb}
\usepackage{verbatim}
\usepackage{amsthm}
\usepackage{amsrefs}
\usepackage{quiver}
\usepackage{tikz-cd}
\usepackage{enumitem}
\usepackage[skip=10pt plus1pt, indent=0pt]{parskip}
\usepackage{soul}

\newcommand{\bbX}{\ensuremath{\mathbb{X}}}
\newcommand{\Weil}{{\ensuremath{\mathbb W \mathrm{eil}}}}
\newcommand{\End}{\ensuremath{{\mathbb E \mathrm{nd}}}}
\newcommand{\Fun}{\ensuremath{\mathrm{Fun}}}
\newcommand{\DObj}{\ensuremath{\mathrm{DObj}}}
\newcommand{\DDObj}{\ensuremath{\mathrm{-DObj}}}
\newcommand{\DBun}{\ensuremath{\mathrm{DBun}}}
\newcommand{\DDBun}{\ensuremath{\mathrm{-DBun}}}

\newcommand{\Set}{\ensuremath{\mathrm{Set}}}

\newtheorem{theorem}{Theorem}[section]
\newtheorem{definition}[theorem]{Definition}
\newtheorem{remark}[theorem]{Remark}

\newtheorem{corollary}[theorem]{Corollary}
\newtheorem{example}[theorem]{Example}
\newtheorem{lemma}[theorem]{Lemma}

\newtheorem{notation}[theorem]{Notation}
\newtheorem{proposition}[theorem]{Proposition}

\newcommand{\WEIL}{\ensuremath{\mathbb W \mathrm{eil}}}
\newcommand{\weil}{{\ensuremath{\mathbb W \mathrm{eil}}}}

\newcommand{\Hom}{\mathrm{Hom}}

\newcommand{\Eval}{\ensuremath{\mathrm{Eval}}}
\newcommand{\Disc}{\ensuremath{\mathrm{Disc}}}

\newcommand{\nice}{\text{differential}}
\newcommand{\Nice}{Differential}
\newcommand{\DiffFun}{\ensuremath{\mathbb D \mathrm{iffFun}}}
\newcommand{\lin}{\ensuremath{\mathrm{lin}}}
\newcommand{\add}{\mathrm {add}}
\newcommand{\strong}{\mathrm {strong}}
\newcommand{\lax}{\mathrm {lax}}
\newcommand{\oplax}{\mathrm {oplax}}
\newcommand{\TANG}{\ensuremath{{\sf{TANG}}}}

\newcommand{\aACT}{\ensuremath{{\sf{-ACT}}}}

\newcommand{\Induce}{\ensuremath{\mathrm{Ind}}}
\newcommand{\Ind}{\ensuremath{\mathrm{Ind}}}
\newcommand{\SmMan}{\ensuremath{\mathbb S \mathrm{mMan}}}
\newcommand{\FinSet}{{\ensuremath{\mathbb F \mathrm{inSet}}}}
\newcommand{\Top}{{\ensuremath{\mathbb T \mathrm{op}}}}
\newcommand{\thatalpha}{distributor}
\newcommand{\eval}{\ensuremath{\mathrm{eval}}}

\newcommand{\bbI}{\mathbf 1}
\newcommand{\T}{T}
\newcommand{\Te}{T_\bullet}

\usepackage[english]{babel}

\usepackage[letterpaper,top=2cm,bottom=2cm,left=3cm,right=3cm,marginparwidth=1.75cm]{geometry}

\usepackage{amsmath}
\usepackage{graphicx}
\usepackage[colorlinks=true, allcolors=blue]{hyperref}
\title{Differential bundles as functors from free modules}
\author{Florian Schwarz}
\begin{document}
\maketitle
\begin{abstract}
    This paper explores differential bundles in tangent categories, characterizing them as functors from a structure category. This is analogous to the actegory perspective of Garner and Leung, which we also use to describe the tangent categories of Rosický, Cockett and Cruttwell. We generalize the Garner-Leung equivalence between tangent categories and Weil algebra actegories to include lax functors and non-linear natural transformations.

    The main result of this paper, is that \nice{} functors between the structure category $\mathbb N^\bullet$ and a tangent category $\mathbb X$ are equivalent to differential bundles in $\mathbb X$. 

    We obtain this result by showing that evaluating a \nice{} functor on the generating object $\mathbb N^1$ of the structure category $\mathbb N^\bullet$ produces a differential bundle in a functorial way. Every differential bundle can be obtained this way. We show that obtaining such a functor from a bundle is a functorial construction.
    
    There are variations of these results for linear and additive morphisms of differential bundles.
\end{abstract}
\tableofcontents
\section{Introduction}
Geometry is one of the most beautiful areas of mathematics, and also one of the richest sources of intuition.  Many constructions in geometry (e.g. gluing local pieces to obtain a global object) have incarnations outside of geometry as well.  The abstract framework of \textbf{tangent categories} makes it possible to formally encode geometric structures in settings which may not traditionally be considered to be part of the field of geometry.
Tangent categories were first discovered by Ji\v{r}\'i Rosický in 1984 \cite{rosicky} and then rediscoved and reinterpreted by Robin Cockett and Geoff Cruttwell in 2012 \cite{Cockett2014DifferentialST}. In 2017 Poon Leung provided a purely algebraic characterization of tangent categories \cite{Leung2017}, an approach which was further developed by Richard Garner in \cite{Garner2018:embedding_theorem}.
A tangent category admits constructions from differential geometry. Smooth manifolds are an example, but many other categories also have a tangent structure. The main feature of a tangent category is an endofunctor which produces an analogue of the tangent bundle, along with natural transformations which correspond to the usual projection from the bundle to the base space.  There are several other important structured involved in the definition, these are laid out in Definition \ref{def:tangent_cat}.

In this paper, we consider one particular geometric construction in tangent categories: differential bundles.
Differential bundles in tangent categories were first discovered by Cockett and Cruttwell in 2018 \cite{cockett2016diffbundles}. They are a categorical generalization of vector bundles. Since many constructions in differential geometry are given by vector bundles (e.g. tangent and cotangent vectors form a vector bundle, de Rham cohomology is built from a vector bundle, and a Riemannian metric is a section of a certain vector bundle), these bundles are a key component of many constructions that we want to generalize in tangent categories.  As defined in \cite{cockett2016diffbundles}, a differential bundle is a pair of objects $E,M$ in a tangent category together with a projection map $q: E \to M$ satisfying certain axioms.  In particular the fiber, expressed as a certain pullback, has an additive structure that satisfies the conditions listed in Definition \ref{def:differential bundle}. A motivating example is the tangent bundle $T(M) \to M$ given as the endofunctor part of any tangent structure. 

However, the definition makes it difficult to characterize or classify differential bundles.  Unlike other algebraic theories such as internal monoids, the definition of a differential bundle is not given by functors from a certain structure category.  The goal of this paper is to provide exactly this kind of characterization.

To search for a structure category which characterizes differential bundles, we considered the aforementioned  characterization of tangent categories through the category of Weil algebras discovered by Leung \cite{Leung2017} and Garner \cite{Garner2018:embedding_theorem}.  This characterization utilizes actegories,  a generalization of group-actions  defined by Benabou in 1967 in \cite{benabou_intro_to_bicategories}.  Instead of a group acting on a set, an actegory completely describes the action of a monoidal category on another category.  Leung and Garner showed that there is an equivalence between the 2-category of tangent categories and the 2-category of certain Weil algebra actegories.  Later, Kristine Bauer, Matthew Burke and Michael Ching used this characterization of tangent categories via Weil algebra actegories to generalize tangent categories into tangent infinity categories \cite{BauerBurkeChing}.  Importantly, they presented the structure of the subcategory of differential objects in a tangent category (the objects with trivial tangent structure) using the category $\mathbb N^\bullet$ of (free, finite-dimensional) $\mathbb N$ modules. This category, with a slight change, can also be used to characterize differential bundles in tangent categories.

The category $\mathbb N^\bullet$ of $\mathbb N$ modules is a tangent category in which the tangent bundle of a $\mathbb N$-module $\mathbb N^k$ is just the product $\mathbb N^k \times \mathbb N^k$ and the tangent bundle projection is just the projection on the first component $\pi_0 : \mathbb N^k \times \mathbb N^k \to \mathbb N^k$. The key property of the tangent category $\mathbb N^\bullet$ is that $\mathbb N^1 \to \mathbb N^0$ is a differential bundle.  Since every object is a product of copies of $\mathbb N$, this implies that for every object $\mathbb N^k$, the map $\mathbb N^k \to \mathbb N^0$ is part of a differential bundle. We will continue to see that $\mathbb N^\bullet$ encodes the structure of differential bundles exactly.

In order to characterize differential bundles through functors, we define a special kind of functors preserving differential structures, called \nice{} functors. They are defined to be tangent functors that preserve certain pullbacks with certain naturality conditions. We will have two flavors of them, lax and strong differential functors, where the difference is that strong \nice{} functors are also required to preserve the terminal object. 

Proposition \ref{prop:preserve_differential_structure} will then show that \nice{} functors send differential objects (i.e. differential bundles over a terminal object) to differential bundles. This will immediately give us that every \nice{} functor with domain $\mathbb N^\bullet$ induces a differential bundle in the target category.

Conversely, Proposition \ref{prop:induce_on_objects} constructs a \nice{} functor from a differential bundle showing that every differential bundle comes from a \nice{} functor. This leads to Theorem \ref{thm:equivalence_bundle_categories}, an equivalence of categories between differential bundles and \nice{} functors.

This new characterization of differential bundles provides 1-categorical context for Michael Ching's and Kaya Arro's forthcoming work on differential bundles in tangent infinity categories and in particular on the differential bundles arising in Goodwillie functor calculus. They use $E_\infty$, an infinity category analogue of $\mathbb N^\bullet$, to define a differential bundle in a tangent infinity category.

We also hope that this characterization of differential bundles will be useful in further developing the connections between classical differential geometry and the theory of tangent categories.  Since differential bundles are the basic setup for connections \cite{cockett2017connections}, we hope we can use the alternate formulation of differential bundles in this paper to reformulate connections and gain insights into their structure.  In particular, we would like to examine the equivalence between horizontal and vertical connection in tangent categories in \cite{cockett2017connections}, with an eye towards extending this result.  Similarly, we hope to consider dynamical systems in tangent categories using our characterization of differential bundles.  Categorical dynamical systems were defined in \cite{jazmeyers_dynamical} and dynamical systems in tangent categories were defined in \cite{cockett2019diffeq}.  A key part of the structure in \cite{cockett2019diffeq} uses differential objects.  We hope that our characterization of differential objects can be used to make the relationship between \cite{jazmeyers_dynamical} and \cite{cockett2019diffeq} concrete.  See Section \ref{sec:concluding_remarks} for further details.

Finally, we hope that this perspective on differential bundles can help to find a correspondence between differential objects in a tangent category and its opposite.

\textbf{Acknowledgements:}
I would like to thank Michael Ching for suggesting the main idea of this paper. I would also like to thank Robin Cockett for going through the proofs in detail, finding the gaps and helping to fill them. Special thanks go to my supervisor Kristine Bauer for her support, discussing all the key proofs on the board in her office as well as for her support in structuring and communicating the result. I am thankful that my work was funded by the Alberta Innovates Graduate Student Scholarship.

\section{Tangent categories, differential bundles and differential objects}

In this section, we will review the definitions of tangent categories and differential bundles as stated by Cockett and Cruttwell in \cite{Cockett2014DifferentialST} and \cite{cockett2016diffbundles}. We will also define the appropriate notion of functors between tangent categories and natural transformations between such functors.

We introduce several important examples of tangent categories: $\SmMan$ and $\mathbb N^\bullet$ will be extremely important throughout the paper.  \SmMan{} is important because it recovers  classical differential geometry.  The category $\mathbb N^\bullet$ is important because it will be used to characterize differential bundles and objects in an arbitrary tangent category in Sections \ref{sec:induce} and \ref{sec:equivalence}.

\begin{notation}
As a convention, we will use throughout the paper that, given a pullback or, in particular, a product, the projections onto the first and second components will be called $\pi_0$ and $\pi_1$, respectively. 
More generally, the $i$-th projection out of a $n$-fold pullback will be $\pi_{i-1}$. 
The morphism to the pullback of $f$ and $g$ induced by the universal property will be denoted by $\langle f, g \rangle$:
\[\begin{tikzcd}
	Z \\
	& {A \times_BC} & A \\
	& C & B
	\arrow["{\langle f,g\rangle }"{description}, dashed, from=1-1, to=2-2]
	\arrow["f", curve={height=-6pt}, from=1-1, to=2-3]
	\arrow["g"', curve={height=6pt}, from=1-1, to=3-2]
	\arrow["{\pi_0}"', from=2-2, to=2-3]
	\arrow["{\pi_1}", from=2-2, to=3-2]
	\arrow["\lrcorner"{anchor=center, pos=0.125}, draw=none, from=2-2, to=3-3]
	\arrow[from=2-3, to=3-3]
	\arrow[from=3-2, to=3-3]
\end{tikzcd}\]
In the case that $Z = D \times_{B'}E $ is another pullback, we will denote (for $h:D \to A$ and $k: E \to C$ that induce the same map into $B$)
$$
h \times_B k := \langle h \circ \pi_0 , k \circ \pi_1 \rangle : D \times_{B'} E \to A \times_B C.
$$
The most relevant case happens when $B$ and $B'$ are the terminal object and the pullback is therefore a product: in this case, $f \times g := \langle f \circ \pi_0 , g \circ \pi_1 \rangle $.  Note that we are using applicable order for composition here: $f\circ \pi_0$ is the composite of the projection $\pi_0$ followed by $f$.  This is consistent with the notation used in \cite{BauerBurkeChing}, \cite{Garner2018:embedding_theorem}, \cite{Leung2017} and \cite{rosicky} and inconsistent with \cite{cockett2016diffbundles}, \cite{Cockett2014DifferentialST} and \cite{lanfranchi2025tangentdisplaymaps}.  In the case where $f=g$, we use a subscript to denote the pullback of $f$ along itself:

\[\begin{tikzcd}
	{A_2= A \times_BA} & A \\
	 A & B
	\arrow["{\pi_0}"', from=1-1, to=1-2]
	\arrow["{\pi_1}", from=1-1, to=2-1]
	\arrow["\lrcorner"{anchor=center, pos=0.125}, draw=none, from=1-1, to=2-2]
	\arrow["f", from=2-1, to=2-2]
	\arrow["f", from=1-2, to=2-2]
\end{tikzcd}\]
Likewise, $A_n$ denotes the $n$-fold pullback of $f$ along itself.    
\end{notation}

\subsection{Tangent categories}
This section aims to recall notions from \cite{Cockett2014DifferentialST}, namely tangent categories, lax/strong tangent functors between tangent categories and (linear) tangent transformations between tangent functors. Tangent categories were originally defined by Rosick\'y in \cite{rosicky} in a slightly different way; the formulation we use is the one from \cite{Cockett2014DifferentialST}.

\begin{definition}\label{def:additive_bundle}\cite[Definition 2.1]{Cockett2014DifferentialST}%(Additive Bundle) 
Let $\mathbb X$ be a category and $A\in \mathbb X_0$ an object. An \textbf{additive bundle} $(X,A,+, 0,p)$ \textbf{over $A$} consists of the following data:
\begin{itemize}
    \item an object $X$ and a morphism $p: X\rightarrow A$ such that pullback powers $X_n$ of $p$ exist, and 
    \item morphisms $+ : X_2 \rightarrow X$ and $0: A \rightarrow X$. %,  where $X_n$ denotes the $n$-th pullback power of $p$.
\end{itemize}

In addition, the morphisms must satisfy the following conditions:
\begin{itemize}
\item $p \circ += p \circ \pi_0 = p \circ \pi_1$ and $p \circ 0= 1_A$,
\item the morphism $+$ must be associative, commutative and unital.  
\end{itemize}
That is, the following associativity, commutativity and unitality diagrams 
\begin{center}
\begin{tikzpicture}
\path (0,1.5) node(a) {$X_3$}
(2.5,1.5) node (b) {$X_2$}
(0,0) node (c) {$X_2$}
(2.5,0) node (d) {$X$};
\draw [->] (a) -- node[above] {$1 \times_A +$} (b);
\draw [->] (a) -- node[left] {$+ \times_A 1$} (c);
\draw [->] (c) -- node[above] {$+$} (d);
\draw [->] (b) -- node[left] {$+$} (d);
\end{tikzpicture}
\begin{tikzpicture}
\path (0,1.5) node(a) {$X_2$}
(0,0) node (c) {$X_2$}
(2.5,0) node (d) {$X$};
\draw [->] (a) -- node[above] {$+$} (d);
\draw [->] (a) -- node[left] {$\langle \pi_1 , \pi_0 \rangle $} (c);
\draw [->] (c) -- node[above] {$+$} (d);
\end{tikzpicture}
\begin{tikzpicture}
\path (0,1.5) node(a) {$X$}
(0,0) node (c) {$X_2$}
(2.5,0) node (d) {$X$};
\draw [->] (a) -- node[above] {$1_X$} (d);
\draw [->] (a) -- node[left] {$\langle 0 \circ p , 1_X \rangle $} (c);
\draw [->] (c) -- node[above] {$+$} (d);
\end{tikzpicture}
\end{center}
must commute.
\end{definition}

A morphism of additive bundles consists of morphisms that commute with these structures, as explained in the following definition.
\begin{definition}\label{def:additive_bundle_morphism}\cite[Definition 2.2]{Cockett2014DifferentialST}%[Additive Bundle morphism]
For two additive bundles $(X,A, p, + , 0)$ and $(Y,B,p',+',0')$ an \textbf{additive bundle morphism} $(f,g)$ is a pair of maps $f: X \rightarrow Y$ and $g: A \rightarrow B$ such that the following diagrams commute:
$$
\begin{tikzpicture}
\path (0,1.5) node(a) {$X$}
(2.5,1.5) node (b) {$Y$}
(0,0) node (c) {$A$}
(2.5,0) node (d) {$B$};
\draw [->] (a) -- node[above] {$f$} (b);
\draw [->] (a) -- node[left] {$p$} (c);
\draw [->] (c) -- node[above] {$g$} (d);
\draw [->] (b) -- node[left] {$p'$} (d);
\end{tikzpicture}
\begin{tikzpicture}
\path (0,1.5) node(a) {$X_2$}
(2.5,1.5) node (b) {$Y_2$}
(0,0) node (c) {$X$}
(2.5,0) node (d) {$Y$};
\draw [->] (a) -- node[above] {$\langle f \circ \pi_0 , f \circ \pi_1 \rangle$} (b);
\draw [->] (a) -- node[left] {$+$} (c);
\draw [->] (c) -- node[above] {$f$} (d);
\draw [->] (b) -- node[left] {$+'$} (d);
\end{tikzpicture}
\begin{tikzpicture}
\path (0,1.5) node(a) {$A$}
(2.5,1.5) node (b) {$B$}
(0,0) node (c) {$X$}
(2.5,0) node (d) {$Y$};
\draw [->] (a) -- node[above] {$g$} (b);
\draw [->] (a) -- node[left] {$0$} (c);
\draw [->] (c) -- node[above] {$f$} (d);
\draw [->] (b) -- node[left] {$0'$} (d);
\end{tikzpicture}
$$
\end{definition}
Additive bundles are a categorification of the basic structure underlying vector spaces.  This will be helpful in defining tangent categories, since each tangent space in the classical case is a vector space.  In the following definition, beware of the notation $T_2$ (which refers to a pullback) and $T^2$ (which refers to two applications of the functor $T$).  For any category $\mathbb X$, let $1:\mathbb X \to \mathbb X$ denote the identity functor.  

\begin{definition}\cite[Definition 2.3]{Cockett2014DifferentialST}%[Tangent category]
\label{def:tangent_cat}
A \textbf{tangent category} $(\mathbb X , T , p , 0,+, \ell , c)$ consists of a category $\mathbb X$ and the following structures:
\begin{itemize}
\item a functor $T: \mathbb X \rightarrow \mathbb X$, called the tangent functor with a natural transformation $p: T \rightarrow 1$ such that pullback powers of $p_M:TM\to M$ exist for all objects $M \in \mathbb X$ and $T$ preserves these pullback powers;
\item natural transformations $+: T_2 \rightarrow T$ and $0: 1 \rightarrow T$ making each $(TM, M, p_M, +_M, 0_M)$ 
an additive bundle;
\item a natural transformation $\ell: T \rightarrow T^2$, called the vertical lift, such that $(\ell_M, 0_M)$ is an additive bundle morphism from $(TM, M, p_M, +_M, 0_M)$ to $(T^2M, TM, Tp_M, T(+_M), T(0_M))$;
\item a natural transformation $c: T^2 \rightarrow T^2$ with $c^2 = 1$ and $lc = l$, such that $(c_M, 1)$ is an additive bundlem morphism from $(T^2M, TM, Tp_M, T(+_M), T(0_M))$ to $(T^2M, TM, p_{TM}, +_{TM}, 0_{TM})$.
\end{itemize}
The pullbacks $T^n T_k$ are denoted as \textbf{foundational pullbacks}.

In addition, $\ell$ and $c$ are required to be compatible with themselves and with each other, meaning that the diagrams
\[\begin{tikzcd}[column sep=2.25em]
	TM && {T^2M} & {T^3M} & {T^3M} & {T^3M} & {T^2M} & {T^3M} & {T^3M} \\
	{T^2M} && {T^3M} & {T^3M} & {T^3M} & {T^3M} & {T^2M} && {T^3M}
	\arrow["{\ell_{TM}}"', from=2-1, to=2-3]
	\arrow["\ell", from=1-1, to=1-3]
	\arrow["\ell"', from=1-1, to=2-1]
	\arrow["{T(\ell)}", from=1-3, to=2-3]
	\arrow["{T(c)}", from=1-4, to=1-5]
	\arrow["{c_{TM}}", from=1-5, to=1-6]
	\arrow["{c_{TM}}"', from=1-4, to=2-4]
	\arrow["{T(c)}"', from=2-4, to=2-5]
	\arrow["{c_{TM}}"', from=2-5, to=2-6]
	\arrow["{T(c)}", from=1-6, to=2-6]
	\arrow["{\ell_{TM}}", from=1-7, to=1-8]
	\arrow["{T(c)}", from=1-8, to=1-9]
	\arrow["c"', from=1-7, to=2-7]
	\arrow["{T(\ell)}"', from=2-7, to=2-9]
	\arrow["{c_{TM}}", from=1-9, to=2-9]
\end{tikzcd}\]
all commute.  Finally, we require that  
\[\begin{tikzcd}
	{T_2M} && {T^2M} \\
	M && TM
	\arrow["\nu", from=1-1, to=1-3]
	\arrow["{p_M \circ \pi_0}"', from=1-1, to=2-1]
	\arrow["\lrcorner"{anchor=center, pos=0.125}, draw=none, from=1-1, to=2-3]
	\arrow["{T(p_M)}", from=1-3, to=2-3]
	\arrow["{0_M}"', from=2-1, to=2-3]
\end{tikzcd}\]
is a pullback preserved by powers of $T$, where $\nu$ is a shorthand notation for $\nu = T(+) \circ \langle \ell_M \circ \pi_0 , 0_{TM} \circ \pi_1 \rangle $. This pullback is denoted as the \textbf{universality of the vertical lift}. The foundational pullbacks together with powers of $T$ applied to the universality of the vertical lift will be called the \textbf{tangent pullbacks}.
\end{definition}
While the last condition is formulated differently in Definition 2.3 of \cite{Cockett2014DifferentialST}, it is shown to be equivalent to the formulation here in Lemma 2.12 of \cite{Cockett2014DifferentialST}.

\begin{remark}
The last condition is called the universality of the vertical lift and it encodes the intuition that $T^2M$ relates to $T_2M$ in the same way that $TM$ relates to $M$. For example, in the category $\SmMan$ of manifolds and smooth maps (see Example \ref{ex:smooth_manifolds_vs}), each object has a dimenion.  %If one thinks of the objects to have a dimension like in \SmMan, this means 
The vertical lift ensures that $\dim(T^2M)-\dim(T_2M) = \dim (TM) - \dim(M)$. In particular if the dimension of the tangent bundle is  a multiple by $n$ of the  dimension of the manifold, this becomes $n^2-(2n-1) = n-1$, which only is the case when $n=1$ or $n=2$. 
Thus the universality of the vertical lift in \SmMan{} enforces the idea that the dimension of the tangent bundle is either the same or twice  the dimension of the original manifold. We will explore the relationship between the vertical lift and the dimension of the tangent bundle in general tangent categories in a forthcoming paper.
\end{remark}

The structures involved in Definition \ref{def:tangent_cat} of tangent categories can be summarized as a diagram:
\[\begin{tikzcd}
	& {T_2(M)} \\
	M & {T(M)} & {T^2(M)}
	\arrow["{+}", from=1-2, to=2-2]
	\arrow["0", shift left, from=2-1, to=2-2]
	\arrow["p", shift left, from=2-2, to=2-1]
	\arrow["\ell", shift left, from=2-2, to=2-3]
	\arrow["{T(p)}", shift left, from=2-3, to=2-2]
	\arrow["c", from=2-3, to=2-3, loop, in=55, out=125, distance=10mm]
\end{tikzcd}\]
for every object $M$.

In general there can be many tangent structures for a given category $\mathbb X$. In particular any category with a nontrivial tangent structure $T\neq 1_\mathbb X$ has at least two tangent structures, $T$ and the trivial tangent structure given by the identity functor.
\begin{example}\label{ex:trivial_tangent_structure}\cite[Section 2.2]{Cockett2014DifferentialST}
Any category $\mathbb X$ with the identity functor and identity transformations for all of the natural transformations required by Defintion \ref{def:tangent_cat} give an example of a tangent category, $(\mathbb X, 1_\mathbb X, 1_{1_{\mathbb X}} , 1_{1_{\mathbb X}}, 1_{1_{\mathbb X}}, 1_{1_{\mathbb X}},1_{1_{\mathbb X}})$.
\end{example}

The main motivating example for the definition are smooth manifolds. This tangent structure is explained in  \cite[Section 2.2]{Cockett2014DifferentialST}.

\begin{example}\label{ex:smooth_manifolds_vs}\cite[Section 2.2]{Cockett2014DifferentialST}
Let \SmMan{} be the category with smooth manifolds as objects and smooth maps between them as morphisms. Then \SmMan{} is a tangent category where the tangent functor is given by forming the usual tangent bundle from differential geometry, $T(M):=TM$.  Elements of the tangent bundle are denoted by $(p, v_p)$, where $p$ is a point on the manifold and $v_p$ is a vector in the tangent space $T_pM$.
The projection $p: TM \to M$, the zero section $0: M \to TM$ and the addition of tangent vectors $+:T_2M \to TM$ are well-known from differential geometry, for example defined in \cite[Section 6.2]{michor2008topics}.

Elements of $T^2(M) = T(T(M))$ the tangent bundle of the tangent bundle are denoted by a 4-tuple $(p,v_p,u_p,w_{p,v_p})$ where $p \in M$, $v_p \in T_pM$, $u_p \in T_pM$ and $w_{p,v_p} \in T_{v_p}T_{p}M$, which implies $(u_p, w_{p,v_p}) \in T_{(p,v_p)}TM$.

The vertical lift $\ell : TM \to TTM$ is less frequently used in differential geometry, it is defined for example at the end of \cite[Section 6.12]{michor2008topics}.
It sends an element $(p,v)$ of the tangent bundle to $(p,0_p,0_p,v_{p,0_p})$. Here we use the canonical identification of the tangent space $T_{v_p}T_pM$ of the vector space $T_pM$ with itself.

The canonical flip sends an element $(p, u_p, v_p , w)$ of $TTM$ to the element  $(p, v_p, u_p , w)$ of $TTM$ which is possible as both $u_p$ and $v_p$ are in $T_pM$ and $T_{v_p}T_pM$ is canonically identified with $T_{u_p}T_pM$.

Every manifold of dimension $n$ is locally diffeomorphic to a subset of $\mathbb{R}^n$, see \cite{michor2008topics}, and using this equivalent formulation we can express points in the tangent bundle in local coordinates.  In local coordinates, on an open neighborhood $U \subset M$, the tangent bundle $T(M)$ is isomorphic to a product of the base manifold with a vector space
$$
T(U) \cong U \times \mathbb R^n
$$
where $n$ is the dimension of $M$. Here, we use $U$ to denote both the open neighborhood of $M$ and its diffeomorphic image in $\mathbb{R}^n$. For mandifolds $M_1, M_2$ with dimensions $n_1, n_2$ and a smooth map $f: M_1 \to M_2$, there are open neighborhoods $U_1 \subset M_1$ and $U_2 \subset M_2$ with $f(U_1) \subset U_2$ such that
$T(U_1) \cong U_1 \times \mathbb R^{n_1}$ and $T(U_2) \cong U_2 \times \mathbb R^{n_2}$.  The evaluation $T(f): T(M_1) \to T(M_2)$ of the tangent bundle functor on the morphism $f$ can also be expressed in these local coordinates as
\begin{align*}
T(f)|_{T(U_1)}:T(U_1 ) &\to T(U_2) \\ 
T(U_1) \cong U_1 \times \mathbb R^{n_1} \ni (x,v) &\mapsto \left(f(x), \frac{\partial f(y)}{\partial y}|_{y=x} \cdot v\right) \in U_2 \times \mathbb R^{n_2} \cong T(U_2).
\end{align*}
This formulation of $T(f)$ in terms of local coordinates is standard knowledge in differential geometry and can, for example, be found in \cite[Section 1.11]{michor2008topics}.
\end{example}
Unlike classical manifolds, tangent structures do not need to be related to the real numbers.
\begin{example}[Example 2.7 in \cite{ikonicoff2023cartesian}]\label{ex:Z_p}
Let $\mathbb F$ be a field. Then we define a category ${\mathbb P \mathrm{oly}}_{\mathbb F}$ whose objects are the finite-dimensional $\mathbb F$-vector spaces, $\mathbb F^k$. Morphisms $g:{\mathbb F}^n \to {\mathbb F}^k$ are $k$-tuples of of polynomials in $n$ variables
, for example,
$$
g: \mathbb F^3 \to \mathbb F^2 \qquad (x,y,z) \mapsto (x^2z-1, y^3 z+2x) .
$$
For a morphism $f$ let $J_f: \vec x \mapsto (\sum_j {\partial f^i \over \partial x_j} x_j)_{1\leq i \leq k}$ be the linear map given by the Jacobi matrix of $f$ (where ${\partial f^i \over \partial x_j} $ denotes the derivative of a polynomial defined by ${\partial x_i^k \over \partial x_i } = k \cdot x_i^{k-1}$).
The category ${\mathbb P \mathrm{oly}}_{\mathbb F}$  has a tangent structure, given by
\begin{align*}
T(\mathbb F^k) &:= \mathbb F^k \times \mathbb F^k\\
T(f) &:= f \times J_f\\
p_{\mathbb F^k} &:= \pi_0: \mathbb F^k \times \mathbb F^k \to \mathbb F^k\\
0_{\mathbb F^k} &:= \langle 1_{\mathbb F^k} , 0 \rangle : \mathbb F^k \to \mathbb F^k \times \mathbb F^k \\
+_{\mathbb F^k} &:= \langle \pi_0 ,\add \circ \langle \pi_1 , \pi_2 \rangle \rangle : \mathbb F^k \times \mathbb F^k \times \mathbb F^k \to \mathbb F^k \times \mathbb F^k
\\
\ell_{\mathbb F^k} &:= \langle \pi_0 , 0 , 0 , \pi_1 \rangle : \mathbb F^k \times \mathbb F^k \to \mathbb F^k \times \mathbb F^k \times \mathbb F^k \times \mathbb F^k\\
c_{\mathbb F^k} &:= \langle \pi_0 , \pi_2 , \pi_1 , \pi_3 \rangle : 
\mathbb F^k \times \mathbb F^k \times \mathbb F^k \times \mathbb F^k
\to \mathbb F^k \times \mathbb F^k \times \mathbb F^k \times \mathbb F^k,
\end{align*}
where $\add: \mathbb F^k \times \mathbb F^k \to \mathbb F^k$ is the addition of vectors.

\end{example}
This example is very useful because finite fields are a well-studied area of mathematics.  In Remark \ref{remark:linear_neq_additive} we use this to show that not every additive map is linear. 

A closely related example is the category $\mathbb S\mathrm{mooth}$ in which the objects are all object $\mathbb R^k$ and the morphisms are all smooth maps.  The same structure maps defined here (with $\mathbb F= \mathbb R$) give rise to a tangent structure on $\mathbb S\mathrm{mooth}$.

\begin{definition}\label{def:N_bullet}\cite[proof of Proposition 5.8]{BauerBurkeChing}
Let $\mathbb N^\bullet$ be the category with objects $\mathbb N^k$ for $k \in \mathbb N$ and morphisms from $\mathbb{N}^m$ to $\mathbb{N}^n$ given by matrices $f = (f_{ij})_{1\leq i\leq n ,1\leq j \leq m}$ where $f_{ij} \in\mathbb N$ for all $i$, $j$. The composition is given by matrix multiplication and the identity matrix is the identity.
\end{definition}

\begin{example}\label{ex:tan_structure_on_N_bullet}
The category $\mathbb N^\bullet$ has products and the product structure is given by $\mathbb N^k \times \mathbb N^l = \mathbb N^{k+l}$. There is an addition map $\mathrm{add}: \mathbb N^k \times \mathbb N^k \to \mathbb N^k$ given by $(\bbI_{k\times k}\vert \bbI_{k\times k})$, where $\bbI_{k\times k}$ is the $k \times k$ identity matrix. There is a tangent structure on $\mathbb N^\bullet$ with a tangent functor $D:\mathbb{N}^\bullet \to \mathbb{N}^\bullet$ and structures defined by:
\begin{align*}
D(A) &= A \times A \\
D(f) &= f \times f\\
p_{A} &:= \pi_0: A \times A \to A\\
0_{A} &:= \langle 1_{A} , 0 \rangle : A \to A \times A \\
+_{A} &:= \langle \pi_0 ,\add \circ \langle \pi_1 , \pi_2 \rangle \rangle :  A \times A \times A \to A \times A \\
\ell_{A} &:= \langle \pi_0 , 0 , 0 , \pi_1 \rangle : A \times A \to A \times A \times A \times A\\
c_{A} &:= \langle \pi_0 , \pi_2 , 
\pi_1 , \pi_3 \rangle : 
A \times A \times A \times A
\to A \times A \times A \times A
\end{align*}

Note that in these definitions, $D^2(A)=A\times A\times A\times A$ and $D_2(A)=A\times A \times A$.
It was proven in Proposition 5.18 of \cite{BauerBurkeChing} that $(\mathbb N^\bullet,D)$ is a tangent category.    
\end{example}

We now turn to establishing a way to compare tangent categories by defining tangent functors.  The main  issue in defining tangent functors is how strongly a functor between tangent categories may preserve the tangent bundle functor.  Whether or not the tangent bundle functor is preserved up to natural isomorphism or only preserved up to a natural transformation gives us two different possible defintions for tangent functors, as in the next definition.

\begin{definition}\label{def:lax_tangent_functor}\cite[Definition 2.7]{Cockett2014DifferentialST}
Suppose that $(\mathbb X , T , p , 0, + ,\ell ,c)$ and $(\mathbb X' , T' , p' , 0', +' ,\ell' ,c')$ are tangent categories.
\begin{enumerate}
\item A \textbf{lax tangent functor} (called a morphism of tangent categories in \cite[Definition 2.7]{Cockett2014DifferentialST}) $(F, \alpha): (\mathbb X,T) \to (\mathbb X',T')$ is a functor $F: \mathbb X \to \mathbb X'$ and a natural transformation $\alpha: F \circ T \Rightarrow T' \circ F$, called the \textbf{\thatalpha}, such that the following diagrams commute:

\[\begin{tikzcd}[column sep=large]
	{F \circ T} & {T' \circ F} & F && {F \circ T_2} & {T'_2\circ F} \\
	& F & {F \circ T} & {T' \circ F} & {F \circ T} & {T' \circ F} \\
	& {F \circ T} & {T' \circ F} & {F \circ T^2} & {T'^2 \circ F} \\
	& {F \circ T^2} & {T'^2 \circ F} & {F \circ T^2} & {T'^2 \circ F}
	\arrow["\alpha", from=1-1, to=1-2]
	\arrow["{p'_F}", from=1-2, to=2-2]
	\arrow["{F(p)}"', from=1-1, to=2-2]
	\arrow["{F(0)}"', from=1-3, to=2-3]
	\arrow["{0'_F}", from=1-3, to=2-4]
	\arrow["\alpha"', from=2-3, to=2-4]
	\arrow["{\alpha_2}", from=1-5, to=1-6]
	\arrow["\alpha"', from=2-5, to=2-6]
	\arrow["{F(+)}"', from=1-5, to=2-5]
	\arrow["{+'_F}", from=1-6, to=2-6]
	\arrow["\alpha", from=3-2, to=3-3]
	\arrow["{F(\ell)}"', from=3-2, to=4-2]
	\arrow["{\ell_F'}", from=3-3, to=4-3]
	\arrow["{\alpha_T \circ T(\alpha)}"', from=4-2, to=4-3]
	\arrow["{F(c)}"', from=3-4, to=4-4]
	\arrow["{c'_F}"', from=3-5, to=4-5]
	\arrow["{\alpha_T \circ T(\alpha)}"', from=4-4, to=4-5]
	\arrow["{\alpha_T \circ T(\alpha)}", from=3-4, to=3-5]
\end{tikzcd}\]
\item A \textbf{strong tangent functor} is a lax tangent functor such that the {\thatalpha} $\alpha$ is a natural isomorphism and $F$ preserves the tangent pullbacks of Definition \ref{def:tangent_cat}, i.e. the foundational pullbacks and the universality of the vertical lift.
\end{enumerate}
\end{definition}
One example of a tangent functor that will be an example of our classification of differential bundles in Theorem \ref{thm:equivalence_bundle_categories} is given by the following functor.
\begin{example}\label{ex:functor N bullet to smooth} There is a tangent functor
$F:{\mathbb N}^\bullet \to {\mathbb S \mathrm{mooth}}$ given by sending an object $\mathbb{N}^k$ to $\mathbb{R}^k$ and each $\mathbb{N}$-linear map $f$ represented by a matrix $(f_{ij})$ to the matrix $(f_{ij})$, considered as an $\mathbb{R}$-linear map using he embedding of $\mathbb N$ into $\mathbb R$.  This linear map is smooth, and it is apparent that this assignment preserves composition and the identity.  The natural isomorphism $\alpha: F \circ T_{\mathbb N^\bullet} \to T_{\mathbb S \mathrm{mooth}} \circ F$ is given by $$\alpha_{\mathbb N^k} := 1_{\mathbb R^{2k}}: F(T_{\mathbb N^\bullet}(\mathbb N^k))= \mathbb R^{2k} \to \mathbb R^{2k} = T_{\mathbb S \mathrm{mooth}}(F(\mathbb N^k)).$$
    
\end{example}

\begin{definition}\label{def:tangent_trafo}
A \textbf{tangent transformation} between (lax/strong) tangent functors $(F,\alpha)\Rightarrow (F', \alpha')$ is a natural transformation of underlying functors $\varphi: F \Rightarrow F'$. It is called a \textbf{linear tangent transformation} if 
\begin{equation} \label{diagram:linearity_tangent_trafor}
\begin{tikzcd}
	{F \circ T} && {F' \circ T} \\
	{T' \circ F} && {T' \circ F'}
	\arrow["{\varphi_T}", from=1-1, to=1-3]
	\arrow["{T'(\varphi)}", from=2-1, to=2-3]
	\arrow["\alpha'", from=1-3, to=2-3]
	\arrow["\alpha", from=1-1, to=2-1]
\end{tikzcd}\end{equation} 
commutes.
\end{definition}

\subsection{Differential bundles}
Many constructions relevant in differential geometry and mathematical physics, such as
metric tensors, curvature, symplectic structures, spinors and differential forms are given by sections of certain vector bundles.  The study of vector bundles is also the central theme of K-theory.
Thus a central object of study in tangent categories are differential bundles, the categorical generalization of the notion of vector bundles from differential geometry. 
We start by recalling the definition of a differential bundles from \cite{cockett2016diffbundles}:
\begin{definition}\label{def:differential bundle}\cite[Definition 2.3]{cockett2016diffbundles}
Given a tangent category $(\mathbb X,T)$, a \textbf{differential bundle} \\ $(E,M,q,\zeta, \sigma, \lambda)$ consists of 
\begin{enumerate}
    \item an object $E$ (the total space),
    \item an object $M$ (the base space),
    \item a morphism $q:E \to M$ (the projection), admitting finite pullback powers $\{E_n\}_{n \in \mathbb N}$ over itself that are preserved by powers of the tangent functor $T^k$,
    \item a morphism $\zeta: M \to E$ (the zero section) and a morphism $\sigma: E_2 \to E$ (fiberwise addition) such that $(q,\zeta,\sigma)$ is an additive bundle, and
    \item a morphism $\lambda: E \to T(E)$ (the lift) such that $$(\lambda, 0): (E,M,q, \sigma,\zeta) \to (T(E),T(M),T(q)), T(\sigma), T(\zeta))$$ and $$(\lambda, \zeta): (E,M,q, \sigma,\zeta) \to (T(E),E,p, + , 0)$$ are additive bundle morphisms.
\end{enumerate}
In addition the diagram
\begin{equation}
\begin{tikzcd}
	{E_2} &&& TE \\
	M &&& TM
	\arrow["{q\circ \pi_0}"', from=1-1, to=2-1]
	\arrow["0", from=2-1, to=2-4]
	\arrow["{T(q)}", from=1-4, to=2-4]
	\arrow["{T(\sigma) \circ \langle \lambda \circ \pi_0 , 0 \circ \pi_1\rangle }", from=1-1, to=1-4]
	\arrow["\lrcorner"{anchor=center, pos=0.125}, draw=none, from=1-1, to=2-4]
\end{tikzcd}\label{diagram:universality}
\end{equation}
is a pullback and the morphisms $\ell_E \circ \lambda , T(\lambda) \circ \lambda : E \to T(T(E))$ are equal.
In the special case when $M$ is the terminal object, the differential bundle  is called a \textbf{differential object}. 

\end{definition}
The maps $\zeta$, $\sigma$ and $\lambda$ will be suppressed from the notation when they are not needed. Then the differential bundle $(E,M,q,\zeta, \sigma, \lambda)$ will be denoted by $E \xrightarrow{q} M$ or $(E,M,q)$. Due to the symmetry of the addition, it is equivalent to ask for Diagram \ref{diagram:universality} or
\[\begin{tikzcd}
	{E_2} &&& TE \\
	M &&& TM
	\arrow["{q\circ \pi_0}"', from=1-1, to=2-1]
	\arrow["0", from=2-1, to=2-4]
	\arrow["{T(q)}", from=1-4, to=2-4]
	\arrow["{T(\sigma) \circ \langle 0 \circ \pi_0 , \lambda \circ \pi_1\rangle }", from=1-1, to=1-4]
	\arrow["\lrcorner"{anchor=center, pos=0.125}, draw=none, from=1-1, to=2-4]
\end{tikzcd}
\]
to be a pullback. Since they are equivalent, we use both conditions, depending on what is more convenient in a given situation.

\begin{remark}\label{rem:DiffObj}
    For a differential object $(V,*,q,\zeta,\sigma,\lambda)$, the morphism $q:V \to *$ to the terminal object must be the unique morphism $!:V \to *$ to the terminal object.\footnote{It is also true that the universality of the vertical lift shows that $\sigma \circ (\lambda \times 0) : V^2 \to T(V)$ is an isomorphism.  We will not need this, but it is an important property of differential objects.} Thus we denote a differential object by $(V,\sigma, \zeta, \lambda)$. If the morphisms $\sigma, \zeta, \lambda$ are not needed they are suppressed and the differential object is just denoted by $V$.

\end{remark}

\begin{example}\label{ex:trivial_bundle}
A Cartesian tangent category is a tangent category in which finite products exist and the product projections are compatible with the tangent structure, see \cite[Section 2.4]{Cockett2014DifferentialST}.
    Given a differential object $(V,\sigma,\zeta,\lambda)$ and an object $M$ in a Cartesian tangent category, there is a differential bundle with:
    \begin{itemize}
        \item total space $M\times V$,
        \item base space $M$,
        \item projection $\pi_0 : M \times V \to M$
        \item addition $1 \times \sigma: M \times V \times V \to M \times V$
        \item zero section $1 \times \zeta : M \cong M \times * \to M \times V$
        \item vertical lift $0_M \times \lambda : M \times V \to T(M) \times T(V) \cong T(M \times V)$
    \end{itemize}
    We call bundles of this form trivial differential bundles.
\end{example}

Classically defined vector bundles from differential geometry were the motivating example of differential bundles.  Theorem 2.5.1 of \cite{McAdam} shows that differential bundles in \SmMan{} are exactly vector bundles. In particular, differential objects in \SmMan{} are simply vector spaces. The subcategory Smooth in \SmMan{} is the full subcategory consisting of the differential objects, i.e. the vector spaces. 

\begin{example}\label{ex:N_bullet_differential}~
\begin{enumerate}[label = (\alph*)]
    \item Every object in $\mathbb N^\bullet$ is a differential object, i.e. every $\mathbb N^k$ is a differential bundle over $\mathbb N^0$ with the structure given by the following maps:
\begin{align*}
\zeta &:= 0 : \mathbb N^0 \to \mathbb N^k \textrm{, the inclusion of zero in }\mathbb N^k;\\
\sigma &:= \text{add} : \mathbb N^{2k} \to \mathbb N^k \textrm{, the usual addition map;}\\
\lambda &:= \langle 0,1_{\mathbb N^k} \rangle : \mathbb N^k \to \mathbb N^{2k}.
\end{align*}
It is an easy exercise to check that these maps fulfill all the conditions required for a differential bundle.
\item Replacing $\mathbb N$ with $\mathbb R$ in the previous example shows that every object of $\mathbb S\mathrm{mooth}$ is a differential object. This is because the tangent functor $(F, \alpha): \mathbb N^\bullet \to \mathbb S \mathrm{mooth}$ from Example \ref{ex:functor N bullet to smooth} preserves the differential object structure. 
Given a differential object in $\mathbb N^\bullet$ with structure maps $\zeta,\sigma$ and $\lambda$, the maps $F(\zeta), F(\sigma)$ and $F(\lambda)$ are the structure maps of a differential object because $F$ leaves these maps essentially unchanged.  Every object of $\mathbb S\mathrm{mooth}$ is of this type because $F$ is essentially surjective.  This, plus the fact that $\alpha$ is the identity transformation in this case, make it trivial to verify that the resulting structures fulfill Definition \ref{def:differential bundle}.
\item Replacing $\mathbb N$ with any field $\mathbb F$ will show that every objection of $\mathbb P \mathrm{oly}_\mathbb F$ is a differential object.  This works essentially the same as the argument in (b).
\end{enumerate}

\end{example}
The differential object structure on $\mathbb N^1$ is a key property of the category $\mathbb N^\bullet$ that will allow us to classify differential bundles in Theorem \ref{thm:equivalence_bundle_categories}.

\begin{definition}\label{def:differential_bundle_morphisms}\cite[Definition 2.3]{cockett2016diffbundles}
A \textbf{morphism of differential bundles} $$(f, g) : (E,M,q,\sigma, \zeta, \lambda) \to (E',M',q',\sigma', \zeta', \lambda')$$
is a pair of maps $f: E \to E', g: M \to M'$ such that the diagram
\[\begin{tikzcd}
	E & {E'} \\
	M & {M'}
	\arrow["g"', from=2-1, to=2-2]
	\arrow["{q'}", from=1-2, to=2-2]
	\arrow["q"', from=1-1, to=2-1]
	\arrow["f", from=1-1, to=1-2]
\end{tikzcd}\]
commutes.
A morphism of differential bundles is called \textbf{linear} if 
\begin{equation}\label{eqn:linear}
\begin{tikzcd}
	E & {E'} \\
	TE & {TE'}
	\arrow["f", from=1-1, to=1-2]
	\arrow["{T(f)}"', from=2-1, to=2-2]
	\arrow["{\lambda'}", from=1-2, to=2-2]
	\arrow["\lambda"', from=1-1, to=2-1]
\end{tikzcd} 
\end{equation}
commutes and it is called \textbf{additive} if the diagrams
\begin{equation}\label{eqn:additive}
\begin{tikzcd}
	{E_2} & {E'_2} & M & {M'} \\
	E & {E'} & E & {E'}
	\arrow["g", from=1-3, to=1-4]
	\arrow["f"', from=2-3, to=2-4]
	\arrow["\zeta"', from=1-3, to=2-3]
	\arrow["{\zeta'}", from=1-4, to=2-4]
	\arrow["{f_2}", from=1-1, to=1-2]
	\arrow["{\sigma'}", from=1-2, to=2-2]
	\arrow["f"', from=2-1, to=2-2]
	\arrow["\sigma"', from=1-1, to=2-1]
\end{tikzcd}  
\end{equation}
commute.
\end{definition}

In the special case of differential objects, a morphism of differential bundles between differential objects is just a morphism in the underlying tangent category. 

It was shown in \cite{cockett2016diffbundles} that every linear morphism of differential bundles is additive.

Linearity means that $(f, T(f))$ preserves the lift, and additivity means that $(f,g)$ preserves the addition and zero morphisms. In order to understand linearity from an intuitive perspective, we consider the example of trivial vector bundles $M \times V \xrightarrow{\pi_0} M$ in \SmMan.

\begin{example}[Linear morphisms in \SmMan]
Let $(f,g): (M \times V , M , \pi_0) \to (M' \times V', M' , \pi_0)$ be a morphism between two trivial differential bundles as in Example \ref{ex:trivial_bundle} in \SmMan.  Since this is a morphism of differential bundles, $g \circ \pi_0 = \pi_0 \circ f$. This means the map $f:M \times V \to M' \times V'$ has to be of the form $f = \langle g \circ \pi_0 , f_1 \rangle$ for some map $f_1: M \to M' \times V'$.
We would like to know when $(f,g)$ is a linear morphism.  The linearity condition from Definition \ref{def:differential_bundle_morphisms}, Diagram \ref{eqn:linear}, says that $\lambda \circ f = T(f) \circ \lambda: M\times V \to T(M\times V) $.  In the local coordinates from Example \ref{ex:smooth_manifolds_vs} the tangent bundle is a product, $T(U) \cong U \times \mathbb R^n$ for some open $U \subset M$. For $x \in U \subset M$ and $v \in T_xM \cong \mathbb R^n$ and a sufficiently small neighbourhood $U'$ around $f(x)$, the 
composition 
$$
\lambda \circ f: U \times V \to T(U' \times V')\cong U' \times T_{f(x)}U' \times V' \times V'$$
sends $(x,v)$ to 
$$
\lambda \circ f (x,v) = \lambda (g(x),f_1(x,v)) = (g(x), 0 , 0 ,f_1(x,v)) . 
$$ 
If  $(f,g)$ is linear, this must be equal to 
\begin{align*}\label{eq:lambdaT(f)local}
T(f) \circ \lambda (x,v) &= T(\langle g \circ \pi_0 , f_1 \rangle )(x,0,0,v) 
\\&= \left( g(x),f_1(x,0),\frac{\partial g(y)}{\partial y}|_{y=x}\cdot 0 , \frac{\partial f_1(x,w)}{\partial w}|_{w=0}\cdot v \right) 
\\& = \left(g(x), f_1(x,0), 0 , \frac{\partial f_1(x,w)}{\partial w}|_{w=0}\cdot v \right)
\end{align*}
which we expanded using the local coordinate expression from Example \ref{ex:smooth_manifolds_vs}. We conclude that 
$$\left(g(x), f_1(x,0), 0 , \frac{\partial f_1(x,w)}{\partial w}|_{w=0}\cdot v \right)=\left(g(x), 0, 0, f_1(x,v)\right)$$.
Therefore $(f,g)$ is linear precisely when $f_1(x,v)= \frac{\partial f(x,w)}{\partial w}|_{w=0}\cdot v$ for all $v\in \mathbb{R}^n$, including $v=0$.
\end{example}

\begin{remark}\label{remark:linear_neq_additive}
It is possible for a bundle morphisms to be additive but not linear. Consider the category ${\mathbb P \mathrm{oly}}_{\mathbb Z /p}$ of polynomials over the field with $p$ elements, as introduced in Example \ref{ex:Z_p}.  Let  $\mathbb Z / p$ be the differential object with the zero map as $\zeta$, the addition of field elements as $\sigma$ and $\langle 0 , 1_{\mathbb Z / p} \rangle$ as $\lambda$ as in Example \ref{ex:N_bullet_differential}(b). Define a morphism  $f: \mathbb Z / p \to \mathbb Z / p$  by $x \mapsto x^p$. This is a morphism of differential bundles because $\mathbb Z/p$ is a differential object and thus every morphism between them is a differential bundle morphism.  Since we are working in $\mathbb Z / p$, the identities $x^p + y^p = (x+y)^p$ and $0^p = 0$ hold, making Diagram (2) from Definition \ref{def:differential_bundle_morphisms} commute.  Thus $f$ is an additive morphism.

 From Example \ref{ex:Z_p} and Example \ref{ex:smooth_manifolds_vs}, $T(f)$ is given by the Jacobi matrix.  However the Jacobi matrix (i.e. derivative) of $f$ is $f'(x) = p x^{p-1}=0$.  Therefore, $T(f)$ is zero. Then $T(f) \circ \lambda  (1) = T(f)(0,1) = (f(0), f'(0)) = (0,0)$ while $\lambda \circ f (1) = \lambda (1^p) = (0,1)$.
Consequently Diagram (1) in Definition \ref{def:differential_bundle_morphisms} fails to commute.  

\end{remark}

The composition of additive/linear maps is additive/linear, thus we can define three different categories of differential bundles: We can include any differential bundle morphisms, additive differential bundle morphisms or linear differential bundle morphisms.

\begin{definition}\label{def:names_of_categories_involved}
Let $\mathbb X$ be a tangent category.
\begin{enumerate}[label=(\alph*)]
    \item Let $\DBun(\mathbb X)$ be the category of differential bundles and morphisms of differential bundles in $\mathbb X$.
    \item  The categories of differential bundles with additive/linear morphisms are denoted by $ \DBun_\add(\mathbb X) $ and $\DBun_\lin(\mathbb X)$ respectively.
    \item The full subcategories of $\DBun(\mathbb X)$, $\DBun_\add(\mathbb X)$ and $\DBun_\lin(\mathbb X)$ consisting of the differential objects (differential bundles over the terminal object) are denoted by $\DObj(\mathbb X)$, $\DObj_\add(\mathbb X)$ and $\DObj_\lin(\mathbb X)$, respectively.
\end{enumerate}
 If there is no terminal object we use the convention that both $\DObj(\mathbb X)$ and $\DObj_\lin(\mathbb X)$ are the empty category.
\end{definition}

\section{Characterizing Tangent categories as \weil-actegories}
As presented in Definition \ref{def:tangent_cat}, tangent categories are an axiomatization of manifolds, tangent bundles, and smooth maps. However, there is an alternative approach to tangent categories using actegories which was first developed by Garner and Leung in \cite{Garner2018:embedding_theorem} and \cite{Leung2017}.  Garner and Leung noticed that the morphisms used in the axioms from Definition \ref{def:tangent_cat} correspond to morphisms in a certain category \weil{} which are essential to the defining characteristics of \weil{}.  Tangent categories can be equivalently defined using the action of \weil{} on the desired tangent category. 

In Section \ref{sec:actegories}, we will begin by recalling what it means for a monoidal category $\mathbb M$ to act on another category, making it an $\mathbb M$-actegory.
Then in Section \ref{sec:tangent_structures_as_weil}, we will define the category \weil{} and will present Leung-Garner's alternative definition of a tangent category, in which tangent categories are characterized as \weil-actegories.  In this new characterization, one can also interpret the tangent functors and tangent transformations of Definitions \ref{def:lax_tangent_functor} and \ref{def:tangent_trafo} in terms of \weil{}-linear functors and \weil{}-linear natural transformations, which are the functors and natural transformations which preserve the  action of \weil{}.  Garner shows that there is a 2-equivalence of the 2-category of tangent categories, strong tangent functors and linear tangent transformations and the 2-category of \weil{}-acteogires, (strong) \weil{}-linear fuctors and \weil{}-linear natural transformations.  In the final result of this section, we adapt Garner's argument to obtain similar equivalences for the 2-categories with lax tangent functors and tangent transformations which are not necessarily linear.

\subsection{Actegories}\label{sec:actegories}

We introduce actegories as a tool to encode tangent categories, tangent functors and tangent transformations. Here in Section \ref{sec:actegories} we define actegories and establish some basic properties. In Section \ref{sec:tangent_structures_as_weil} we will use actegories to characterize tangent categories.

Here we use the definition of monoidal categories as in  \cite[Appendix E2]{riehl2017category}. A monoidal category $(\mathbb A, \otimes , I)$ has a monoidal product $\otimes$ and unit $I$.   A monoidal category also comes equipped with an associator, $a$, and left and right unitors, $l$ and $r$, respectively.  These are natural transformations, and their components at objects $x,y,z$ in $\mathbb A$ will be denoted $a_{x,y,z}: (x \otimes y) \otimes z \cong x \otimes (y \otimes z)$, $l_x : I \otimes x \cong x$ and $r_x: x \otimes I \cong x$.  When want to emphasize these natural transformations, we denote the monoidal category structure by $(\mathbb A, \otimes , I, a, l, r)$.  The notation $\mathbb A^\otimes$ is shorthand for $(\mathbb A, \otimes , I, a, l, r)$. 

The following is an example of a monoidal category which is important for characterizing tangent categories.  

\begin{example}
 Given a category $\mathbb X$, the category $\End(\mathbb X)$ of endofunctors of $\mathbb X$ has a monoidal structure given by the composition of endofunctors as monoidal product and the identity functor as unit, denoted as $\End( \mathbb X)^\circ$.
\end{example}

Monoidal functors from any monoidal category $\mathbb M^\otimes$ into $\End(\mathbb X)^\circ$ describe $\mathbb M^\otimes$ acting on $\mathbb X$. These actions are made precise below in the notion of an actegory.
Actegories were first defined more than 50 years ago in \cite{benabou_intro_to_bicategories}.  We present an equivalent definition from \cite{capucci2023actegories}, a more modern reference that provides more details.
While it is more common to denote the action by $\cdot$ or $\otimes$, we use $T$ for it to resemble the tangent bundle endofunctor $T$.

\begin{definition}\label{def:actegories}\cite[Definition 3.1.1]{capucci2023actegories}
Let $\mathbb M^\otimes = (\mathbb M, \otimes , I, a, l, r)$ be a monoidal category with symmetric monoidal product $\otimes$ and unit $I$. A (left) \textbf{$\mathbb M^\otimes$-actegory} $\mathbb X$ (or left
$\mathbb M$-action) is a category $\mathbb X$ equipped with a functor
$$
\Te : \mathbb M \times \mathbb X \to \mathbb X
$$
and two natural isomorphisms with components at $X \in \mathrm{Obj}(\mathbb X)$ and $m,n \in \mathrm{Obj}(\mathbb M)$
$$
\eta_X : X \cong \Te(I, X) ,  \qquad  \mu_X: \Te( X ,\Te( m , n)) \cong \Te(X \otimes m, n),
$$

respectively called unitor and multiplicator, satisfying coherence laws expressed as the following commutative diagrams:
\[\begin{tikzcd}[column sep=tiny]
	  {\Te(m, \Te(n, \Te (p, X)))} && {\Te(m \otimes n, \Te (p, X))} \\
	{\Te(m, \Te(n \otimes p, X))} && {\Te((m \otimes n) \otimes p , X)} \\
	& {\Te (m \otimes (n \otimes p) , X)}
	\arrow["{\mu_{m,n,\Te(p,X)}}", from=1-1, to=1-3]
	\arrow["{\Te(a_{m,n,p},1_X)}", from=2-3, to=3-2]
	\arrow["{\mu_{m \otimes n, p , X}}", from=1-3, to=2-3]
	\arrow["{\Te(1_m,\mu_{n,p,X})}"', from=1-1, to=2-1]
	\arrow["{\mu_{m,n\otimes p,X}}"', from=2-1, to=3-2]
\end{tikzcd}\]
\[\begin{tikzcd}[column sep=tiny]
	{\Te(I,\Te(m,X))} && {\Te(I \otimes m,X))} && {\Te(m, \Te(I,X))} && {\Te(m \otimes I,X)} \\
	& {\Te(m,X)} &&&& {\Te(m,X)}
	\arrow["{\eta_{\Te(m,X)}}", from=2-2, to=1-1]
	\arrow["{\Te(l_m,1_X)}"', from=2-2, to=1-3]
	\arrow["{\mu_{I,m,X}}", from=1-1, to=1-3]
	\arrow["{\mu_{m,I,X}}", from=1-5, to=1-7]
	\arrow["{\Te(r_m,1_X)}"', from=2-6, to=1-7]
	\arrow["{\Te(1_m, \eta_X)}", from=2-6, to=1-5]
\end{tikzcd}\]
where $a_{m,n,p}$, $l_m$ and $r_m$ are the natural isomorphisms encoding associativity,  left-unitality, and right-unitality of $\mathbb M$. 
\end{definition}

Any monoidal category $\mathbb M^\otimes$ is an actegory over itself with the functor $\Te : \mathbb M \times \mathbb M \to \mathbb M$ given by the monoidal product, $T(M,N) = M \otimes N$. In particular this makes $\weil$ into a $\weil^\otimes$-actegory. 

\begin{remark}\label{rem:actegory_notation_index}
    For a given object $m \in \mathbb M_0$, we sometimes denote $T_\bullet(m,-): \mathbb X \to \mathbb X$ as $T_m:\mathbb X \to \mathbb X$ and thus $T_\bullet(m,X)$ as $T_m(X)$.
\end{remark}

Functors between $M$-actegories should preserve the actegory structure.  There are two main ways of defining these functors, depending on whether the action must be preserve up to natural isomorphism or up to natural transformation.  In the latter case, there is also a choice of which direction the natural transformation should go.  These choices are all encoded in the following definition.

\begin{definition}\label{def:linear_functors}\cite[Definition 3.3.2]{capucci2023actegories}
Let $(\mathbb X, \Te, \eta , \mu)$ and $(\mathbb X', \Te', \eta' , \mu')$ be left $\mathbb M$-actegories. A \textbf{lax $\mathbb M$-linear functor} $(F, \alpha)$ between them is a functor $F : \mathbb X \to \mathbb X'$ equipped with a natural transformation
$$
\alpha:\Te'\circ (1_\mathbb M \times F)\Rightarrow F\circ \Te,
$$
i.e. a natural collection of maps $\alpha_{m,x}  : \Te' (m , F (x)) \to F (\Te(m, x))$
for $x \in \mathrm{Obj}(\mathbb X), m \in \mathrm{Obj}(\mathbb M)$,
called \textbf{lineator} satisfying the following coherence laws for all $m, n \in \mathrm{Obj}(M)$ and $x \in \mathrm{Obj}(\mathbb X)$:
\[\begin{tikzcd}[column sep=small]
	& {\Te'(m,\Te'(n,F(x)))} && {\Te'(m,F(\Te(n,x)))} \\
	&{\Te'(m \otimes n,F(x))} && {F(\Te(m, \Te(n,x)))} \\
	&& {F(\Te (m \otimes n, x))}
	\arrow["{\Te' (m,\alpha_{n,x})}", from=1-2, to=1-4]
	\arrow["{\mu'_{m,n,F(x)}}"', from=1-2, to=2-2]
	\arrow["{\alpha_{m,\Te(n,x)}}", from=1-4, to=2-4]
	\arrow["{F(\mu_{m,n,x})}", from=2-4, to=3-3]
	\arrow["{\alpha_{m \otimes n , x}}"', from=2-2, to=3-3]
\end{tikzcd}\]
\[\begin{tikzcd}
	{\Te'(I_\mathbb M, F(X) )} && {F(\Te(I_\mathbb M , X))} \\
	& {F(X)}
	\arrow["{\alpha_{I_\mathbb M, X}}", from=1-1, to=1-3]
	\arrow["{\eta'_{F(X)}}"', from=1-1, to=2-2]
	\arrow["{F(\eta_X)}", from=1-3, to=2-2]
\end{tikzcd}\]
We call $(F, \alpha)$ a \textbf{strong $\mathbb M$-linear functor}, if $\alpha$ is a natural isomorphism. An \textbf{oplax $\mathbb M$-linear functor} $(F,\alpha)$ is a lax $\mathbb M$-linear functor $(F, \alpha): \mathbb X^\mathrm{op} \to \mathbb {X'}^\mathrm{op}$, i.e. a functor $F: \mathbb X \to \mathbb X'$ with a lineator $\alpha: F \circ \Te \Rightarrow \Te' \circ (1_\mathbb M \times F)$.

\end{definition}
Having defined actegories and functors between actegories, we will continue defining natural transformations between actegories, that will have a special linearity condition. Unlike \cite{capucci2023actegories} we will not always require this linearity condition and call the natural transformation linear if it fulfills the condition.
\begin{definition}\label{def:actegory_linear_transformations}\cite[Definition 3.3.9]{capucci2023actegories}
Let $(F,\alpha)$ and $(F', \alpha')$ be $\mathbb M$-linear functors between the $\mathbb M$-actegories $(\mathbb X, \Te, \eta , \mu)$ and $(\mathbb X', \Te', \eta' , \mu')$.
\begin{enumerate}[label=(\alph*)]
\item A \textbf{$\mathbb M$-natural transformation} between $(F,\alpha)$ and $(F', \alpha')$ is simply a natural transformation $\varphi: F \Rightarrow F'$ of underlying functors.
\item It is called a \textbf{linear $\mathbb M$-natural transformation} if in addition the diagram
\begin{equation}
\label{diagram:linear_actegory_trafo}
\begin{tikzcd}[column sep=large,row sep=2.25em]
	{\Te(m,F(X))} & {\Te(m,F'(X))} \\
	{F(\Te(m,X))} & {F'(\Te(m,X))}
	\arrow["{\alpha_{m,X}}"', from=1-1, to=2-1]
	\arrow["{\varphi(\Te(m,X))}"', from=2-1, to=2-2]
	\arrow["{\alpha'_{m,X}}", from=1-2, to=2-2]
	\arrow["{\Te(1_m, \varphi)}", from=1-1, to=1-2]
\end{tikzcd}
\end{equation}
commutes.
\end{enumerate}
\end{definition}

\begin{example}
Let \FinSet{} be the category of finite sets and \Top{} be the category of topological spaces. With the Cartesian product as monoidal product and the one point space $*=\{0\}$ as the monoidal unit, \FinSet{} is made into a monoidal category $\FinSet^\times$. We will now define two actions making $\Top$ into a \FinSet-actegory:

\begin{enumerate}
    \item The functor
    $$
    T_0: \FinSet \times \Top \to \Top, (A,X) \mapsto \Disc(A) \times X
    $$
    where $\Disc$ sends the set $A$ to the topological space that is $A$ with the discrete topology.
    The pair $(\Top, \T_0)$ is a  \Set-actegory because the Cartesian product of sets with the discrete topology is the product of discrete topological spaces, the Cartesian product is associative and  the one point set with the discrete topology is the one point space.
    \item The projection functor
    $$
    T_1: \FinSet \times \Top \to \Top, (A,X) \mapsto X
    $$
    also is trivially an actegory. All the coherence-morphisms become identities.
\end{enumerate}
Now the pair $(1_\Top, \pi_1)$, comprised of the identity functor $1_\Top: \Top \to \Top$ of topological spaces together with the natural transformation $\pi_1 : T_0(A,X) = \Disc(A) \times X \Rightarrow X = T_1(A,X)$ as lineator, is an oplax $\Set$-linear functor.
\end{example}

\subsection{Tangent structures and Weil algebras}\label{sec:tangent_structures_as_weil}

In \cite{Leung2017}, Leung presented an alternative characterization of tangent categories which was inspired by synthetic differential geometry (SDG).  The characterization uses a monoidal category \weil{} whose objects are certain commutative algebras over the monoid $\mathbb{N}$ (see Definition \ref{def:weil}).  Importantly, the category \weil{} contains the dual numbers, i.e. the algebra $\mathbb{N}[x]/\langle x^2 \rangle $.  In SDG, the dual numbers play the role of an infinitesimal which represents tangent lines.  Leung uses this idea to recast tangent categories: from a strong monoidal functor from \weil{} to $\End^\circ(\mathbb X)$, one obtains a functor $T: \mathbb{X}\to \mathbb{X}$ as the image of the dual numbers.  The natural transformations $p,0,+,\ell,c$ of Definition \ref{def:tangent_cat} the images of certain maps $p,0,+,\ell,c$, which are intentionally given the same names.  The functor $T$ together with these natural transformations fulfill Definition \ref{def:tangent_cat}.  Leung's result is summarized as Theorem \ref{thm:Leung}.  

In \cite{Garner2018:embedding_theorem}, Garner extended this result to show that there is a 2-equivalence between a 2-category of Weil algebras and a 2-category of tangent categories.  That is, the equivalence extends to include strong tangent functors between tangent categories and linear tangent transformations between tangent functors.  This result is summarized in Theorem \ref{thm:Garner}.  In Garner's equivalence, strong tangent functors correspond to strong \weil-linear functors, and linear tangent transformations correspond to linear transformations. The goal of this section is to  generalize this result to 2-categories of tangent categories with lax tangent functors and general tangent transformations.  Our strategy for proving this will be to modify Garner's proof of Theorem \ref{thm:Garner}.  For this reason, in this section we review Weil algebras, Leung's correspondence, Garner's 2-equivalence and the proof with modifications as needed.

\begin{definition}\label{def:weil}\cite[Remark 3.18]{Leung2017}
Weil algebras are commutative algebras that are of the form $$\mathbb N[x_1, ... , x_n]/<x_i x_j | i \sim j>$$ for some equivalence relation $\sim$. A morphism of Weil algebras is  an $\mathbb N$-linear semi-ring homomorphism that preserves the unit. 
The category \weil{} consists of Weil algebras and morphisms of Weil algebras.
\end{definition}

\begin{remark}\label{rem:characterization_of_weil_objects}
Every object of \weil{} is generated by the dual numbers, $W= \mathbb N [x]/\langle x^2 \rangle$.  The category \weil{} has all finite coproducts and certain products.  
The coproduct, $W\otimes W$, is given by
$$
W \otimes W = \mathbb N[x,y]/<x^2 , y^2>.
$$
  The product, $W\times W$ is explicitly:
$$
W \times W = \mathbb N[x,y]/<x^2 , y^2,xy>.
$$
Note that $W\otimes W$ contains an element $xy$.  However, the product $W^2=W\times W$ does not contain $xy$. 
It is possible to use the explicit formulae for $W\otimes W$ and $W\times W$ to express coproducts and products in \weil{}.

Every Weil Algebra $A$ can be written uniquely as 
$$
A = W^{n_1}  \otimes ... \otimes W^{n_k}
$$
In fact, in \cite[Definition 4.2]{Leung2017}, Leung uses this characterization to define the category $\weil$ which he calls Weil$_1$. Later, after Proposition 6.6 he remarks that the objects of $\weil$ are described by a relation $\sim$ that corresponds to a disjoint union of full graphs, to an equivalence relation in our perspective.
\end{remark}
\begin{remark}\label{rem:double_use_p+0cell}
There are certain maps in \weil{} that behave like the natural transformations used  in Definition \ref{def:tangent_cat}.  For this reason, we use the same notation to denote these maps in \WEIL{} as the natural transformations in a tangent category.  Theorem \ref{thm:Leung} makes it clear that this ambiguity is justified. The maps in question are:
\begin{align*}
p&: W \to \mathbb N &  ax+b &\mapsto b
\\
+&: W \times W \mapsto W & ax_1 + bx_2 +c &\mapsto (a+b)x + c
\\
0&: \mathbb N \mapsto W & a & \mapsto 0 \cdot x + a
\\
c&: \mathrm W \otimes W \to W \otimes W & ax_1 x_2 + b x_1 + cx_2 +d & \mapsto  ax_1 x_2 + c x_1 + b x_2 +d
\\
\ell&: W \to W \otimes W & ax+b & \mapsto a x_1 x_2 + b
\end{align*}

\end{remark}
An immediate result of the observation that these maps behave like the tangent transformations is that \weil{} itself is a tangent category using this structure.  This is the content of the next Theorem.

\begin{theorem}\label{thm:Leung_tangent_on_Weil}\cite[Proposition 4.1]{Leung2017}
The (endo)functor
$$
W \otimes \underline{~}: \weil{} \to \weil{} 
$$
is the tangent functor of a tangent structure on \weil. 
\end{theorem}

The category $\End(\mathbb X)$ of endofunctors of a category $\mathbb X$ has a monoidal structure, given by the composition of endofunctors. The category {\weil} has a monoidal structure, given by the coproduct $\otimes$.

A strong monoidal functor is a functor $F: \mathbb X^\otimes \to \mathbb X'^\otimes$ between monoidal categories together with  natural isomorphisms
$$
\mu: F(- \otimes -) \cong F(-) \otimes F(-) ~\text{ and }~ \eta: F(I_\mathbb X) \cong I_\mathbb X'
$$
fulfilling some compatibility properties that are explained in \cite[Definition 1.27]{Heunen_book}. Using these monoidal structures, we can now present the anticipated characterization of tangent structures by Leung.

\begin{theorem}\label{thm:Leung}\cite[Theorem 14.1]{Leung2017}
Giving a tangent structure $(T,p,0,+,\ell,c)$ on $\mathbb X$ is equivalent to giving a strong monoidal functor $\tilde T: \weil^\otimes \to \End^\circ(\mathbb X)$ satisfying the following two conditions:
\begin{enumerate}
    \item
    The diagrams of the form
\[\begin{tikzcd}
	{A \otimes (B \times C)} & {A \otimes B} \\
	{A \otimes C} & A
	\arrow["{1_A \otimes\pi_C}"', from=1-1, to=2-1]
	\arrow["{1_A \otimes \epsilon_B}", from=1-2, to=2-2]
	\arrow["{1_A \otimes \epsilon_C}"', from=2-1, to=2-2]
	\arrow["{1_A \otimes\pi_B}", from=1-1, to=1-2]
	\arrow["\lrcorner"{anchor=center, pos=0.125}, draw=none, from=1-1, to=2-2]
\end{tikzcd}.\]
    are pullbacks in \weil{} and $\tilde T$ preserves them. 
    \item The diagrams of the form 
\[\begin{tikzcd}
	{W\times W} && {W \otimes W} \\
	{\mathbb N} && W
	\arrow["\nu", from=1-1, to=1-3]
	\arrow["{\mathrm{ev}_0}"', from=1-1, to=2-1]
	\arrow["\lrcorner"{anchor=center, pos=0.125}, draw=none, from=1-1, to=2-3]
	\arrow["{1_W \otimes p}", from=1-3, to=2-3]
	\arrow["0"', from=2-1, to=2-3]
\end{tikzcd}\]
where $\nu$ is a shorthand for $(1_W \otimes +)\circ (\ell \times (0_W \otimes 1_W ))$,
are pullbacks in \weil{} and $\tilde T$ preserves them. 
\end{enumerate}

The tangent bundle endofunctor $T$ is obtained from $\tilde T$ by setting  $T=\tilde T(W)$, the natural transformations $p,0,+,\ell,c$ in $\mathbb X$ are the images $\tilde T(p), \tilde T(0), \tilde T(+), \tilde T(\ell), \tilde T(c)$ of the morphisms in \weil{} with the same name. 

\end{theorem} 
We will call the pullbacks of Theorem \ref{thm:Leung} the \textbf{tangent pullbacks}.\footnote{In \cite{Leung2017}, the diagrams in Theorem \ref{thm:Leung}.1 are called \textbf{foundational} pullbacks.  Here we have used the terminology \textbf{tangent} pullbacks in order to include the pullbacks in the second part of the theorem.}
That the functor is monoidal means in particular that it sends the coproduct of algebras in  $\weil$ to the composition of functors.

\begin{remark} \label{remark:T_vs_hat_T}
    As a convention, we denote a tangent bundle endofunctor like in Definition \ref{def:tangent_cat} as $T: \mathbb X \to \mathbb X$ (without decoration) and the corresponding \weil-action
    as $\Te: \weil \to \End(\mathbb X)$ (with a hat). 
    We use the same notation for the corresponding functor $\Te: \weil \times \mathbb X \to \mathbb X$. 
\end{remark}

\begin{remark}
As a special case of the correspondence in Theorem \ref{thm:Leung}, the monoidal category \weil{} acts on itself by the tensor product. This monoidal functor $\weil \to \End(\weil)$ encodes the tangent structure of Theorem \ref{thm:Leung_tangent_on_Weil}.
\end{remark}

In order to make the correspondence between \weil-actegories and tangent categories more precise we use the language of 2-categories. 
For convenience, we review the main structures of 2-categories here - these will be required when we review Garner's proof of Theorem \ref{thm:Garner}.  For a full definition which includes the precise formulation of the associativity and unitality properties, see \cite[Definition 1.7.8]{riehl2017category}, \cite[Chapter VII]{maclane:71} or \cite[Proposition 2.3.4]{Johnson_Yau_2021_2-dimensional_categories}.  
A \textbf{2-category} consists of objects and for each pair of objects $A,B$ a hom-category $\underline{\Hom}(A,B)$ with a unit $1_A \in \underline{\Hom}(A,A)$ and a composition bifunctor $\circ: \underline{\Hom}(B,C) \times \underline{\Hom} (A,B) \to \Hom(A,C)$. The objects of $\underline{\Hom}(A,B)$ are called morphisms, the morphisms of $\underline{\Hom}(A,B)$ are called 2-morphisms.  The 2-morhpisms can be composed in two ways. Given morphisms $f,g,h : A \to B$ and 2-morphisms $\alpha: f \Rightarrow g$ and $\beta : g \Rightarrow h$, their vertical composition $\beta * \alpha: f \Rightarrow h$ is the composition in the category $\underline{\Hom}(A,B)$. Given morphisms $f_1, f_2: A \to B$ and $g_1, g_2 : B \to C$ and 2-morphisms $\gamma: f_1 \Rightarrow f_2, \delta: g_1 \Rightarrow g_2$, their horizontal composition $\delta \circ \gamma: g_1 \circ f_1 \Rightarrow g_2 \circ f_2$
is the application of the bifunctor $\circ$.

Both compositions are strictly associative and unital. We denote the names of 2-categories with all capitals in simplified font. For example the 2-category of categories, functors and natural transformations we denote as $\sf{CAT}$.

A \textbf{2-functor} $F: {\sf X} \to {\sf X}'$ between 2-categories is an assignment that sends objects to objects together with functors that send the hom-categories of $\sf X$ to the hom-categories of ${\sf X}'$ in a way that preserves compositions and units strictly. Thus it assigns morphisms to morphisms and 2-morphisms to 2-morphisms. For composable morphisms $f$ and $g$, a 2-functor $F$ fulfills $F(f \circ g) = F(f) \circ F(g)$ and $F(1) = 1$ (strict preservation of composition and unit). Since $F$ is a functor of hom-categories it also preserves vertical composition of 2-morphisms, i.e. for $\alpha: f \Rightarrow g$ and $\beta: g \Rightarrow h$ $F(\beta * \alpha) = F(\beta) * F(\alpha)$. The precise definition can be found in \cite[Definition 4.1.2]{Johnson_Yau_2021_2-dimensional_categories}.

A \textbf{2-natural transformation} between 2-functors $F,G: {\sf X} \to {\sf X}'$ is a collection of morphisms $\varphi_X: F(X) \to G(X)$ in ${\sf X}'$ for all objects $X$ of $\sf X$ such that for every morphism $f: X \to X'$ of $\sf X$ the diagram
\[\begin{tikzcd}
	{F(X)} & {F(X')} \\
	{G(X)} & {G(X')}
	\arrow["{G(f)}", from=1-1, to=1-2]
	\arrow["{\varphi_X}"', from=1-1, to=2-1]
	\arrow["{\varphi_{X'}}", from=1-2, to=2-2]
	\arrow["{G(f)}"', from=2-1, to=2-2]
\end{tikzcd}\]
strictly commutes. The notions of 2-functors and 2-natural transformations are completely analogous to the the classical notions of functors and natural transformations in 1-categories. 

\begin{definition}\cite[Definition 6.2.10]{Johnson_Yau_2021_2-dimensional_categories}
A \textbf{2-equivalence} is a 2-functor $F : {\sf X} \to {\sf X}'$, together with a 2-functor $G : {\sf X}' \to {\sf X}$ and 2-natural
isomorphisms
$$
1_{\sf X'} \cong  F \circ G \qquad 1_{\sf X} \cong  G \circ F.
$$
\end{definition}
The Whitehead Theorem for 2-categories \cite[Theorem 7.5.8]{Johnson_Yau_2021_2-dimensional_categories} states that a 2-functor $F: {\sf X} \to {\sf X}'$ is a 2-equivalence if and only if it is 
\begin{itemize}
    \item essentially surjective, i.e. surjective on isomorphism-classes of objects,
    \item fully faithful on morphisms, i.e. bijective on morphisms, and
    \item fully faithful on 2-morphisms, i.e. bijective on 2-morphisms.
\end{itemize}

We will use the language of 2-categories to describe the correspondence between 
various 2-categories of tangent categories and  $\weil$-actegories.  These 2-categories are our most important examples.

\begin{example}[The 2-category of tangent categories]\label{ex:2catsofTan}~
\begin{enumerate}
\item Let  $\TANG_\strong^\lin$ denote the 2-category with
\begin{itemize}
    \item tangent categories as in Definition \ref{def:tangent_cat} as objects,
    \item strong tangent functors as in Definition \ref{def:lax_tangent_functor} as morphisms, and
    \item linear tangent transformations as in Definition \ref{def:tangent_trafo} as 2-morphisms.
\end{itemize}
The 2-category $\TANG_\strong^\lin$ is the 2-category used in Theorem \ref{thm:Garner}. 
\item Let $\TANG_\lax$ denote the 2-category of tangent categories, lax tangent functors as in Definition \ref{def:lax_tangent_functor} and tangent transformations.  The other 2-categories in this example are sub 2-categories of this one.
\item Let $\TANG_\strong$ denote 2-category of tangent categories, strong tangent functors and tangent transformations as in Definition \ref{def:tangent_trafo}.
\item Let $\TANG_\lax^\lin$ denote 2-category of tangent categories, lax tangent functors and linear tangent transformations.
\end{enumerate}
\end{example}

These are the four choices that arise when considering all combinations of  morphisms and transformations from Definitions \ref{def:lax_tangent_functor} and \ref{def:tangent_trafo}.
We need all four combinations because in the next sections, strong functors will lead to differential objects and lax functors will lead to differential bundles. Differential objects and bundles have respective notions of additive and linear morphisms which will correspond to tangent transformations and linear tangent transformations. 

In order to state the correspondence between tangent categories and $\weil$-actegories we also need to have a 2-categorical description of \weil{}-actegories, which is explained in the following example.

\begin{example}\cite[Section 3]{Garner2018:embedding_theorem}
\begin{enumerate}
    \item Let $\weil \aACT_\strong^\lin$ denote the 2-category with
\begin{itemize}
    \item \weil{}-actegories as in Definition \ref{def:actegories} preserving the tangent pullbacks of Definition \ref{def:tangent_cat} (the foundational pullbacks and the universality of the vertical lift) as objects,
    \item strong \weil{}-linear functors as in Definition \ref{def:linear_functors} as 1-morphisms, and
    \item linear \weil{}-natural transformations as in Definition \ref{def:actegory_linear_transformations} as 2-morphisms.
\end{itemize}
\item Let $\weil \aACT_\oplax$ denote the 2-category of \weil-actegories preserving foundational pullbacks and the universality of the vertical lift, oplax \weil{}-linear functors as in Definition \ref{def:linear_functors} and \weil{}-natural transformations. The other 2-categories in this example are sub 2-categories of this one.
\item Let $\weil \aACT_\strong$ denote the 2-category of \weil-actegories preserving foundational pullbacks and the universality of the vertical lift, strong \weil{}-linear functors  and \weil{}-natural transformations as in Definition \ref{def:actegory_linear_transformations}. 
\item Let $\weil \aACT_\oplax^\lin$ denote the 2-category of \weil-actegories preserving foundational  pullbacks and the universality of the vertical lift, oplax \weil{}-linear functors and linear \weil{}-natural transformations. 
\end{enumerate}
\end{example}

For the 2-category $\TANG_\strong^\lin$, the   following theorem, due to Garner \cite{Garner2018:embedding_theorem}, is useful, because it
gives an equivalence of categories. This extends the correspondence between objects given in Theorem \ref{thm:Leung}. 

\begin{theorem}\cite[Theorem 9]{Garner2018:embedding_theorem}\label{thm:Garner}
The 2-category $\TANG_\strong^\lin$ is equivalent to $\weil \aACT_{\strong,t}^\lin$. 
\end{theorem}

In the following we want to consider lax functors and natural transformations because these will encode differential bundles and additive morphisms between them. Therefore we will now do a slight generalization of Theorem \ref{thm:Garner} with different choices for the morphisms and 2-morphisms.

Concretely, we will now extend Theorem \ref{thm:Garner} to $\TANG^\add_\strong$, $\TANG^\add_\lax$ and $\TANG^\lin_\lax$.

\begin{theorem}\label{thm:general_garner}
There are equivalences of 2-categories: 
\begin{align*}
\TANG_\lax &\simeq \weil \aACT_{\oplax}
\\
\TANG_\lax^\lin &\simeq \weil \aACT_{\oplax}^\lin
\\
\TANG_\strong &\simeq \weil \aACT_{\strong}
\end{align*}
\end{theorem}

Garner's proof of Theorem \ref{thm:Garner} proceeds by showing that there is a 2-functor between $\TANG_\strong^\lin$ and $\weil \aACT_\strong^\lin$ which provides an essentially surjective correspondence between the objects and an isomorphism between the hom categories.  In the remaining cases (lax or non-linear), Garner's proof will still provide a good framework.  The construction of the equivalence in this case works analogously with only minor differences.  However, the proof that there is a correspondence between 2-morphisms is significantly simplified in the non-linear case:  the correspondence becomes trivial because Diagram \ref{diagram:linear_actegory_trafo} in Definition \ref{def:actegory_linear_transformations} does not need to commute in the non-linear case.

We will show now that lax tangent functors and tangent transformations correspond exactly to lax \weil-linear functors and natural transformations under \ref{thm:Garner}. To prove this we use Garner's correspondence from \cite[Theorem 9]{Garner2018:embedding_theorem}. The following summary of Garner's proof is included in order to provide enough detail so that we can change the functors and natural transformations in Theorem \ref{thm:general_garner} and this change does not impact the proof.

\begin{lemma}\label{lem:lineators_are_unique}
    Let $(\mathbb X, \Te)$ and $(\mathbb X', \Te')$ be objects of $\weil \aACT_\oplax$ and let $F: \mathbb X \to \mathbb X'$ be a functor of underlying categories. Let $\hat \alpha, \hat \beta : F \circ \Te \to \Te' \circ (1_\weil \times F)$ be natural transformation fulfilling the lineator conditions of Definition \ref{def:linear_functors} that coincide on the dual numbers
    $$
    \hat \alpha_W = \hat \beta_W : F \circ \Te(W, -) \Rightarrow \Te'(W, - ) \circ F .
    $$
    Then $\hat \alpha$ and $\hat \beta$ coincide: $\hat \alpha = \hat \beta$.
\end{lemma}
\begin{proof}
    Every object $V$ of \weil{} is a coproduct of powers of the dual numbers, $V = W^{n_1} \otimes ... \otimes W^{n_k}$. 
    Since $\Te$ and $\Te'$ preserve foundational pullbacks and morphisms into pullbacks are uniquely determined by their components, $\hat \alpha_W = \hat \beta_W$ implies $\hat \alpha_{W^n} = \hat \beta_{W^n}$. Thus it only remains to show that the equality is also preserved for tensor products of Weil algebras. For that let $V$ and $V'$ be objects of \weil{} such that $\hat \alpha_V =  \hat \beta_V$ and $\hat \alpha_{V'} =  \hat \beta_{V'}$. Then the lineator condition from Definition \ref{def:linear_functors} means that
    \[\begin{tikzcd}
    	{F \circ \Te(V \otimes V', -)} & {F \circ \Te(V,-) \circ \Te(V', -)} \\
    	& {\Te'(V,-) \circ F \circ \Te(V', -)} \\
    	{\Te'(V \otimes V', -) \circ F} & {\Te'(V,-) \circ \Te'(V', -) \circ F}
    	\arrow["\cong"', from=1-1, to=1-2]
    	\arrow["{\hat \alpha_{V\otimes V'}}"', from=1-1, to=3-1]
    	\arrow["{\hat \alpha_V \circ \Te'(V',-)}", from=1-2, to=2-2]
    	\arrow["{\Te(V,-) \circ \alpha_{V'}}", from=2-2, to=3-2]
    	\arrow["\cong"', from=3-2, to=3-1]
    \end{tikzcd}\]
    commutes (and the same for $\hat \beta_{V \otimes V}$). Thus $\hat \alpha_{V \otimes V}$ and $\hat \beta_{V \otimes V}$ are completely determined by their $V$ and $V'$ components and therefore coincide. Since every object of $\weil$ is a tensor product of powers of $W$ this implies that $\hat \alpha$ and $\hat \beta$ coincide.
\end{proof}

\begin{proof}[Proof (of Theorem \ref{thm:Garner})]
To prove Theorem \ref{thm:Garner}, Garner explicitly constructs a 2-functor 
$$
E: \TANG_\strong^\lin \to \weil \aACT_{\strong,t}^\lin
$$
and verifies that it is an equivalence. The components of $E$ are defined as follows.

On objects,  $E$ must send a tangent category to a \weil-actegory. By Theorem \ref{thm:Leung}, a tangent structure on $\mathbb X$ is given by a functor
$$
T : \weil \to \End (\mathbb X)
$$
The transpose of this functor is the functor associated to a \weil-actegory structure
$$
\Te : \weil \times \mathbb X \to \mathbb X
$$
where $\Te (A,x) = T(A)(x)$. Since $T$ was monoidal, this is an actegory, and this defines $E(\mathbb{X})$. 

On morphisms $E$ must send a (strong) tangent functor $(F, \alpha): (\mathbb X, T) \to (\mathbb X', {T'})$ to a (strong) functor of \weil-actegories $(\hat F, \hat \alpha): (\mathbb X, \Te) \to (\mathbb X', \Te')$. Define $\hat F = F: \mathbb X \to \mathbb X'$. The natural transformation $\hat \alpha: F \circ \Te \Rightarrow \hat {T'} \circ (1 \times F)$ is defined more subtly using an intermediary category whose construction we now review.  

Let $\mathbb K_F$ be the 1-category whose objects are triples $(A,B, \tau)$ with $A \in \mathrm{End}(\mathbb X), B \in \mathrm{End}(\mathbb X')$ and $\tau: F \circ A \Rightarrow B \circ F$ a natural isomorphism. The morphisms $(\varphi, \psi): (A,B, \tau) \to (A',B', \tau')$ of $\mathbb K_F$ are pairs of natural transformations $\varphi: A \Rightarrow A'$ and $\psi: B \Rightarrow B'$ satisfying $\tau'* F(\varphi) = \psi_F * \tau$.  There is a tangent structure on $\mathbb K_F$ given by a tangent functor $T_{\mathbb K_F}:\mathbb K_F\to \mathbb K_F$ defined as follows:
$$
T_{\mathbb K_F}(A, B , \tau) := (T \circ A , {T'} \circ B , {T'}(\tau) * \hat \alpha_A  )
$$
and defined on morphisms using $T$ and $T'$ in a similar way.
By Theorem \ref{thm:Leung} this induces a functor $H: \weil \to \End(\mathbb K_F)$. The first two components of $T_{\mathbb K_F}$ are the precomposition with $T$ or ${T'}$ respectively, so the first two components of $H(V)$ for any Weil algebra, $V$, are given by the precomposition with $\Te$ and $\hat {T'}$.  
In particular for the object $(1_\mathbb X, 1_{\mathbb X'}, 1_F)\in \mathbb K_F$ and a Weil algebra $V$, there exists a natural transformation $\hat \alpha_V: F \circ \Te(V,-) \Rightarrow {T'}(V,-) \circ (F \times 1_{\weil})$ such that
$$
H(V)(1_\mathbb X, 1_{\mathbb X'}, 1_F) = (\Te (V,-) , \hat {T'} (V,-), \hat \alpha_V).
$$ 
For a morphism $\theta: V \to V'$ in \weil, we have $H(\theta)=(\Te(\theta, - ), \Te'(\theta, - ))$ such that $\hat \alpha_{V'} * \Te(\theta, - ) = \Te'(\theta, - ) \circ \hat \alpha_V   $
Therefore, the natural transformation $\hat \alpha_V$ is natural in $\weil$.  Thus, we have produced a natural transformation $\hat \alpha: F\circ\Te \to \Te'\circ (1\times F)$.
The functor $E$ sends a tangent functor $(F,\alpha_F)$ to the \weil-linear functor $(F, \hat \alpha)$ so defined. This is an object of $\mathbb K_F$, thus $\hat \alpha$ is an isomorphism, which means that strong tangent functors are sent to strong \weil-linear functors. One can prove that $\hat \alpha$ fulfills the lineator conditions using that $H$ is monoidal in order to expand $H(V \otimes V')(1_\mathbb X, 1_{\mathbb X'},1_F)$. The formula for the tangent structure $T_{\mathbb K_F}$ then forces the third component to be $\Te'(V, \hat \alpha_{V'}) * \hat \alpha_V$ proving that the lineator property in Definition \ref{def:linear_functors} holds.

To complete the definition of $E$, we must define $E$ on 2-morphisms.  For a 2-morphisms $\varphi: (F, \alpha) \Rightarrow (F', \alpha')$ in $\mathrm{TAN}_\strong^\lin$,  we must construct a natural transformation $F \Rightarrow F'$ fulfilling the linearity condition in Diagram \ref{diagram:linear_actegory_trafo}. We choose $E(\varphi):= \varphi$ as the natural transformation.
Garner indicates that a construction similar to the morphism argument using the auxiliary category $\mathbb K_F$ can be used to show that the linearity of $\varphi$ in the sense of Definition \ref{def:tangent_trafo} implies linearity of $E(\varphi)=\varphi$ in the sense of Definition \ref{def:actegory_linear_transformations}. In order to point out the modifications needed for the proof of Theorem \ref{thm:general_garner}, we will make this analogy explicit. The analogue to $\mathbb K_F$ is an intermediary category $\mathbb L_\varphi$, defined next. 

Suppose 
$\varphi:(F,\alpha) \to (F',\alpha')$ is a 2-morphism of $\TANG_\lax^\lin$, i.e. a linear tangent transformation. Then we define the 1-category $\mathbb L_\varphi$ whose objects are triples $(V,a,b)$, where $V\in\mathrm{Obj}(\weil )$ is a Weil algebra, $a: F \circ \Te(V, -) \Rightarrow \hat {T'}(V,-)\circ F$ and $b: F' \circ \Te(V, -) \Rightarrow \Te'(V,-)\circ F'$ are natural transformations such that  
\[\begin{tikzcd}
	{F \circ \Te(V, -) } & {\hat {T'}(V,-)\circ F} \\
	{F' \circ \Te(V, -) } & {\hat {T'}(V,-)\circ F'}
	\arrow["{\hat {T'}(V,-)\circ \varphi}", from=1-2, to=2-2]
	\arrow["\varphi"', from=1-1, to=2-1]
	\arrow["b"', from=2-1, to=2-2]
	\arrow["a", from=1-1, to=1-2]
\end{tikzcd}\] commmutes. A morphism $(V,a,b)\to(V',a',b')$ of $L_\varphi$ is a \weil -morphism $\gamma: V \to V'$ such that the diagrams
\[\begin{tikzcd}
	{F \circ \Te(V, -)} & {\hat {T'}(V,-)\circ F} && {F' \circ \Te(V, -)} & {\hat {T'}(V,-)\circ F'} \\
	{F \circ \Te(V', -)} & {\hat {T'}(V',-)\circ F} && {F' \circ \Te(V', -)} & {\hat {T'}(V',-)\circ F'}
	\arrow["a", from=1-1, to=1-2]
	\arrow["{a'}"', from=2-1, to=2-2]
	\arrow["{F\circ\Te(\gamma, - )}"', from=1-1, to=2-1]
	\arrow["{\hat {T'}(\gamma,-)\circ F}", from=1-2, to=2-2]
	\arrow["{\hat {T'}(\gamma,-)\circ F'}", from=1-5, to=2-5]
	\arrow["{\hat {T'}(\gamma,-)\circ F'}"', from=1-4, to=2-4]
	\arrow["{b'}"', from=2-4, to=2-5]
	\arrow["b", from=1-4, to=1-5]
\end{tikzcd}\]
commute.

Given objects $(V,a,b)$ and $(V',a',b')$ of $\mathbb L_\varphi$ such that the product $V \times V'$ exists in \weil, the diagram
\begin{equation}\label{diagram:VtimesV}
\begin{tikzcd}
	{(V \times V' , a \times_{1_F}a' , b \times_{1_{F'}}b')} & {(V',a',b')} \\
	{(V,a,b)} & {(\mathbb N , 1_F, 1_{F'})}
	\arrow["{\pi_1}"', from=1-1, to=1-2]
	\arrow["{\pi_0}", from=1-1, to=2-1]
	\arrow["{!}", from=1-2, to=2-2]
	\arrow[""{name=0, anchor=center, inner sep=0}, "{!}"', from=2-1, to=2-2]
	\arrow["\lrcorner"{anchor=center, pos=0.125}, draw=none, from=1-1, to=0]
\end{tikzcd}
\end{equation}
is a pullback in $\mathbb L_\varphi$. 
The tuple $(V \times V' , a \times_{1_F}a' , b \times_{1_{F'}}b')$ is a well defined object of $\mathbb L_\varphi$ because $\Te$ and $\Te'$ preserve foundational pullbacks and therefore 
$$
F \circ \Te(V \times V' , -) = F \circ \Te(V,-) \times F \circ \Te(V', -) \text{ and }  \Te'(V \times V' , -) \circ F = \Te'(V,-) \times F \circ \Te'(V', -) \circ F.
$$
The diagram is a diagram in $\mathbb L_\varphi$ because $a \times_{1_F} a'$ and $b \times_{1_{F'}} b'$ are compatible with the projection maps $\pi_0$ and $\pi_1$.
The diagram fulfills the universal property because any induced map into $(V \times V')$ is component-wise compatible with $a,a',b$ and $b'$ and thus compatible with $a \times a'$ and $b \times b'$.

There is a tangent structure on $\mathbb L_\varphi$ whose tangent functor, $T:\mathbb L_\varphi \to \mathbb L_\varphi$, is defined by tensoring with the dual numbers $W$:
$$
T(V, a, b) = (W \otimes V, {T'}(a) \circ \alpha , {T'}(b) \circ \alpha ').
$$
Here $T'(a) \circ \alpha$ is shorthand for the composite
$$
F \circ \Te (W \otimes V , -) \xrightarrow{\cong} F \circ T \circ \Te(V, -) \xrightarrow{\alpha} {T'} \circ F \circ \Te (V, -) \xrightarrow{{T'}(a)} {T'} \circ \hat {T'} (V, -) \circ F \xrightarrow{\cong} \hat {T'} (W \otimes V, -) \circ F,
$$
which includes the identification $\Te(W, -)$ with $T$ as in Remark \ref{remark:T_vs_hat_T}.  We define $T'(b)\circ \alpha'$ in the same way.
The square
\[\begin{tikzcd}
	{F \circ T \circ \Te(V,-)} & {T \circ F \circ \Te(V,-)} & {T \circ \Te(V,-) \circ F} \\
	{F' \circ T \circ \Te(V,-)} & {T \circ F' \circ \Te(V,-)} & {T \circ \Te(V,-) \circ F'}
	\arrow["\alpha", from=1-1, to=1-2]
	\arrow["{T (a)}", from=1-2, to=1-3]
	\arrow["{\alpha'}"', from=2-1, to=2-2]
	\arrow["{T(b)}"', from=2-2, to=2-3]
	\arrow["\varphi"', from=1-1, to=2-1]
	\arrow["{T(\varphi)}"', from=1-2, to=2-2]
	\arrow["{T(\Te(1_V,\varphi))}", from=1-3, to=2-3]
\end{tikzcd}\]
commutes because $\varphi: (F,\alpha) \to (F',\alpha')$ is a linear natural transformation (left square commutes) and $(V,a,b)$ is an object of $\mathbb L_\varphi$ (right square commutes).

In the proof of \cite[Theorem 9]{Garner2018:embedding_theorem}, Garner states that the projection from $\mathbb K_F$ to $\End(\mathbb X)$ preserves the tangent structure. Analogously, here the projection $\mathbb L_\varphi \to \weil$ preserves the tangent structure which is given by the tensor product with Weil algebras. This means that all foundational pullbacks in $\mathbb L_\varphi$ are pullbacks in the $\weil$-component. The reason for this is that, when combining the pullback diagram of \ref{diagram:VtimesV} with the monoidal structure of \weil,
foundational pullbacks in \weil{} remain pullbacks in $\mathbb L_\varphi$ when taking the tensor product with a constant object.

Since $\mathbb L_\varphi$ is a tangent category Theorem \ref{thm:Leung} gives us a functor $G_\varphi: \weil \to \End(\mathbb L_\varphi)$ which we sometimes also denote as $G_\varphi: \weil \times \mathbb L_\varphi \to \mathbb L_\varphi$. Since $G_\varphi$ preserves foundational pullbacks, the first component of $G_\varphi(W^n , (V,a,b))$ will be $W^n \otimes V$. Since $G_\varphi: \weil \to \End(\mathbb L_\varphi)$ is monoidal and every object of $\weil$ is a tensor product of powers of the dual numbers $W$, for any Weil-algebras $V,V'$, the first component of $G_\varphi(V, (V',a,b))$ will be $V \otimes V'$. Let 
$$
G_\varphi(V,(\mathbb N, 1_F, 1_{F'})) = (V, \tilde \alpha_V, \tilde \alpha_V')
$$
Then $\tilde \alpha_V : F \circ \Te(V, -) \Rightarrow \Te'(V,-) \circ F$ and $\tilde \alpha_V' : F' \circ \Te(V, -) \Rightarrow \Te'(V,-) \circ F'$ are natural in $V$ because $G_\varphi$ is a functor. In order to show that $\tilde \alpha$ and $\tilde \alpha'$ fulfill the lineator condition of Definition \ref{def:linear_functors} we can expand $G_\varphi(V \otimes V', (\mathbb N, 1_F, 1_{F'}))$ using monoidality into $G_\varphi(V , G_\varphi( V', (\mathbb N, 1_F, 1_{F'}))) = G_\varphi(V , (V', \tilde \alpha_{V'} ,\tilde \alpha_{V'}' ))$ and the tangent structure $T$ on $\mathbb L_\varphi$ together with the fact that every object of $\weil$ is built from $W$ then shows that 
$$
(V \otimes V' ,\tilde \alpha_{V \otimes V'} ,\tilde \alpha_{V \otimes V'}' ) = G_\varphi(V \otimes V', (\mathbb N, 1_F, 1_{F'})) = (V \otimes V' , \Te'(V,\tilde \alpha_{V'})*\tilde \alpha_V , \Te'(V,\tilde \alpha_{V'}')*\tilde \alpha'_V )
$$
proving the lineator property of Definition \ref{def:linear_functors}.

Since 
$$
G_\varphi(W,(\mathbb N, 1_F, 1_{F'})) = (W \otimes \mathbb N,T'( 1_F ) \circ \alpha , T'(1_{F'}) \circ \alpha') = (W, \alpha , \alpha'),
$$
the components $\hat \alpha_W = \alpha = \tilde \alpha_W$ and $\hat \alpha_W' = \alpha' = \tilde \alpha_W'$. Thus Lemma \ref{lem:lineators_are_unique} shows that for any Weil-algebra $V$, $\hat \alpha_V = \tilde \alpha_V$ and $\hat \alpha_V' = \tilde \alpha_V'$.

Therefore 
$$
G_\varphi(V,(\mathbb N, 1_F, 1_{F'})) = (V, \hat \alpha_V, \hat \alpha_V'),
$$
is an object of $\mathbb L_\varphi$ which implies that the diagram
\[\begin{tikzcd}
	{F \circ \Te(V, -) } & {\hat {T'}(V,-)\circ F} \\
	{F' \circ \Te(V, -) } & {\hat {T'}(V,-)\circ F'}
	\arrow["{\hat {T'}(V,-)\circ \varphi}", from=1-2, to=2-2]
	\arrow["\varphi"', from=1-1, to=2-1]
	\arrow["{\hat \alpha'}"', from=2-1, to=2-2]
	\arrow["{\hat \alpha}", from=1-1, to=1-2]
\end{tikzcd}\]
commutes. Thus $\varphi$ is a linear natural transformation of actegories, so we may define $E(\varphi)=\varphi$.

Garner's proof concludes by showing that $E$ is essentially surjective on objects and fully faithful on morphisms and 2-morphisms.

It is essentially surjective on objects because of Theorem \ref{thm:Leung}. 

It is full on morphisms because for any strong \weil-linear functor $(F,\hat \alpha)$, the $W$-component $(F, \hat \alpha_W)$ is a strong tangent functor such that $E(F,\hat \alpha_W) = (F, \hat \alpha)$.

It is faithful on morphisms because $E(F, \alpha) = (F, \hat \alpha)$ where $\hat \alpha_W = \alpha$.  Thus, two different strong tangent functors $(F, \alpha)$ and $(F',\alpha')$ will be sent to two \weil-linear functors $(F,\hat \alpha)$ and $(F', \hat \alpha')$ which are different because the $W$ component of $\hat \alpha$ is $\alpha$ and analogous for $\hat \alpha'$.

It is full on 2-morphisms because the $W$ component of the linearity condition of a \weil-natural transformation implies the linearity condition of a tangent transformation.

It is faithful on 2-morphisms because the data of a 2-morphisms is the same in $\TANG_\strong^\lin$ and $\weil \aACT^\lin_{\strong,t}$.

\end{proof}
\begin{proof}[Proof (of Theorem \ref{thm:general_garner})] 
There are three cases to prove, and in each case we adapt Garner's proof as necessary to account for the differences between the 2-categories which are involved.  We take each of them in turn.

\paragraph*{Tangent Categories with lax functors and linear transformations}
In order to prove that $\TANG^\lin_\lax \cong \weil \aACT^\lin_{\oplax}$, we need to relax the condition of strong morphisms to oplax morphisms. We need to define a functor 
$$
E': \TANG^\lin_{\lax} \to \weil\aACT^\lin_{\oplax}
$$
and show that $E'$ is essentially surjective on objects, fully faithful on morphisms and fully faithful on 2-morphisms.
Let $(\mathbb X, T)$ be an object of $\TANG^\lin_\lax$.
The functor $E'$ is defined on objects by letting $E'(\mathbb X)$ be the Weil-actegory whose Weil action is given by the adjoint  $\Te: \weil \times \mathbb X \to \mathbb X$ of the tangent functor $T: \weil \to \End (\mathbb X)$, exactly as was the case for the functor $E$ in the proof of Theorem \ref{thm:Garner}.

As in that proof, 
given a lax tangent functor $(F,\alpha):(\mathbb X, T)\to (\mathbb X', {T'})$ let
$\mathbb K'_F$, to be the category  of triples $(A \in \End (\mathbb X), B \in \End (\mathbb X') , \tau : F \circ A \Rightarrow B \circ F)$ with 
morphisms $(\varphi, \psi): (A,B, \tau) \to (A',B', \tau')$ of $\mathbb K_F'$ given by pairs of natural transformations $\varphi: A \Rightarrow A'$ and $\psi: B \Rightarrow B'$ satisfying $\tau'* F(\varphi) = \psi_F * \tau$. 
The only difference between $\mathbb K_F$ and $\mathbb K'_F$ is that $\tau$ is not necessarily an isomorphism. 
%One can give $\mathbb K_F'$ the structure of a tangent structure where the tangent bundle functor is 
The functor $T_{\mathbb K_F'}: \mathbb K_F' \to \mathbb K_F'$ given by
$$
T_{\mathbb K_F'}(A, B , \tau) := (T \circ A , {T'} \circ B , {T'}(\tau) \circ \alpha_A ).
$$
is a tangent bundle functor.  
This is because the only difference from the proof of Theorem \ref{thm:Garner} is that in that proof, $\tau$ and $\alpha$ needed to be isomorphisms.  Since this did not contribute to the tangent structure, the same proof works in this case.
This tangent structure gives rise to a monoidal functor $H: \weil \to \End(\mathbb K_F)$, and evaluation of this functor at the object $(1_\mathbb X, 1_{\mathbb X'}, 1_F)\in \mathbb K_F$ gives us, for each Weil algebra $V$, 
$$
H(V)(1_\mathbb X, 1_{\mathbb X'}, 1_F) = (\Te (V,-) , \hat {T'} (V,-), \hat \alpha_V)
$$ where $\hat \alpha_V: F \circ \Te(V,-) \Rightarrow {T'}(V,-) \circ (F \times 1_{\weil})$ is a natural transformation. The functor $E'$ sends $(F,\alpha)$ to $(F, \hat \alpha)$ so defined. Since $\hat \alpha$ is the third component of an object of $\mathbb K_F'$, it is not required to be an isomorphism. This means that lax tangent functors are sent to lax \weil-linear functors. The functor $E'$ is defined on 2-morphisms exactly like in the original proof (of Theorem \ref{thm:Garner}).

The difference to the proof of Theorem \ref{thm:Garner} is that we don't just allow strong \weil{}-linear functors but also lax \weil{}-linear functors as morphisms in the target. Therefore every object in $\weil\aACT^\lin_\oplax$ is still isomorphic the image of an object of $\TANG_\lax^\lin$. (It even is isomorphic by a strong functor which in particular is a lax functor). Thereby $E'$ is automatically essentially surjective on objects. It also is fully faithful on 2-morphisms for the very same reasons as in the proof of \ref{thm:Garner}.

Next we need to show that $E'$ is fully faithful on morphisms. So let $\mathbb X$ and $\mathbb X'$ be
tangent categories and $(F, \hat \alpha) : E(\mathbb X) \to  E(\mathbb X')$ a map of the corresponding \weil -actegories.
The maps $\alpha(W,-)$ constitute a natural transformation $F (\Te(W, - )) \Rightarrow  \hat {T'} F (-)$ which, since
$\Te(W,-) \cong T$ in both domain and codomain, determines and is determined by one $\alpha: F \circ  T \Rightarrow  T' \circ F$.

\paragraph*{Tangent categories with strong functors} In order to prove that $\TANG_\strong \cong \weil \aACT_{\strong,t}$, we need to relax the condition of linear natural transformations to any natural transformations.  We need to define a functor 
$$
E'': \TANG_\strong \to \weil \aACT_{\strong,t}
$$
and prove that it is essentially surjective on objects, fully faithful on morphisms and fully faithful on 2-morphisms.
On objects it is defined through the adjunction and on morphisms through the category $\mathbb K_F$, like $E$ from the  proof of Theorem \ref{thm:Garner}.

The only substantial difference to the proof of Theorem \ref{thm:Garner} is on 2-morphisms. Since the transformations are not linear, we choose $E''(\varphi) = \varphi$, but now we do not need to verify linearity.

We defined $E''$ to be the identity on 2-morphisms, thus $E''$ is automatically fully faithful on 2-morphisms. It is essentially surjective on objects and fully faithful on morphisms because on objects and morphisms it is the same as $E$ from the original proof (of Theorem \ref{thm:Garner}).

\paragraph*{Tangent categories with lax functors} In order to prove that $\TANG_\lax \cong \weil \aACT_{\oplax}$ we need to relax both conditions, the strong functors to lax functors and the linear natural transformations to general natural transformations. We need to define a functor 
$$
E''': \TANG_\lax \to \weil \aACT_{\oplax}
$$
and prove that it is essentially surjective on objects, fully faithful on morphisms and fully faithful on 2-morphisms. We define $E'''$ to be the same as $E'$ on objects and morphisms and the same as $E''$ on 2-morphisms. From the analogous results about $E'$ and $E'$ it now follows that $E'''$ is essentially surjective, fully faithful on morphisms and fully faithful on 2-morphisms.

\end{proof}

\begin{remark}\label{rem:tanCat_as_actegories}
Inspired by Theorem \ref{thm:general_garner} we will from now on describe tangent categories through their \weil-actegory structure $(\mathbb X, \Te: \weil \times \mathbb X \to \mathbb X)$, lax/strong tangent functors through their corresponding colax/strong \weil-linear functors and (linear) tangent transformations through their (linear) \weil-natural transformations. When describing a tangent category through its \weil-actegory structure, we call it a \textbf{tangent actegory}.

Recall from Remark \ref{rem:actegory_notation_index}, that given a \Weil{} algebra $A$ we will denote the functor $\Te(A, -): \mathbb X \to \mathbb X$ by $\T_A$. In addition, given a tangent actegory $(\mathbb X,\Te)$ we say $(E,M,q,\zeta, \sigma, \lambda)$ is a differential object in the \weil{}-actegory $(\mathbb X,\Te)$ if it is a differential object in the corresponding tangent category $(\mathbb  X, \T_W)$.
\end{remark}

\subsection{The tangent actegory $\mathbb N^\bullet$}
In this section we will study the category $\mathbb N^\bullet$ from Definition \ref{def:N_bullet} with the tangent structure $(\mathbb N^\bullet, D)$ defined in Example \ref{ex:tan_structure_on_N_bullet}. The tangent category $\mathbb N^\bullet$ is the universal tangent category with a differential object, in the sense that it will classify differential bundles and differential objects in Theorem \ref{thm:equivalence_bundle_categories}. 

Due to Theorem \ref{thm:Garner} this tangent structure has a corresponding \weil-actegory structure via via the functor $D_\bullet$ satisfying 
$$D_\bullet(A, \mathbb N^k)= \underline A \otimes \mathbb N^k,$$ where $\underline A$ is the underlying additive monoid of $A$. Under the correspondence of Theorem \ref{thm:Garner} this encodes the tangent structure of Example \ref{ex:tan_structure_on_N_bullet}, because 
$$
D_\bullet(W, \mathbb N^k) = W \otimes \mathbb N^k \cong \mathbb N^k \times \mathbb N^k.
$$
One can check that the images of the maps in \weil{} also provide the $p, 0, + , \ell$ and $c$ from Example \ref{ex:tan_structure_on_N_bullet}.
as well. Equivalently this can be seen as an example of Leung's Theorem \ref{thm:Leung}. The functor $\weil \to \Fun(\mathbb N^\bullet , \mathbb N^\bullet)$ is the adjoint of the tensor product sending the Weil algebra $A$ to the endofunctor $\underline A \otimes  -$.

However, the equivalences in Theorems \ref{thm:Garner} and \ref{thm:general_garner} do not only provide a correspondence on objects. They also give a correspondence on morphisms and 2-morphisms. We will showcase this correspondence for the functor $F: \mathbb N^\bullet \to \SmMan$ from Example \ref{ex:functor N bullet to smooth}.

\begin{example}\label{ex:functor_N_to_smooth_as_weil_linear_functor}
The functor $F: \mathbb N^\bullet \to \SmMan$ from Example \ref{ex:functor N bullet to smooth} that sends $\mathbb N^k$ can be seen as a strong $\weil$-linear functor of actegories in the following way: The underlying functor is the functor from Example \ref{ex:functor N bullet to smooth}. 
For the lineator we choose the identity:
$$
\hat \alpha_{W^{n_1} \otimes ... \otimes W^{n_k} , \mathbb N^m} = 1:
F(D_\bullet(W^{n_1} \otimes ... \otimes W^{n_k} , \mathbb N^m)) \to \Te(W^{n_1} \otimes ... \otimes W^{n_k} , F(\mathbb N^m)).
$$
We can do this because the domain can be simplified into
$$
F(D_\bullet(W^{n_1} \otimes ... \otimes W^{n_k} , \mathbb N^m)) = 
F(\mathbb N^{m \cdot (n_1+1) \cdot ... \cdot (n_k+1)}) = 
\mathbb R^{m \cdot (n_1+1) \cdot ... \cdot (n_k+1)}.
$$
The target can be simplified into 
$$
\Te(W^{n_1} \otimes ... \otimes W^{n_k} , F(\mathbb N^m)) = \Te(W^{n_1} \otimes ... \otimes W^{n_k} , \mathbb R^m) = \mathbb R^{m \cdot (n_1+1) \cdot ... \cdot (n_k+1)}.
$$
As the domain and the target coincide, we can choose the lineator $\hat \alpha$ to be the identity.
\end{example}
This example is important since the main result of this paper, the classification in Theorem \ref{thm:equivalence_bundle_categories}, will state that such functors exactly correspond to differential bundles.

Recall from Definition \ref{def:N_bullet} that morphisms in $\mathbb N^\bullet$ from $\mathbb{N}^m$ to $\mathbb{N}^n$ are given by matrices with components $f = (f_{ij})_{1\leq i\leq n ,1\leq j \leq m}$ where $f_{ij} \in\mathbb N$ for all $i$, $j$.
Lemma \ref{lemma:morph_in_N_bullet} will now re-write the application of a matrix as an iterated sum.

For this we use the additive bundle structure $\zeta: \mathbb N^0 \to \mathbb N^1$ and $\sigma: \mathbb N^2 \to \mathbb N^1$ on $\mathbb N^1$ from Example \ref{ex:N_bullet_differential}. Let $\sigma_k : \mathbb N^k \to \mathbb N^1$ be defined recursively by $\sigma_0 = \zeta : \mathbb N^0 \to \mathbb N^1$ and $\sigma_k = \sigma \circ (1_\mathbb N \times \sigma_{k-1}) : \mathbb N^k \to \mathbb N^1$. Recall that $\Delta_k: \mathbb N \to \mathbb N^k$ is the diagonal.

\begin{lemma}\label{lemma:morph_in_N_bullet}
Let 
$$
f = (f_{ij})_{1\leq i\leq n ,1\leq j \leq m} : \mathbb{N}^m \to \mathbb{N}^n
$$
be an arbitrary morphism in $\mathbb N^\bullet$. Then $f$ can be decomposed into components
$$
f = \langle f_1 , ... , f_n\rangle
$$
with components $f_i : \mathbb N^m \to \mathbb N$ that are given by composition
$$
f_i = \sigma_m \circ \langle \sigma_{f_{i1}} \circ \Delta_{f_{i1}} \circ \pi_1 , ... , \sigma_{f_{im}} \circ \Delta_{f_{im}} \circ \pi_m \rangle
$$
\end{lemma}
\begin{proof}
    This is just a way to formulate the application of a matrix $f = (f_{ij})_{1\leq i\leq n ,1\leq j \leq m}$ to a vector $v$
    $$
    f(v)_i = \sum_{j=1}^m f_{ij} \cdot v_j
    $$
    using the addition from Example \ref{ex:N_bullet_differential}. We first replace the multiplication with scalars by a sum of constant terms
    $$
    \sum_{j=1}^m f_{ij} \cdot v_j = \sum_{j=1}^n \sum_{r = 1}^{f_{ij}} v_j
    $$
    and then replace the sums with $\sigma$ to obtain
    \begin{align*}
    \sum_{j=1}^n \sum_{r = 1}^{f_{ij}} v_j &= \sigma_m \circ \langle \sigma_{f_{i1}}(v_1, ... , v_1) , ... , \sigma_{f_{im}}(v_m, ... , v_m)\rangle 
    \\&= \sigma_m \circ \langle \sigma_{f_{i1}}\circ \Delta_{f_{i1}} (v_1) , ... , \sigma_{f_{im}}\circ \Delta_{f_{im}} (v_n)\rangle
    \\&= \sigma_m \circ \langle \sigma_{f_{i1}}\circ \Delta_{f_{i1}} \circ \pi_1, ... , \sigma_{f_{im}}\circ \Delta_{f_{im}}\circ \pi_n\rangle (v).
    \end{align*}
    
\end{proof}

\section{\Nice{} functors}

This section introduces \nice{} functors from on tangent actegory, $(\mathbb X, T_\bullet)$, to another, $(\mathbb X', T'_\bullet)$.  Like a tangent functor, a differential functor will consist of a functor $F:\bbX\to \bbX'$ together with a natural transformation $\hat \alpha:F\circ T \Rightarrow T'\circ (1\times F)$, but $F$ and $\hat \alpha$ will be required to preserve additional structure (as in Definition \ref{def:differential_functors}). The key difference between  differential functors and the tangent functors of Definition \ref{def:lax_tangent_functor} or the \weil{}-linear functors of Definition \ref{def:linear_functors} is that differential functors preserve differential bundles (Definition \ref{def:differential bundle}) and bundle morphisms (Definition \ref{def:differential_bundle_morphisms}). 

We construct a category of differential functors and differential natural transformations.  In light of the results of Section 3, we construct differential functors from \weil{}-linear functors and we use Weil-actegories in place of the tangent categories of Definition \ref{def:tangent_cat}. In light of the Theorems \ref{thm:Garner} and \ref{thm:general_garner} , we use \weil{}-actegories as our model of tangent categories, and differential functors will be \weil{}-linear functors with additional structure. We begin this section by providing a definition of what it means to be a differential bundle in a Weil-actegory by defining it to be what corresponds to Definition \ref{def:differential bundle} under the correspondence of Theorem \ref{thm:general_garner}.  In Section 3, we considered strong and lax \weil{}-linear morphisms as well as tangent transformations and linear tangent transformations. This begs the question of what choices we should make in defining both objects and morphisms of ``the'' category of differential functors.  Instead of making a single choice, we will define all four possible categories as sub-categories of the category of morphisms in the 2-category $\TANG_{\lax}$ and denote them $\DiffFun, \DiffFun^\strong, \DiffFun_\lin$ and $\DiffFun^\strong_\lin$.

Finally, we will establish the fact that both lax and strong differential functors send differential objects to differential bundles in Proposition \ref{prop:Dobj_to_Dbun}.  Given a lax differential functor $(F, \alpha):(\bbX, T) \to (\bbX', T')$, we define an evaluation functor $\Eval:\DObj(\bbX) \times \DiffFun(\mathbb X , \mathbb X') \to \DBun(\bbX')$ which sends a functor $F$ and a differential object $A$ to the bundle obtained by evaluating $F$ at the object $A$ over  $F(*)$, where $*$ is the terminal object of $\mathbb X$.

Analogous statements hold for $\DiffFun^\strong, \DiffFun_\lin$ and $\DiffFun^\strong_\lin$.

The fact that this assignment is functorial is established in Propopositions \ref{prop:Dobj_to_Dbun}, \ref{prop:preserve_differential_structure}, \ref{prop:eval_on_morphisms} and Theorem \ref{thm:Eval_Dobj_DiffFun_DBun}. Each of the four propositions is for one of the four different choices of categories of \nice{} functors. Its correspondence to the classical evaluation functor $\mathrm{eval}: {\mathbb X \times \Fun(\mathbb X , \mathbb X')} \to {\mathbb X'}$, which sends $ A,F $ to $ F(A)$, makes the diagram

\[\begin{tikzcd}
	{\DObj(\mathbb X) \times\DiffFun(\mathbb X, \mathbb X')} & {\DBun(\mathbb X')} \\
	{\mathbb X \times \Fun(\mathbb X , \mathbb X')} & {\mathbb X'}
	\arrow["\Eval", from=1-1, to=1-2]
	\arrow[hook, from=1-1, to=2-1]
	\arrow["{\pi_E}", from=1-2, to=2-2]
	\arrow["{\mathrm{eval}}", from=2-1, to=2-2]
\end{tikzcd}\] commute.
Here the functor $\pi_E$ sends a differential bundle $(E,M,q,\zeta,\sigma,\lambda)$ to the total space of the bundle, $E$.

There are variations of the functor $\Eval$ whose target is either $\DBun(\bbX')$ or $\DObj(\bbX')$, depending on whether $(F, \alpha)$ is lax or strong. All the variations of $\Eval$ are lined up in Corollary \ref{cor:eval_difffun_to_dbun_dobj}.

Theorem \ref{thm:Eval_Dobj_DiffFun_DBun} and Corollary \ref{cor:eval_difffun_to_dbun_dobj} give us a way to produce new differential bundles in $\mathbb X'$ out of differential objects in $\mathbb X$.  We conclude this section with the Example \ref{ex:given_functor_from_N_bullet_get_diffobj} using the tangent category $(\mathbb{N}^\bullet, T)$ of Example \ref{ex:tan_structure_on_N_bullet}.  In subsequent sections, we will show that in fact every differential bundle in $\bbX$ is obtained from $\mathbb N^\bullet$.

 We begin by identifying objects of a \weil-actegory which are differential objects in the corresponding tangent category.  The following definition makes this precise. Recall Definitions \ref{def:differential bundle} and \ref{def:differential_bundle_morphisms} for the definition of a differential bunle in a tangent category.

Recall from Remark \ref{rem:actegory_notation_index} that we denote the functor $\Te(A, -): \mathbb X \to \mathbb X$ by $\T_A$.
\begin{definition}\label{def:weil-differential bundles}
    Let $(\mathbb X, \Te: \weil \times \mathbb X \to \mathbb X)$ be a tangent actegory, and let $A$ be a  Weil algebra.
    \begin{enumerate}
        \item A tuple $(E,M,q:E\to M,\zeta:M\to E, \sigma:E_2\to E, \lambda:E\to \T_W(E))$ is a \textbf{\weil-differential bundle} in $(\mathbb X,\Te)$ if it is a differential bundle in the sense of Definition \ref{def:differential bundle} in the corresponding tangent category $(\mathbb  X, \T_W)$.  We will abbreviate $(E, M, q, \zeta, \sigma, \lambda)$ to $(E,M)$ when the name of the structure maps are not needed.
        \item Given \weil{}-differential bundles $(E,M)$ and $(E',M')$, a pair of maps $f: E \to E', g:M\to M'$ is a \textbf{\weil-differential bundle morphism} if $(f,g)$ is a differential bundle morphism in the corresponding tangent category $(\mathbb  X, \T_W)$.
        \item A \weil-differential bundle morphism is linear/additive if it is linear/additive in the corresponding tangent category $(\mathbb  X, \T_W)$.
    \end{enumerate}
Let $\weil \DDBun(\bbX)$, $\weil \DDBun_{\add}(\bbX)$ and $\weil \DDBun_{\lin}(\bbX)$ be the categories of \weil-differential bundles with any, linear or additive morphisms.
Denote $\weil \DDObj(\bbX)$, $\weil \DDObj_{(\add)}(\bbX)$ and $\weil \DDObj_\lin(\bbX)$ the full subcategories of \weil-differential objects with any, linear or additive morphisms. 
\end{definition}
\begin{lemma}\label{lemma:weil-diffbundles_vs_diffbundles}
    Let $(\bbX, \Te)$ be a tangent actegory with corresponding tangent category $(\bbX, \T_W)$.  There are isomorphisms of categories 
    $$
    \DBun(\bbX,\T_W) \cong \weil\DDBun(\bbX,\Te)~~~,~~~
    \DObj(\bbX,\T_W) \cong \weil\DDObj(\bbX,\Te)~,~
    $$$$
    \DBun_\add(\bbX,\T_W) \cong \weil\DDBun_\add(\bbX,T)~,~
    \DObj_\add(\bbX,\T_W) \cong \weil\DDObj_\add(\bbX,\Te)~,~
    $$$$
    \DBun_\lin(\bbX,\T_W) \cong \weil\DDBun_\lin(\bbX,\Te)~,~
    \DObj_\lin(\bbX,\T_W) \cong \weil\DDObj_\lin(\bbX,\Te)~,~
    $$.
\end{lemma}
\begin{proof}
By Definition \ref{def:weil-differential bundles} there are functors in both directions that do not change the data of an object or a morphism. The fact that these functors are inverse to one another is immediate.
\end{proof}
Lemma \ref{lemma:weil-diffbundles_vs_diffbundles} allows us to slightly abuse terminology by using the word ``differential bundle'' for objects on both sides of the equivalence (using the conventions from Remarks \ref{remark:T_vs_hat_T} and \ref{rem:tanCat_as_actegories}). In particular we will also use $\DBun(\bbX,\Te)$ to denote $\weil\DDBun_\add(\bbX,T)$.

\begin{definition}[Strong and lax \nice{} functors]\label{def:differential_functors}
Suppose $(\mathbb X, \Te)$ and $(\mathbb X' , \Te')$ are \weil{}-actegories %(i.e. tangent categories), 
and suppose that that $\mathbb X$ has a terminal object $*$. 
\begin{enumerate}[label = (\alph*)]
    \item A \textbf{lax \nice{} functor} from $(\mathbb X, \Te)$ to $(\mathbb X' , \Te')$ is an oplax \weil{}-linear functor, 
    $
    (F:\mathbb X \to \mathbb X',\hat \alpha: F \circ \Te \Rightarrow \Te' \circ (1\times F) ),
    $
    such that for all objects $A$ of \weil{},
    \begin{enumerate}
        \item $T' _A \circ F$  preserves pullbacks over the terminal object $*$, and 
        \item $\hat \alpha_{A, -}$ is Cartesian, i.e. for each $f\in \Hom_{\mathbb X}(X,Y)$, the naturality diagram 
        \[\begin{tikzcd}
        	{F \circ \T_A(X)} && {\T_A'\circ F(X)} \\
        	{F \circ \T_A(Y)} && {\T_A' \circ F(Y)}
        	\arrow["{\hat \alpha_{A,X}}", from=1-1, to=1-3]
        	\arrow["{F \circ \T_A(f)}"', from=1-1, to=2-1]
        	\arrow["{\T_A' \circ F(f)}", from=1-3, to=2-3]
        	\arrow[""{name=0, anchor=center, inner sep=0}, "{\hat \alpha_{A,Y}}", from=2-1, to=2-3]
        	\arrow["\lrcorner"{anchor=center, pos=0.125}, draw=none, from=1-1, to=0]
        \end{tikzcd}\]
        for the natural transformation $\hat \alpha_{A, -}:F \circ \T_A \Rightarrow \T'_A \circ F$
        is a pullback.
    \end{enumerate}
    \item A \textbf{strong \nice{} functor} is a lax \nice{} functor
    such that $F(*)$ is a terminal object in $\mathbb X'$.
\end{enumerate}
\end{definition}
If the tangent bundle of the terminal object is terminal, $T(*) \cong *$ in both $\mathbb X$ and $\mathbb X'$,
a consequence of this definition is that a strong \nice{} functor is a strong \WEIL -linear functor,
i.e. the natural transformation $\alpha$ is actually a natural isomorphism. Let $*_\mathbb X$ (resp. $*_{\mathbb X'}$ denote the terminal object in $\mathbb X$ (resp. $\mathbb X'$). The reason a strong \nice{} functor is a strong \Weil-linear functor is that the naturality diagram of $ \alpha$ with respect to $!:X \to *_\mathbb X$
\[\begin{tikzcd}
	{F(\T_A(X))} && {\T'_A(F(X))} \\
	{*_{\mathbb X'} = F(\T_A(*_\mathbb X))} && {\T'_A(F(*_\mathbb X)) = *_{\mathbb X'}}
	\arrow["{!}", from=1-3, to=2-3]
	\arrow["{!}"', from=1-1, to=2-1]
	\arrow["{ \hat \alpha_{(A,X)}}", from=1-1, to=1-3]
	\arrow["{1_*}"', from=2-1, to=2-3]
\end{tikzcd}\]
is a pullback and thus $\hat \alpha_{(A,X)}$ is an isomorphism.
\begin{example}~
Let $(\SmMan{},T)$ be the tangent category of smooth manifolds with smooth maps. It was defined in Example \ref{ex:smooth_manifolds_vs} using the axiomatic formulation of tangent categories in Definition \ref{def:tangent_cat}. It can be turned into a \weil-actegory using Theorem \ref{thm:general_garner}).
\begin{enumerate}[label = (\alph*)]
\item 
Let $(-)^2: \SmMan{}\to \SmMan{}$ denote the squaring functor.  For smooth manifolds $M$ and $N$ and a smooth map $f:M\to N$, this functor satisfies $(-)^2(M)=M^2=M\times M$ and $(-)^2(f)=f\times f$.  This functor,  together with the canonical isomorphism $\alpha:T(M \times M) \to T(M) \times T(M)$, is a tangent functor functor as in Definition \ref{def:lax_tangent_functor}. Under Theorem \ref{thm:general_garner} it corresponds to a \weil-linear functor $(F,\hat \alpha)$ which is a strong differential functor.
\item 
Let $S^1 \times -$ be the product with the circle functor. For smooth manifolds $M$ and $N$ and a smooth map $F:M \to N$, this functor satisfies $(S^1 \times -)(M) = S^1 \times M$ and $(S^1 \times -)(f) = 1_{S^1} \times f$.
This functor, together with the map $\beta: = 0 \times 1: S^1 \times T(M) \to T(S^1) \times T(M) \cong T(S^1 \times M)$, is a tangent functor and thus corresponds to a \weil-linear functor $(F, \hat \beta)$. It is a lax \nice{} functor but not a strong \nice{} functor because it sends the terminal object $*$ to $S^1 \times * \cong S^1$ which is not terminal.
\end{enumerate}
An important difference between these examples is that $S^1 \times * \ncong *$ whereas $* \times * \cong *$.  This showcases the difference between a lax differential functor such as $S^1\times -$ and a strong differential functor such as $(-)^2$.
\end{example}

The main goal of this section is to construct an evaluation functor
$$
\Eval : \DObj_\add(\mathbb X) \times \DiffFun(\mathbb X , \mathbb X') \to \DBun_\add(\mathbb X')
$$
in Theorem \ref{thm:Eval_Dobj_DiffFun_DBun}, that sends a differential object and a \nice{} functor to a differential bundle. The explicit definition of the category \DiffFun{} is in Definition \ref{def:DiffFun}.

The next three propositions show (in order) that this assignment is  well-defined on objects, that it is well-defined on morphisms of $\DObj$ and that it is well-defined on morphisms of \DBun.

First we describe the assignment on objects. Recall notation convention for differential objects from Remark \ref{rem:DiffObj}.
\begin{proposition}\label{prop:Dobj_to_Dbun}
    Let $(\bbX, \Te)$ and $(\bbX',   \Te')$ be tangent actegories and let $(F, \hat{\alpha}): (\bbX, \Te)\to (\bbX', \Te')$ be an oplax \weil-linear functor.  Let $(X, \sigma, \zeta, \lambda)$ be a differential object in $\bbX$.
    \begin{enumerate}[label = (\alph*)]
        \item If $(F, \hat{\alpha})$ is a lax differential functor, then then $(F(X), F(\ast), F(!), F(\sigma), F(\zeta), \hat{\alpha}_X\circ F(\lambda))$ is a differential bundle.
        \item If $(F, \hat{\alpha})$ is a strong differential functor, then then $(F(X), F(\ast), F(!), F(\sigma), F(\zeta), \hat{\alpha}_X\circ F(\lambda))$ is a differential object.
    \end{enumerate}
\end{proposition}
\begin{proof}
\begin{enumerate}[label=(\alph*)]
\item 

Let $(\mathbb X , T , p , 0 , +, \ell)$ and $(\mathbb X' , T' , p' , 0' , +', \ell')$ be the tangent structures corresponding to the tangent actegories $(\mathbb X, T_\bullet)$ and $(\mathbb X', T_\bullet')$. Denote $\hat \alpha_W$ by $\alpha : F \circ T \Rightarrow T' \circ F$. 

To prove the statement, we simply need to check that $(F(X), F(\ast), F(!), F(\sigma), F(0), \hat{\alpha}_X \circ F(\lambda))$ satisfies Definition \ref{def:differential bundle} in the case where the tangent functor is given by $T'_W$. 
Finite pullback powers of $F(!): F(X) \to F(*)$ exist and the $k$-th pullback power equals $F(X^k)$ and $F = \T'_\mathbb N \circ F$ preserves pullbacks over $*$ by Definition \ref{def:differential_functors} .
In order to show that these pullbacks are preserved by the $n$-th power of the tangent functor $T' = \T'_W$ we need to show that
\[\begin{tikzcd}
	{T'^n(F( X^2))} & {T'^n(F(X))} \\
	{T'^n(F(X))} & {T'^n(F(*))}
	\arrow[from=1-2, to=2-2]
	\arrow[from=2-1, to=2-2]
	\arrow[from=1-1, to=2-1]
	\arrow[from=1-1, to=1-2]
	\arrow["\lrcorner"{anchor=center, pos=0.125}, draw=none, from=1-1, to=2-2]
\end{tikzcd}\]
is a pullback. This is the case as $T'^n \circ F=\T'_{W^{\otimes n}} \circ F$ preserves pullbacks over the terminal object.

The remainder of Properties 1-4 in Definition \ref{def:differential bundle} follows from functoriality.
In order to show Property 5, we need to show that
$(\hat \alpha_X \circ F(\lambda),0')$ is an additive bundle morphism of type
$$ 
(F(X) , F(*) , F(!) , F(\sigma) , F(\zeta) ) \to (T'(F(X)), T'(F(*)),T'(F(!)) ,  T'(F(\sigma)), T'(F(\zeta)) )
$$ 
and that $(\hat \alpha_X \circ F(\lambda), F(\zeta)) $ is an additive bundle morphism of type
$$
(F(X) , F(*) , F(!) , F(\sigma) , F(\zeta) ) \to (T'(F(X)), F(X), p'_{F(X)} , +'_{F(X)} , 0'_{F(X)}) .
$$
This means we need to show that the three diagrams of Definition \ref{def:additive_bundle_morphism} commute for each of these morphisms.  Each of them will follow from functoriality of $F$, naturality of $\alpha$ and the commuting diagrams in Definition \ref{def:lax_tangent_functor}.

To show that $(\alpha \circ F(\lambda),0')$ preserves the projection, we must show that $T'F(!)\circ (\alpha_X\circ F(\lambda)=0_{\mathbb X'}\circ F(!)$. We proceed as follows:
\begin{align*}
T'(F(!)) \circ \alpha \circ F(\lambda) &= \alpha \circ F(T(!)) \circ F(\lambda) 
\\&= \alpha \circ F(0) \circ F(!) 
\\&= 0' \circ F(!)
\end{align*}
To show that $(\alpha \circ F(\lambda),0')$ preserves the addition, we must show that $T'(F(\sigma)) \circ ( \alpha \circ F(\lambda) )_2 = \alpha \circ F(\lambda) \circ T'(F(\sigma))$. We proceed as follows:
\begin{align*}
T'(F(\sigma)) \circ \langle \alpha \circ F( \lambda) \circ \pi_0, \alpha \circ F(\lambda) \circ \pi_1 \rangle 
&= \alpha \circ F(T'(\sigma)) \circ \langle F( \lambda) \circ \pi_0, F(\lambda) \circ \pi_1 \rangle 
\\&= \alpha \circ F(T'(\sigma) \circ \langle  \lambda \circ \pi_0, \lambda \circ \pi_1 \rangle )
\\&= \alpha \circ F(\lambda) \circ F(\sigma)
\end{align*}
To show that $(\alpha \circ F(\lambda),0')$ preserves the zero section, we must show that $T'(F(\zeta)) \circ 0' =\alpha \circ F(\lambda) \circ F(\zeta) $. We proceed as follows:
\begin{align*}
T'(F(\zeta_X)) \circ 0'  & = T'(F(\zeta_X)) \circ \alpha \circ F(0_\mathbb X) 
\\ & = \alpha \circ F(T(\zeta_X)) \circ F(0_\mathbb X) 
\\ & = \alpha \circ F(\lambda_X) \circ F(\zeta_X) 
\end{align*}
To show that the additive bundle morphism $(\alpha \circ F(\lambda), F(\zeta))$ preserves the projection, we must show that $ p \circ \alpha \circ F(\lambda) =  F(\zeta) \circ F(!)$. We proceed as follows:
\begin{align*}
p' \circ \alpha \circ F(\lambda) & = F(p) \circ F(\lambda)
\\& = F(\zeta) \circ F(!) 
\end{align*}
To show that the additive bundle morphism $(\alpha \circ F(\lambda), F(\zeta))$ preserves the addition, we must show that
$+' \circ (\alpha \circ F(\lambda))_2= \alpha \circ F(\lambda) \circ F(\sigma) $. We proceed as follows:
\begin{align*}
+  \circ \langle \alpha \circ F(\lambda) \circ \pi_0,\alpha \circ  F(\lambda) \circ \pi_1  \rangle 
& = \alpha \circ F(+) \circ \langle F(\lambda) \circ \pi_0, F(\lambda) \circ \pi_1  \rangle 
\\&= \alpha \circ F(+ \circ \langle \lambda \circ \pi_0, \lambda \circ \pi_1  \rangle )
\\&= \alpha \circ F(\lambda) \circ F(\sigma)
\end{align*}
To show that the additive bundle morphism $(\alpha \circ F(\lambda), F(\zeta))$ preserves the zero section, we must show that
$0' \circ F(\zeta) = \alpha \circ F(\lambda) \circ F(\sigma) $. We proceed as follows:
The following equation shows that the additive bundle morphism $(\lambda, 0)$ preserves the zero section.
\begin{align*}
\alpha \circ F(\lambda) \circ F(\zeta) 
& = \alpha \circ F(0) \circ F(\zeta ) 
\\& = 0'_{F(X)} \circ F(\zeta)
\end{align*}
In order to show that the vertical lift $\alpha \circ F(\lambda)$ is compatible with $\ell'$, we must show that $T'(\alpha \circ F(\lambda)) \circ \alpha \circ F(\lambda) = \ell' \circ \alpha \circ F(\lambda)$. We proceed as follows:
\begin{align*}
T'(\alpha ) \circ T'(F(\lambda)) \circ \alpha \circ F(\lambda)  
&=  T'(\alpha ) \circ \alpha \circ F(T'(\lambda)) \circ F(\lambda) 
\\& = T'(\alpha ) \circ \alpha \circ F(\ell) \circ F(\lambda) 
\\&= \ell'  \circ \alpha \circ F(\lambda)
\end{align*}

Since $\alpha \circ F(0_X)=0_{F(X)}$ (which is part of Definition \ref{def:lax_tangent_functor} of a lax tangent functor) the map
$\mu' : F(X^2) \to T'(F(X))$ used in the universality of the vertical lift is
$$
\mu' 
=
T'(F(\sigma)) \circ 
\langle \alpha \circ F(0_X) \circ \pi_0,  \alpha \circ F(\lambda) \circ \pi_1 \rangle 
=
 \alpha \circ 
F(\sigma \circ \langle 0_X , \lambda \rangle ) = \alpha \circ F(\mu).
$$
Thus, Diagram \ref{diagram:universality} of Definition \ref{def:differential bundle}, the universality diagram is 
\[\begin{tikzcd}
	{F(X^2)} & {F(T(X))} & {T'(F(X))} \\
	{F(*)} & {F(T(*))} & {T'(F(*)).}
	\arrow["\alpha"', from=2-2, to=2-3]
	\arrow["\alpha", from=1-2, to=1-3]
	\arrow["{F(T(!))}"', from=1-2, to=2-2]
	\arrow["{T(F(!))}", from=1-3, to=2-3]
	\arrow["\lrcorner"{anchor=center, pos=0.125}, draw=none, from=1-2, to=2-3]
	\arrow["{F(\mu_X)}", from=1-1, to=1-2]
	\arrow["{F(!)}"', from=1-1, to=2-1]
	\arrow["{F(0)}"', from=2-1, to=2-2]
	\arrow["\lrcorner"{anchor=center, pos=0.125}, draw=none, from=1-1, to=2-2]
\end{tikzcd}\]
The left square is a pullback since $F$ preserves pullbacks over the terminal object. The right square is a pullback because it is the naturality square of $\alpha$.
\item By part (a) of this proof, $F(X)\to F(*)$ is a differential bundle. Since $F$ is a strong morphism, $F(*)=*'$ and $F(X)$ is a differential bundle over the terminal object which by Definition \ref{def:differential bundle}, is a differential object.
\end{enumerate}
\end{proof}

Now we show that the evaluation of a \nice{} functor sends additive morphisms of differential objects to additive morphisms of differential bundles. This defines the evaluation functor of Theorem \ref{thm:Eval_Dobj_DiffFun_DBun} on morphisms in the first component.
\begin{proposition}\label{prop:preserve_differential_structure}
Let $(F,\alpha): (\mathbb X, T) \to (\mathbb X',T')$ be a \nice{} functor, $(A,\zeta,\sigma,\lambda)$ and $(B,\zeta',\sigma', \lambda')$ be differential objects in $\mathbb X$ and $f: A \to B$ be an additive morphism of differential objects. Then 
$$
(F(f),1_{F(*)}): (F(A),F(*),F(!),F(\zeta),F(\sigma),\alpha \circ F(\lambda)) \to (F(B),F(*),F(!),F(\zeta'),F(\sigma'),\alpha \circ F(\lambda'))
$$ 
is an additive morphism of differential bundles.
\end{proposition}
\begin{proof}
Due to Definition \ref{def:differential_bundle_morphisms}, showing that $(F(f),1_{F(*)})$ is an additive morphism of differential bundles amounts to showing that 
\[\begin{tikzcd}
	{F(A)} & {F(B)} & {F(A_2)} & {F(B_2)} & {F(*)} & {F(*)} \\
	{F(*)} & {F(*)} & {F(A)} & {F(B)} & {F(A)} & {F(B)}
	\arrow["{F(f)}", from=1-1, to=1-2]
	\arrow["{F(!)}"', from=1-1, to=2-1]
	\arrow["{F(!)}", from=1-2, to=2-2]
	\arrow["{F(f_2)}", from=1-3, to=1-4]
	\arrow["{F(\sigma)}"', from=1-3, to=2-3]
	\arrow["{F(\sigma')}", from=1-4, to=2-4]
	\arrow["{1_{F(*)}}", from=1-5, to=1-6]
	\arrow["{F(\zeta)}"', from=1-5, to=2-5]
	\arrow["{F(\zeta)}", from=1-6, to=2-6]
	\arrow["{1_{F(*)}}"', from=2-1, to=2-2]
	\arrow["{F(f)}"', from=2-3, to=2-4]
	\arrow["{F(f)}"', from=2-5, to=2-6]
\end{tikzcd}\]
commute. They all commute by functoriality because they are images of the diagrams expressing that $f$ was a morphism of differential objects.
\end{proof}

The \WEIL-natural transformations from Definition \ref{def:actegory_linear_transformations} provide a notion of morphisms between differential functors.

The following proposition shows that \WEIL-natural transformations between \nice{} functors induce additive morphisms of differential bundles and linear \WEIL-natural transformations between \nice{} functors induce linear morphisms of differential bundles. Thereby it shows that the evalutation functor of Theorem \ref{thm:Eval_Dobj_DiffFun_DBun} is well defined on morphisms in the second component.

\begin{proposition}\label{prop:eval_on_morphisms} 
Let $(\mathbb X, T), (\mathbb X', T')$ be tangent categories and let  $(F,\alpha),(F',\alpha') : (\mathbb X, T) \to (\mathbb X', T')$ be differential functors.  
\begin{enumerate}[label = (\alph*)]
    \item If $A$ is a differential object in $\mathbb X$ and  $\varphi:F\Rightarrow F'$ is a natural transformation of \nice{} functors, then the pair $(\varphi_A, \varphi_*)$ consisting of $\varphi_A:F(A)\to F'(A)$ and $\varphi_*:F(*)\to F'(*)$ is an additive morphism of differential bundles.
    \item If $\varphi:F\Rightarrow F'$ is a linear natural transformation of \nice{} functors, then the pair $(\varphi_A, \varphi_*)$ consisting of $\varphi_A$ and $\varphi_*$ is a linear morphisms of differential bundles.
\end{enumerate}
\end{proposition}
\begin{proof}
\begin{enumerate}[label = (\alph*)]
\item A morphism between \nice{} functors is exactly a natural transformation $\varphi: F \Rightarrow F'$. The pair of maps
$(\varphi_{A} , \varphi_{*})$ forms an additive morphism of differential bundles as in Definition \ref{def:differential_bundle_morphisms} since they commute with the images of $\sigma_A, 0_A $ and $q_A$ by naturality.
\item In order to show that the morphism $(\varphi_{A}, \varphi_{*})$ of differential bundles is linear, we need to show that 
\[\begin{tikzcd}
	{TF(A)} & {TF'(A)} \\
	{F(A)} & {F'(A)}
	\arrow["{\varphi_A}"', from=2-1, to=2-2]
	\arrow["{T ( \varphi_A )}", from=1-1, to=1-2]
	\arrow["\lambda", from=2-1, to=1-1]
	\arrow["{\lambda'}"', from=2-2, to=1-2]
\end{tikzcd}
\]
commutes.
This diagram factors as
\[\begin{tikzcd}
	{T(F(A))} & {T(F'(A))}\\
	{F(T(A))} & {F'(T(A))} \\
    {F(A)} & {F'(A)} 
	\arrow["{\varphi_{T (A)}}", from=2-1, to=2-2]
	\arrow["{T ( \varphi_{(A)})}", from=1-1, to=1-2]
	\arrow["{F(\lambda_A)}", from=3-1, to=2-1]
	\arrow["{F'(\lambda_A)}"', from=3-2, to=2-2]
	\arrow["\alpha_W", from=2-1, to=1-1]
	\arrow["{\alpha_W'}"', from=2-2, to=1-2]
	\arrow["{\varphi_{A}}"', from=3-1, to=3-2]
\end{tikzcd}\]
where the lower part commutes by naturality of $\varphi$ and the upper part commutes if $\varphi$ as linear.
\end{enumerate}
\end{proof}

The different flavors of functors and natural transformations in Propositions \ref{prop:Dobj_to_Dbun} and \ref{prop:eval_on_morphisms} define four full subcategories of the hom-categories $\underline \Hom (\mathbb X, \mathbb X')$ in the four variants of $\weil\aACT^\bullet_{\bullet,t}$.

\begin{definition}\label{def:DiffFun} 
Let $(\mathbb X, T_\bullet)$ and $(\mathbb X', T'_\bullet)$ be tangent actegories.
\begin{enumerate}[label = (\alph*)]
\item The category $\DiffFun(\mathbb X, \mathbb X')$ has as objects lax \nice{} functors from $\mathbb X$ to $\mathbb X'$ and the morphisms between them are \weil-natural transformations as in Definition \ref{def:actegory_linear_transformations}.
\item The category $\DiffFun^\strong(\mathbb X, \mathbb X')$ has as objects strong \nice{} functors from $\mathbb X$ to $\mathbb X'$ and the morphisms between them are \weil-natural transformations.
\item The category $\DiffFun_\lin(\mathbb X, \mathbb X')$ has as objects lax \nice{} functors from $\mathbb X$ to $\mathbb X'$ and the morphisms between them are linear \weil-natural transformations as in Definition \ref{def:actegory_linear_transformations}.
\item The category $\DiffFun^\strong_\lin(\mathbb X, \mathbb X')$ has as objects strong \nice{} functors from $\mathbb X$ to $\mathbb X'$ and the morphisms between them are linear \weil-natural transformations.
\end{enumerate}
\end{definition}
Next we will define a functor $\Eval$ that evaluates differential functors on differential objects. The key insight here is that one does not just get any object, but an object together with the structure that makes it a differential bundle. For strong \nice{} functors one even obtains differential objects.

This means strong \nice{} functors preserve differential objects, lax \nice{} functors almost preserve them by sending them to differential bundles. In addition they preserve additive and linear maps between differential objects (or send them to additive and linear maps of differential bundles). Thus strong \nice{} functors preserve the category of differential objects, lax \nice{} functors send it into the category of differential bundles.

In order to capture this more precisely we formulate this result as an evaluation functor. 

Let $\eval: \mathbb X \times \Fun(\mathbb X, \mathbb X') \to \mathbb X'$ be the usual evaluation functor.
Let $u: \DObj(\mathbb X) \to \mathbb X$ be the functor that sends a differential object to its underlying object. Let $U: \DiffFun_\bullet (\mathbb X, \mathbb X') \to \Fun  (\mathbb X, \mathbb X')$ be the functor that sends a \nice{} functor to its underlying functor. Let $t: \DBun(\mathbb X') \to \mathbb X'$ be the functor that sends a differential bundle $(E,M,q,\zeta,\sigma)$ to its total bundle space $E$. We will construct a functor \Eval{} so that \eval{} factors through $\DBun(\mathbb X')$ in the sense that the following diagrams commute.
\begin{equation}
\hspace{-.5cm}
\begin{tikzcd}
	{\DObj(\mathbb X) \times \DiffFun_\bullet (\mathbb X, \mathbb X')} & {\DBun(\mathbb X')} \\
	{\mathbb X \times \Fun(\mathbb X,\mathbb X')} & {\mathbb X'}
	\arrow["\Eval", from=1-1, to=1-2]
	\arrow["u \times U"', from=1-1, to=2-1]
	\arrow["t", from=1-2, to=2-2]
	\arrow["\eval"', from=2-1, to=2-2]
\end{tikzcd} \quad
\begin{tikzcd}
	{\DObj(\mathbb X) \times \DiffFun^\strong_\bullet (\mathbb X, \mathbb X')} & {\DObj(\mathbb X')} \\
	{\mathbb X \times \Fun(\mathbb X,\mathbb X')} & {\mathbb X'}
	\arrow["\Eval", from=1-1, to=1-2]
	\arrow["u \times U"', from=1-1, to=2-1]
	\arrow["u", from=1-2, to=2-2]
	\arrow["\eval"', from=2-1, to=2-2]
\end{tikzcd}
\label{diagram:Eval_functor}
\end{equation}

\begin{definition}\label{def:Eval}
Let $(\mathbb X,T,p,0,+,\ell,c)$ and $(\mathbb X',T',p',0',+',\ell',c')$ be tangent categories. Define the assignment of objects and morphisms $\Eval: \DObj(\mathbb X) \times \DiffFun_\bullet (\mathbb X, \mathbb X') \to \DBun(\mathbb X')$ by the following rule:
\begin{itemize}
    \item On objects, $\Eval$ sends a differential object $A$ and a \nice{} functor $F:\mathbb X \to \mathbb X'$, to the differential bundle in $\mathbb X'$ that is given by $(F(A), F(*),F(q),F(\sigma),F(0),\alpha \circ F(\lambda))$.
    \item On morphisms $\Eval$ sends an additive differential object morphism $(f,1_*): A \to A'$ and a natural transformation $\varphi: F \Rightarrow F'$ to the differential bundle morphism $(\varphi_{A'} \circ F(f), \varphi_*)$.
\end{itemize} 
\end{definition}
Proposition \ref{prop:Dobj_to_Dbun} shows that $\Eval$ is well defined because $(F(A), F(*),F(q),F(\sigma),F(0),\alpha \circ F(\lambda))$ is indeed a differential bundle. The pair $(\varphi_{A'} \circ F(f), \varphi_*)$ is a differential bundle morphism because it is the composition of $(\varphi_{A'},\varphi_*)$ with $(F(f),1_{F(*)})$, both of which are differential bundle morphisms by Propositions 4.6 and 4.7.

\begin{theorem}\label{thm:Eval_Dobj_DiffFun_DBun}
Let $(\mathbb X,T,p,0,+,\ell,c)$ and $(\mathbb X',T',p',0',+',\ell',c')$ be tangent categories. Then the assignment
$$
\Eval : \DObj_\add(\mathbb X) \times \DiffFun_\bullet (\mathbb X, \mathbb X') \to \DBun_{\add}(\mathbb X')
$$
as defined in Definition \ref{def:Eval} is a functor and the diagrams (\ref{diagram:Eval_functor}) commute.
\end{theorem}
\begin{proof}
The assignments of objects and morphisms are given by Propositions \ref{prop:Dobj_to_Dbun}, \ref{prop:preserve_differential_structure} and \ref{prop:eval_on_morphisms}. Thus it only remains to show that $\Eval$ preserves identity and compositions.
By its definition on morphisms from the usual evaluation, $\Eval$ sends the identity to the identity.
For composition let $A\xrightarrow{(f,1_*)}B \xrightarrow{(g,1_*)}C$ be additive morphisms of differential objects and let $F \overset{\varphi}{\Rightarrow} G \overset{\psi}{\Rightarrow} H$ be natural transformations between \nice{} functors.

Then  
$$
\Eval(g,\psi) \circ \Eval (f, \varphi) = (\psi_C,\psi_*)\circ(G(g),1_{G(*)})\circ (\varphi_{B},\varphi_*)\circ (F(f),1_{F(*)}).
$$
Due to naturality of $\varphi$ this is equal to
$$
(\psi_C,\psi_*)\circ (\varphi_{B},\varphi_*)\circ(F(g),1_{F(*)})\circ (F(f),1_{F(*)}) = (\psi_C \circ \varphi_B,\psi_* \circ \varphi_*) \circ (F(g)\circ F(f), 1_{F(*)})
$$
which equals  $\Eval ((g,\psi) \circ (f, \varphi)$.

The diagrams (\ref{diagram:Eval_functor}) commute because the first component of $\Eval$ sends $A,F$ to the evaluation $F(A)$.
\end{proof}

The following Corollary summarizes the linear and strong cases of Theorem \ref{thm:Eval_Dobj_DiffFun_DBun} succinctly.  By a slight abuse of notation, we use $\Eval_A$ to denote all four cases.

\begin{corollary}\label{cor:eval_difffun_to_dbun_dobj}
Suppose $\mathbb X$ and $\mathbb X'$ are tangent categories and $A$ is a differential object in $\mathbb X$. Then there are four evaluation functors, %which all will be called $\Eval_A$, 
namely
\begin{enumerate}[label = (\alph*)]
\item $\Eval_A: \DiffFun(\mathbb X,\mathbb X') \to \DBun_\add(\mathbb X')$, 
\item $\Eval_A: \DiffFun_\lin(\mathbb X,\mathbb X') \to \DBun_\lin(\mathbb X')$,
\item $\Eval_A: \DiffFun^\strong(\mathbb X,\mathbb X') \to \DObj_\add(\mathbb X')$, and
\item $\Eval_A: \DiffFun^\strong_\lin(\mathbb X,\mathbb X') \to \DObj_\lin(\mathbb X')$.
\end{enumerate}
\end{corollary}
\begin{proof}
\item  
Case (a) is just the functor $\Eval$ of theorem \ref{thm:Eval_Dobj_DiffFun_DBun} evaluated on $A$ in the first component.
Cases (b), (c) and (d) follow from (a) in terms of functoriality as they are just a restriction to a subcategory. Proposition \ref{prop:preserve_differential_structure}.b and Proposition \ref{prop:eval_on_morphisms}.b ensure that the range is correct.
\end{proof}
A typical example for a category with a differential object is $\mathbb N^\bullet$ from Definition \ref{def:N_bullet}. 
\begin{example}\label{ex:given_functor_from_N_bullet_get_diffobj}
    Let $\mathbb N^\bullet$ be the category from Definition \ref{def:N_bullet}. Then any \nice{} functor $F$ from $\mathbb N^\bullet$ into a tangent category $\mathbb X$ induces a differential bundle in $\mathbb X$. Concretely the functor will produce a differential bundle with
    $$
    M = F(\mathbb N^0) , \qquad E = F(\mathbb N^1) , \qquad E_n = F(\mathbb N ^n).
    $$
\end{example}

\section{Inducing a \nice{} functor from $\mathbb N^\bullet$}\label{sec:induce}
In this section we focus in particular on the differential object $\mathbb N$ in $\mathbb N^\bullet$. The main construction is the functor $\Induce$ that is an inverse for $\Eval_{\mathbb N^\bullet}$ sending a differential bundle $(E,M,q,\zeta,\sigma, \lambda)$ in a tangent category $\mathbb X$ to a functor $\mathbb N^\bullet \to \mathbb X$ whose evaluation on $\mathbb N$ returns the differential bundle $(E,M,q,\zeta,\sigma, \lambda)$ as in Example \ref{ex:given_functor_from_N_bullet_get_diffobj}. This means $\mathbb N \to *$ is an initial differential bundle in the sense that there is a \nice{} functor from $\mathbb N \to *$ to every other differential bundle. In Section \ref{sec:equivalence} we will see in which sense this functor is unique.

Recall from Example \ref{ex:tan_structure_on_N_bullet} that the category $\mathbb N^\bullet$ from Definition \ref{def:N_bullet} forms a Cartesian tangent category (as defined in Example \ref{ex:trivial_bundle}) and thus a tangent actegory $(\mathbb N^\bullet,D_\bullet)$. We will not further use that the tangent structure is Cartesian. Objects of $(\mathbb N^\bullet,D_\bullet)$ are of the form $\mathbb N^k$ for $k\geq 0$, whose morphisms are matrices and whose tangent functor is $D_W(\mathbb N^k) = \mathbb N^k \times \mathbb N^k$.
As we saw in Example \ref{ex:N_bullet_differential}. the natural numbers $\mathbb N^1$ are a differential object in $\mathbb N^\bullet$.  

In the next lemma, we show that the set of bundle morphisms from the pullback powers of the bundle $E$ over $M$ to $E$ is a commutative monoid.  This result is important because it allows us to use linear combinations, resembling the matrix addition, as a tool to build morphisms between pullback powers.

The set of natural numbers $\mathbb N = \{0,1,2,...\}$ forms a semi-ring, and every commutative monoid $M$ can be seen as a module over the semi-ring $\mathbb N$ by defining $n\cdot x$ for $n \in \mathbb N$ and $x \in X$ as the $n$-fold addition of $x$ to itself, i.e. $n \cdot x = \sum_{i=1}^n x$.

The following Lemma holds for any additive bundle, though we are particularly interested for the case of differential bundles. Thus, unlike in Definition \ref{def:additive_bundle}, we use the notation $(E, M, q, \zeta, \sigma)$ instead of $(X,A,+,0,p)$, which resembles the notation of differential bundles.
\begin{lemma}\label{lem:module_structure_on_hom}
    Let $(E, M, q, \zeta, \sigma)$ be an additive bundle in a category $\mathbb X$.
    \begin{enumerate}
        \item For every $n \in \mathbb N$, the set 
        $$\Hom_{/M}(E_n,E):=\{ f : E_n \to E | q \circ f = q \circ \pi_0 \}
        $$ 
        together with the addition
        $$
        s: \Hom_{/M}(E_n,E) \times \Hom_{/M}(E_n,E) \to \Hom_{/M}(E_n,E) \quad , \quad (f,g) \mapsto \sigma \circ \langle f,g\rangle
        $$
        and the zero element
        $$
        z = \zeta \circ q \circ \pi_0 \in \Hom_{/M}(E_n,E)
        $$
        is a commutative monoid and thereby an $\mathbb N$-module 
        \item For $g:E_n \to E_m$ and $f:E_m \to E$, fulfilling $q \circ \pi_0  = q \circ  f$ and $q \circ \pi_0 = q \circ \pi_0 \circ g$,
        $$
        (n \cdot f) \circ g = n \cdot (f \circ g)
        $$
        and for morphisms $(f_i)_{1 \leq i \leq k}$,
        $$
        \left(\sum_{i=1}^k f_i \right) \circ g = \sum_{i=1}^k (f_i \circ g),
        $$
        where $\sum$ and $\cdot$ are from the $\mathbb N$-module structure defined in part 1. 
    \end{enumerate}
\end{lemma}
\begin{proof}
    \begin{enumerate}
        \item The definition of $s$ is well-defined because $f$ and $g$ fulfill $q \circ f = q \circ g = q \circ \pi_0$ and therefore $\langle f,g \rangle: E_n \to E_2$ is defined.

        The addition $s$ is commutative and associative, because $\sigma$ is. The element $z$ is the unit for $s$ because $\zeta$ is the unit for $\sigma$.
        \item Given a commutative monoid, the $\mathbb N$-module structure is defined as repeated addition, i.e. $n \cdot f := \sum_{i=1}^n f$. Therefore the second equation implies the first one. The second equation holds because
        $$
        \left(\sum_{i=1}^k f_i \right) \circ g = \sigma \circ \langle f_1 , ... , f_k \rangle \circ g = \sigma \circ \langle f_1 \circ g , ... , f_k \circ g \rangle = \sum_{i=1}^k (f_i \circ g).
        $$
        The  condition that $q \circ \pi_0  = q \circ  f$ and $q \circ \pi_0 = q \circ \pi_0 \circ g$ is required because the commutative monoid structure from Part 1 is only defined for morphisms in the slice category $\Hom_{/M}(E_n,E):=\{ f : E_n \to E | q \circ f = q \circ \pi_0 \}$.
    \end{enumerate}
\end{proof}
Using this module structure we now can define a functor $F_E$. 
\begin{definition}\label{def:induce_on_objects_part_1}
    Given a differential bundle $(E,M,q,\zeta,\sigma,\lambda)$ in a tangent category (formulated as a \weil-actegory) $(\mathbb X, \Te)$, then there is an assignment of objects and morphisms, $F_E: \mathbb N^\bullet \to \mathbb X$, defined 
    \begin{itemize}
        \item on objects as $F_E(\mathbb N^k) = E_k$, in particular $F_E(\mathbb N^0) = M$, and
        \item on morphisms (i.e. matrices) $A = (a_{ij})_{1 \leq i \leq n, 1 \leq j \leq m}: \mathbb N^m \to \mathbb N^n$ as 
        $$
        F_E(A) = \left\langle \sum_{j=1}^m a_{1j} \cdot \pi_j , ... , \sum_{j=1}^m a_{nj} \cdot \pi_j \right\rangle
        $$
        where $\sum$ and $\cdot$ are from the module structure in Lemma \ref{lem:module_structure_on_hom}.
    \end{itemize}
\end{definition}
\begin{proposition}\label{prop:induce_on_objects_part_1}
    The assignment $F_E$ of Definition \ref{def:induce_on_objects_part_1} is a functor $F_E: \mathbb N^\bullet \to \mathbb X$.
\end{proposition}
    The proof is essentially the same as the proof from classical linear algebra that composition of linear maps corresponds to matrix multiplication, just in a more general setting.
\begin{proof}
    In order to show that this assignment is a functor, we need to check that it preserves composition and identity.

    It preserves the identity because, for the identity matrix $\mathbb I_n: \mathbb N^n \to \mathbb N^n$, the functor evaluates as 
    $$
    F_E(\mathbb I_n) = \left\langle \sum_{j=1}^n (\mathbb I_n)_{1j} \cdot \pi_j , ... , \sum_{j=1}^n (\mathbb I_n)_{nj} \cdot \pi_j \right\rangle = \langle \pi_1 , ... , \pi_n \rangle = 1_{E_n}.
    $$

    In order to show that $F_E$ preserves composition, let $A=(a_{ij})_{1 \leq i \leq m , 1 \leq j \leq n} : \mathbb N^n \to \mathbb N^m$ and $B = (b_{ij})_{1 \leq i \leq n, 1 \leq j \leq k}: \mathbb N^k \to \mathbb N^n$ be composable morphisms of $\mathbb N^\bullet$. Then the composition of their images gives in the $h$-th component
    \begin{align*}
    (F_E(A) \circ F_E(B))_h &= \left( \sum_{i=1}^n a_{hi} \cdot \pi_i \right) \circ \left\langle \sum_{j=1}^k b_{1j} \cdot  \pi_j , ... , \sum_{j=1}^k b_{nj} \cdot  \pi_j  \right\rangle
    \\ & =   \sum_{i=1}^n \left(a_{hi}  \cdot \pi_i  \circ \left\langle \sum_{j=1}^k b_{1j} \cdot  \pi_j , ... , \sum_{j=1}^k b_{nj} \cdot  \pi_j  \right\rangle \right)
    \\
    &= \sum_{i=1}^n \left(a_{hi} \cdot   \sum_{j=1}^k b_{ij} \cdot  \pi_j \right)
    \\
    &= \sum_{j=1}^k \left( \sum_{i=1}^n a_{hi} b_{ij} \right) \pi_j = \sum_{j=1}^k (AB)_{hj} \pi_j = (F_E(AB))_h.
    \end{align*}
\end{proof}

In the following constructions we will construct a natural transformation $\hat \alpha: F_E \circ D_\bullet \to \Te \circ (1_\weil \times F_E)$ in order to make $(F_E, \hat \alpha)$ a \nice{} functor. We will define it first for the generators of $\mathbb N^\bullet$ and $\Weil$ and then construct the other components by pullback powers. For a differential bundle $E \xrightarrow{q} M$, the tangent bundle $T(E)$ has a projection $T(q):T(E) \to T(M)$ and a projection $p:T(E) \to E$, over both of which we take pullback powers. In order to keep track which base the pullback power is taken over, we will introduce a notation looking similar to fractions. This notation $\left(\frac{A}{B}\right)^k$ evokes the idea of products in the slice category $\mathbb X/B$, and is more compact and more informative than the somewhat standard notation $A\times_B \cdots \times_B A$ (which makes it hard to keep track of $k$).

\begin{proposition}\label{prop:fubini}~
Let $\mathbb X$ be a category with pullbacks and let 
\[\begin{tikzcd}
	A & B \\
	C & D
	\arrow["f", from=1-1, to=1-2]
	\arrow["\beta"', from=1-1, to=2-1]
	\arrow["\alpha", from=1-2, to=2-2]
	\arrow["g"', from=2-1, to=2-2]
\end{tikzcd}\]
be a commutative square in $\mathbb X$. Let $\left( \frac{A}{B} \right)^k := A \times_B ... \times_B A$ denote the $k$-th pullback power of $A \xrightarrow{f} B$ and let $\left( \frac{A}{C} \right)^k$, $\left( \frac{B}{D} \right)^k$ and $\left( \frac{C}{D} \right)^k$ be defined analogously. Let $\left( \frac{\left( \frac{A}{B} \right)^k}{\left( \frac{C}{D} \right)^k} \right)^m$ be the $m$-th pullback power of $\left( \frac{A}{B} \right)^k \xrightarrow{\left(\frac{\beta}{\alpha}\right)^k} \left( \frac{C}{D} \right)^k$.  Then 
$$
\left( \frac{\left( \frac{A}{B} \right)^k}{\left( \frac{C}{D} \right)^k} \right)^l \cong \left( \frac{\left( \frac{A}{C} \right)^l}{\left( \frac{B}{D} \right)^l} \right)^k .$$ In particular when $C=D$, 
$$ \left( \frac{\left( \frac{A}{B} \right)^k}{C} \right)^l \cong \left( \frac{\left( \frac{A}{C} \right)^l}{\left( \frac{B}{C} \right)^l} \right)^k .
$$
\end{proposition}
\begin{proof}
    This is the fact that we can commute limits, stated in our notation.
\end{proof}
Now we will combine this fraction calculus of pullbacks with the notation more common for differential bundles and tangent categories, where
$$
E_k = \left( \frac{E}{M} \right)^k \quad \text{and}\quad T_k(A) = \left( \frac{T(A)}{A} \right)^k.
$$
Recall that these pullbacks are preserved by the tangent structure due to Definition \ref{def:differential bundle}.
\begin{corollary}~
\begin{enumerate}[label = (\alph*)]
    \item Let $E \to M$ be a differential bundle in a tangent actegory $(\mathbb X, \Te)$. Then
    $$
    \Te(W^k,E_l) \cong \left( \frac{T_kE}{T_kM} \right)^l \cong \left( \frac{T(E_l)}{E_l} \right)^k.
    $$
    \item In particular for the differential bundle $\mathbb N^1 \to \mathbb N^0$ in $(\mathbb N^\bullet, D_\bullet)$
    $$
    D_\bullet(W^k, \mathbb N^l) \cong \left( \frac{\left( \frac{\mathbb N^2}{\mathbb N} \right)^k}{\mathbb N^0} \right)^l \cong \left( {\mathbb N^{2l}} \right)^k
    .
    $$
\end{enumerate}
\end{corollary}
\begin{proof}
    Part (a) is just Proposition \ref{prop:fubini} applied to the square
    \[\begin{tikzcd}
    	{T(E)} & E \\
    	{T(M)} & M
    	\arrow["{p_E}", from=1-1, to=1-2]
    	\arrow["{T(q)}"', from=1-1, to=2-1]
    	\arrow["q", from=1-2, to=2-2]
    	\arrow["{p_M}"', from=2-1, to=2-2]
    \end{tikzcd}\]
    that commutes by naturality of p. 
    Part (b) is part (a) for the differential bundle $\mathbb N^1 \to \mathbb N^0$ in $(\mathbb N^\bullet,D_\bullet)$.
\end{proof}

\begin{lemma}
    Let $(E,M,q,\zeta,\sigma,\lambda)$ be a differential bundle in a tangent actegory $(\mathbb X, \Te)$. Then the functor $F_E$ from Proposition \ref{prop:induce_on_objects_part_1} evaluates as
    $$
    F_E\left( \left( \frac{\mathbb N^l}{\mathbb N} \right)^k \right) \cong \left(\frac{E_l}{E} \right)^k \qquad \text{and} \qquad F_E \left( \frac{\mathbb N^{2l}}{\mathbb N^l} \right)^k \cong \left( \frac{E_{2l}}{E_l} \right)^k .
    $$
\end{lemma}
\begin{proof}
$$
F_E\left( \left( \frac{\mathbb N^l}{\mathbb N} \right)^k \right) = F_E(\mathbb N^{(l-1)\cdot k +1}) = F_E(E_{(l-1)\cdot k +1}) \cong 
$$
$$
F_E \left( \frac{\mathbb N^{2l}}{\mathbb N^l} \right)^k = F_E(\mathbb N^{(k+1)l}) = E_{(k+1)l}\cong \left( \frac{E_{2l}}{E_l} \right)^k
$$
\end{proof}

Given a differential bundle $(E,M,q)$ in a tangent actegory $(\mathbb X, T_\bullet)$, we now finally are able to construct the \nice{} functor 
$$
\Ind(E) = (F_E, \hat \alpha): \mathbb (\mathbb N^\bullet, D_\bullet) \to (\mathbb X, T_\bullet)
$$
that sends $\mathbb N$ to $E$. The underlying functor $F_E$ is the functor $F_E$ from Proposition \ref{prop:induce_on_objects_part_1}.

\begin{definition}\label{def:alpha_hat_for_induce}
Given a differential bundle $(E, M, q, \zeta, \sigma , \lambda)$ in a tangent actegory $(\mathbb X, T_\bullet)$, we define a set of morphisms
$$
\hat \alpha: F_E \circ D_A(\mathbb N^k) \to \T_A \circ F_E (\mathbb N^k)
$$
for every object $A$ of $\weil$ and $\mathbb N^k$ of $\mathbb N^\bullet$. We define it inductively out of the differential bundle structure of $(E,M,q, \zeta, \sigma, \lambda)$ by
\begin{align}
    \hat \alpha_{W  , \mathbb N^1} := T(\sigma) \circ \langle 0 \circ \pi_0 , \lambda \circ \pi_1 \rangle   &:  E_2 \to T(E), \label{eqn:alpha_on_W_N1}
    \\
    \hat \alpha_{W, \mathbb N^0} := 0 &:  M \to T(M) ,
    \\
    \hat \alpha_{\mathbb N, \mathbb N^1} := 1_E &:  E \to E
    \\
    \hat \alpha_{W^k ,\mathbb N^1} := \left( \frac{\hat \alpha_{W, \mathbb N^1}}{E}\right)^k  &:  F_E\circ D_{W^k}(\mathbb N^1) = \left(\frac{E_2}{E}\right)^k \to \left( \frac{TE}{E} \right)^k = \T_{W^k} \circ F_E(\mathbb N^1),
    \\
    \hat \alpha_{W^k ,\mathbb N^l} := \left( \frac{ \left(\frac{\hat \alpha_{W, \mathbb N}}{E}\right)^k}{M}\right)^l &: F_E \circ D_{W^k} (\mathbb N^{l})= \left( \frac{ \left(\frac{E_2}{E}\right)^k}{M}\right)^l \to \left( \frac{ \left(\frac{TE}{E}\right)^k}{M}\right)^l = \T_{W^k} \circ F_E(\mathbb N^{l}),
\end{align}
and 
\begin{equation}
    \hat \alpha_{A \otimes B, \mathbb N^l} := \T_A (\hat \alpha_{B,\mathbb N^l}) \circ \hat \alpha_{A,D_B(\mathbb N^l)}  : \quad F_E \circ D_{A} \circ D_{B} (\mathbb N^{l}) \to \T_{A} \circ \T_{B} \circ F_E (\mathbb N^{l}), \label{eqn:alpha_on_tensor}
\end{equation}
where $\left(\frac{A}{B} \right)^n$ denotes the $n$-th pullback power of $A$ over $B$. The expression $E_n = \left(\frac{E}{M} \right)^n$ continues to denote the $n$-th pullback power of $E$ over $M$.
\end{definition}

In Proposition \ref{prop:induce_on_objects} we show that this is a natural transformation $\hat \alpha: F_E \circ D_\bullet \to (\Te \circ 1_{\weil} \times F_E)$ and that $(F_E, \hat \alpha)$ is a \nice{} functor.
Before that, we now showcase how Equations \ref{eqn:alpha_on_W_N1}-\ref{eqn:alpha_on_tensor} determine $\hat \alpha$ for a nontrivial Weil algebra at the example of $W \otimes W$. The component of $\hat \alpha$ is 
    $$
    \hat \alpha_{W \otimes W, \mathbb N} = \T_W( \alpha_{W, \mathbb N}) \circ \hat \alpha_{W,D_W(\mathbb N)}.
    $$
    The parts unpack as
    \begin{align*}
    \hat \alpha_{W,D_W(\mathbb N)} &= \langle T(\sigma) \circ \langle 0 \circ \pi_0, \lambda \circ \pi_1 \rangle , T(\sigma) \circ \langle 0 \circ \pi_2, \lambda \circ \pi_3 \rangle \rangle : E_4 \to T(E_2) \cong \left( \frac{T(E)}{T(M)} \right)^2 \\
    T_W(\hat \alpha_{W, \mathbb N}) &= T\left(T(\sigma) \circ \langle 0 \circ \pi_0, \lambda \circ \pi_1 \rangle \right) : T(E_2) \to T\circ T(E)  = T^2(E)
    \end{align*}
    and the composition is 
    $$
    \hat \alpha_{W \otimes W, \mathbb N} = T^2(\sigma \circ (\sigma \times_M \sigma))  \circ \langle T(0) \circ 0 \circ \pi_0 , T(0) \circ \lambda \circ \pi_1 , T(\lambda) \circ 0 \circ \pi_2 , T(\lambda) \circ \lambda \circ \pi_3 \rangle.
    $$

\begin{proposition}\label{prop:induce_on_objects}
Given a differential bundle $E \xrightarrow{q} M$ in $\mathbb X$, Proposition \ref{prop:induce_on_objects_part_1} and Equations \ref{eqn:alpha_on_W_N1}-\ref{eqn:alpha_on_tensor} define a \nice{} functor 
$$
\Ind(E) = (F_E,\hat  \alpha): (\mathbb N^\bullet, D_\bullet) \to (\mathbb X, \Te).
$$

\end{proposition}

\begin{proof}
By Definition, $F_E$ preserves pullbacks over the terminal object.

The construction of $\hat \alpha$ defines a morphism in $\mathbb X$ for every \Weil{} algebra $A$ and every $\mathbb N^k$ in $\mathbb N^\bullet$ because every \Weil{} algebra can be expressed as a coproduct of product powers of $W$, and we have defined $\hat \alpha$ on those values.

To show that the transformation $\hat\alpha$ is natural, we show that the construction is natural with respect to the generating maps of $\mathbb N^\bullet$ as in Lemma 3.22.  This proof essentially follows from the fact that $\alpha$ is constructed out of these maps, and naturality with respect to them follows as a consequence.  However, checking this carefully is technical and tedious.  For this reason, we postpone full details of this proof until Appendix \ref{Appendix:naturality}.

The functor $F_E$ is constructed to preserve pullbacks over the terminal object and $\T^\mathbb X_A$ preserves them as $E$ is a differential object, thus the tangent structure commutes with these pullbacks. Next we need to show that $\hat \alpha$ is Cartesian, i.e. that for any Weil algebra $A$, the naturality diagrams of $\hat \alpha_{A,-}$ with respect to all morphisms in $\mathbb N^\bullet$ are pullback diagrams.

The naturality diagram of $\hat \alpha_{W^n,-}$ with respect to $\Delta_k: \mathbb N \to \mathbb N^k$,
\[
\begin{tikzcd}[column sep=large]
	{T(E_n)} & {T_n(E)} \\
	{T(E_{nk})} & {T_n(E_k)}
	\arrow["{\hat \alpha_{W^n,\mathbb N}}", from=1-1, to=1-2]
	\arrow["{T^n(\Delta_k)}", from=1-2, to=2-2]
	\arrow["{\hat \alpha_{W^n,\mathbb N^k}}"', from=2-1, to=2-2]
	\arrow["{T (\Delta_k)}"', from=1-1, to=2-1]
\end{tikzcd}
\]
is a pullback because in each component it is the $n$-th pullback power of 
\[\begin{tikzcd}[column sep=large]
	{T(E)} & {T_n(E)} \\
	{T(E_{k})} & {T_n(E_k)}
	\arrow["{\hat \alpha_{W^n,\mathbb N}}", from=1-1, to=1-2]
	\arrow["{T(\Delta_k)}", from=1-2, to=2-2]
	\arrow["{\hat \alpha_{W,\mathbb N^k}}"', from=2-1, to=2-2]
	\arrow["{T (\Delta_k)}"', from=1-1, to=2-1]
\end{tikzcd}
\qquad \text{over} \qquad 
\begin{tikzcd}[column sep=large]
	{T(M)} & E \\
	{T(M)} & {E_k}.
	\arrow["{\hat \alpha_{\mathbb N,\mathbb N}}", from=1-1, to=1-2]
	\arrow["{\Delta_k}", from=1-2, to=2-2]
	\arrow["{\hat \alpha_{W^n,\mathbb N^k}}"', from=2-1, to=2-2]
	\arrow["{1_{T(M)}}"', from=1-1, to=2-1]
\end{tikzcd}\]
Analogously, the diagram for $\sigma_k : \mathbb N^k \to \mathbb N$ is a pullback diagram.
Lemma \ref{lemma:morph_in_N_bullet} shows that all morphisms in $\mathbb N^\bullet$ are generated from $\sigma_k$ and $\Delta_k$. 
Thus the naturality diagrams of $\hat \alpha_{W^n,-}$ with respect to any morphisms in $\mathbb N^\bullet$ are pullbacks.

If the naturality diagrams of $\hat \alpha_A$ and $\hat \alpha_B$ are pullbacks, then the naturality diagrams of $\hat \alpha_{A \otimes B}$ are pullbacks beacuse both subsqaures are pullbacks:
\[\begin{tikzcd}
	{F_E \circ \T_A \circ \T_B(X)} & {\T_A \circ F_E \circ \T_B (X)} & {\T_A  \circ \T_B \circ F_E  (X)} \\
	{F_E \circ \T_A \circ \T_B(Y)} & {\T_A \circ F_E \circ \T_B (Y)} & {\T_A  \circ \T_B \circ F  (Y)}
	\arrow[from=1-1, to=2-1]
	\arrow[from=1-2, to=2-2]
	\arrow[from=1-3, to=2-3]
	\arrow["{\hat \alpha_{A, \T_B (X)}}", from=1-1, to=1-2]
	\arrow["{\hat \alpha_{A, \T_B (Y)}}"', from=2-1, to=2-2]
	\arrow["{\T_A(\hat \alpha_{B, X})}", from=1-2, to=1-3]
	\arrow["{\T_A(\hat \alpha_{B, X})}"', from=2-2, to=2-3]
\end{tikzcd}\]
The left square is a pullback by the assumption that the naturality diagrams of $\hat \alpha_A$ are pullbacks. The right square is a pullback since the naturality diagrams of $\hat \alpha_B$ are pullbacks built out of the universality diagram for the vertical lift. The universality diagram is a pullback preserved by the tangent functor, thus the right square is still a pullback.

Since every object of $\weil$ is a tensor product of powers of the dual numbers $W$, this shows that for any Weil algebra $A$, the naturality diagrams of $\hat \alpha_{A,-}$ are pullbacks, i.e. $\hat \alpha_{A,-}$ is Cartesian. This was all that remained to show that $(F_E, \hat \alpha)$ is a \nice{} functor. 
\end{proof}

\begin{theorem}\label{thm:induce_on_morphisms}
Let $\mathbb X$ be a tangent category and $(E,M,q,\zeta, \sigma, \lambda)$ and $(E',M',q',\zeta', \sigma', \lambda')$ be differential bundles in $\mathbb X$. Denote $\Ind(E) = (F_E, \hat \alpha)$ and $\Ind(E') = (F_E', \hat \alpha')$.
\begin{enumerate}[label = (\alph*)]
\item An additive morphism of differential bundles induces a natural transformation.
$\varphi : F_E \Rightarrow F_{E'}$
between the induced functors
\item A linear morphism of differential bundles induces a linear natural transformation.
\item This defines two functors 
$$
\Ind: \mathrm{\DBun_\add}(\mathbb X) \to \DiffFun (\mathbb N^\bullet, \mathbb X)
$$
and
$$
\Ind: \mathrm{\DBun_\lin}(\mathbb X) \to \DiffFun_\lin (\mathbb N^\bullet, \mathbb X).
$$
\end{enumerate}
\end{theorem}
\begin{proof}
\begin{enumerate}[label = (\alph*)]
    \item 
Suppose $(f,g): (E,M,p,\zeta, \sigma , \lambda) \to (E',M',p',\zeta', \sigma' , \lambda')$ is an additive morphism of differential bundles, i.e. a pair of maps $f: E \to E', g: M \to M'$ compatible with the projection, the addition and zero as in Definition \ref{def:differential_bundle_morphisms}. We define a natural transformation, $\varphi:F_E\Rightarrow F_{E'}$, by defining its $\mathbb N^k$ components for all $k \in \mathbb N$ as follows: 
$$
\varphi_{\mathbb N^k} = \left\lbrace
\begin{matrix}
g: M \to M' & \text{if }k=0 \\
f_k: E_k \to E'_k & \text{if }k>0. 
\end{matrix}\right.
$$
The transformation $\varphi$ %is natural as it 
is natural with respect to $+_k: \mathbb N^k \to \mathbb N$  (in particular for $k=0$, $0:\mathbb N^0 \to \mathbb N$)  and $\Delta_k: \mathbb N \to \mathbb N^k$ (in particular for $k=0$, $p:\mathbb N \to \mathbb N^0$). That is, the diagrams
\[\begin{tikzcd}
	{F_E(\mathbb N^k)} & {F_{E'}(\mathbb N^k)} & {F_E(\mathbb N^0)} & {M'} \\
	{F_E(\mathbb N)} & {F_{E'}} & {F_E(\mathbb N)} & {F_{E'}(\mathbb N)} \\
	{F_E(\mathbb N^1)} & {F_{E'}(\mathbb N^1)} & {F_E(\mathbb N^1)} & {F_{E'}(\mathbb N^1)} \\
	{F_E(\mathbb N^k)} & {F_{E'}(\mathbb N^k)} & {F_{E}(\mathbb N^0)} & {F_{E'}(\mathbb N^0)}
	\arrow["{\varphi_{\mathbb N^k}}", from=1-1, to=1-2]
	\arrow["{F_E(\sigma_k)}"', from=1-1, to=2-1]
	\arrow["{F_{E'}(\sigma_k)}", from=1-2, to=2-2]
	\arrow["g", from=1-3, to=1-4]
	\arrow["{F_{E}(\zeta)}"', from=1-3, to=2-3]
	\arrow["{F_{E'}(\zeta')}", from=1-4, to=2-4]
	\arrow["{\varphi_{\mathbb N^1}}"', from=2-1, to=2-2]
	\arrow["f"', from=2-3, to=2-4]
	\arrow["f", from=3-1, to=3-2]
	\arrow["{F_{E}(\Delta_k)}"', from=3-1, to=4-1]
	\arrow["{F_{E'}(\Delta_k)}", from=3-2, to=4-2]
	\arrow["f", from=3-3, to=3-4]
	\arrow["{F_{E}(q)}"', from=3-3, to=4-3]
	\arrow["{F_{E'}(q)}", from=3-4, to=4-4]
	\arrow["{f_k}"', from=4-1, to=4-2]
	\arrow["g"', from=4-3, to=4-4]
\end{tikzcd}\]
commute. Evaluating the functor $F_E$, we see that they are in fact
\[\begin{tikzcd}
	{E_k} & {E_k'} & M & {M'} & E & {E'} & E & {E'} \\
	E & {E'} & E & {E'} & {E_k} & {E_k'} & M & {M'}
	\arrow["g", from=1-3, to=1-4]
	\arrow["f"', from=2-3, to=2-4]
	\arrow["f"', from=2-1, to=2-2]
	\arrow["{f_k}", from=1-1, to=1-2]
	\arrow["g"', from=2-7, to=2-8]
	\arrow["f", from=1-7, to=1-8]
	\arrow["f", from=1-5, to=1-6]
	\arrow["{f_k}"', from=2-5, to=2-6]
	\arrow["{\sigma_k}"', from=1-1, to=2-1]
	\arrow["{\sigma_k'}", from=1-2, to=2-2]
	\arrow["\zeta"', from=1-3, to=2-3]
	\arrow["{\zeta'}", from=1-4, to=2-4]
	\arrow["{\Delta_k}"', from=1-5, to=2-5]
	\arrow["{\Delta_k}", from=1-6, to=2-6]
	\arrow["q"', from=1-7, to=2-7]
	\arrow["q'", from=1-8, to=2-8]
\end{tikzcd}.\] 
 The first diagram commutes because $(f,g)$ is additive and thus $\sigma$ and $\sigma'$ are compatible with $(f,g)$. The second diagram commutes because $(f,g)$ is additive and thus $\zeta$ and $\zeta'$ are compatible with $(f,g)$. The third diagram commutes by the definition of $f_k$. The fourth diagram commutes because $q$ and $q'$ are compatible with $(f,g)$.

The two  classes of morphisms $\Delta_k$ and $\sigma_k$ (including $\Delta_0 = q$ and $\sigma_0 = \zeta$) generate all morphisms in $\mathbb N^\bullet$  as their induced map into pullbacks and composition, according to Lemma \ref{lemma:morph_in_N_bullet}.

All morphisms in $\mathbb N^\bullet$ can be decomposed into morphisms whose components are composites of $\sigma_k$'s and $\Delta_k$'s using pullbacks and projections.  The pullbacks and projections are preserved by $F_E$ since $F_E$ is a \nice{} functor.  Therefore, $\varphi$ is natural with respect to all morphisms in $\mathbb N^\bullet$.

\item If $(f,g)$ is a linear morphism,  the diagram
\[\begin{tikzcd}[column sep=huge]
	{T(E)} & {T(E')} \\
	E & {E'}
	\arrow["{T(\varphi_{\mathbb N^1})=T(f)}", from=1-1, to=1-2]
	\arrow["{F_E(\lambda)=\lambda}", from=2-1, to=1-1]
	\arrow["{\varphi_{\mathbb N^1}=f}"', from=2-1, to=2-2]
	\arrow["{\lambda' = F_{E'}(\lambda)}"', from=2-2, to=1-2]
\end{tikzcd}\]
commutes by Definition \ref{def:differential_bundle_morphisms}. Thus $\varphi \circ \lambda = \lambda' \circ \varphi$.
Because $0$ is natural  with respect to $f$, the diagram
\[\begin{tikzcd}[column sep=huge]
	{T(E)} & {T(E')} \\
	E & E
	\arrow["{T(\varphi_{\mathbb N^1}) = T(f)}", from=1-1, to=1-2]
	\arrow["{0_E}", from=2-1, to=1-1]
	\arrow["{\varphi_{\mathbb N^1} = f}"', from=2-1, to=2-2]
	\arrow["{0_{E'}}"', from=2-2, to=1-2]
\end{tikzcd}\]
commutes. Recall from Definition \ref{def:alpha_hat_for_induce}, that the natural transformation $\hat \alpha: F \circ T_\bullet \Rightarrow T_\bullet \circ (1_ \weil \times F)$ is constructed using $\sigma , 0$ and $\lambda$ in Definition \ref{def:alpha_hat_for_induce}. The natural transformation $\varphi$ commutes with all parts of $\hat \alpha$ and therefore with $\hat \alpha$, e.g. for $W^2 \in \weil{}$
\[\begin{tikzcd}[column sep=large]
	{E_3} &&&& {T_2(E)} \\
	{E_3'} &&&& {T_2(E')}
	\arrow["{\alpha_{W^2,\mathbb N}=(T(\sigma) \circ \langle 0 \circ \pi_0, \lambda\circ \pi_1 \rangle)_2}", from=1-1, to=1-5]
	\arrow["{\alpha'_{W^2,\mathbb N}=(T(\sigma') \circ \langle 0 \circ \pi_0, \lambda'\circ \pi_1 \rangle)_2}"', from=2-1, to=2-5]
	\arrow["{\varphi_{\mathbb N^3}}"', from=1-1, to=2-1]
	\arrow["{T_2(\varphi_\mathbb N)}", from=1-5, to=2-5]
\end{tikzcd}\]
commutes.
This is what is required for it to be linear.
\item The assignment on objects and morphisms follows from (a), (b) and Proposition \ref{prop:induce_on_objects}. Functoriality holds, since taking the pullback of morphisms preserves compositions and the identity.
\end{enumerate}
\end{proof}
We will now see that $\varphi$ is uniquely defined as soon as we specify $\varphi_{\mathbb N^1} $ and $\varphi_{\mathbb N^0}$ as we did in the proof of Theorem \ref{thm:induce_on_morphisms}. This is important because we will use it to prove that \Eval{} from Theorem \ref{thm:Eval_Dobj_DiffFun_DBun} is faithful.
\begin{proposition}\label{prop:nat_trafos_generated_on}
Natural transformations between $\nice$ functors $(F,\alpha),(F',\alpha'): \mathbb N^\bullet \to \mathbb X$ are determined by their $\mathbb N^0$ and $\mathbb N^1$ components. 
\end{proposition}
\begin{proof}
Let $\varphi, \varphi' : (F,\alpha) \to (F',\alpha')$ be natural transformations between \nice{} functors such that $\varphi_{\mathbb N^1} = \varphi'_{\mathbb N^1}$ and $\varphi_{\mathbb N^0} = \varphi'_{\mathbb N^0}$. We will show by induction that for all $k\geq 0$, $\varphi_{\mathbb N^k} = \varphi'_{\mathbb N^k}$. The base cases hold by assumption. In the inductive step, due to naturality of $\varphi$
\[\begin{tikzcd}
	{F(\mathbb N)} & {F'(\mathbb N)} \\
	{F(\mathbb N^k)} & {F'(\mathbb N^k)} \\
	{F'(\mathbb N^{k-1})} & {F'(\mathbb N^{k-1})}
	\arrow["{\varphi_{\mathbb N}}", from=1-1, to=1-2]
	\arrow["{\varphi_{\mathbb N^k}}", from=2-1, to=2-2]
	\arrow["{\varphi_{\mathbb N^{k-1}}}", from=3-1, to=3-2]
	\arrow["{F(\pi_0)}", from=2-1, to=1-1]
	\arrow["{F'(\pi_0)}"', from=2-2, to=1-2]
	\arrow["{F'(\pi_0)}"', from=2-1, to=3-1]
	\arrow["{F'(\pi_1)}", from=2-2, to=3-2]
\end{tikzcd}\]
commutes. Replacing $\varphi$ with $\varphi'$ produces an analogous commuting diagram. Therefore, since $\varphi_{\mathbb N^1}=\varphi'_{\mathbb N^1}$ and  $\varphi_{\mathbb N^{k-1}} = \varphi'_{\mathbb N^{k-1}}$ by the inductive hypothesis, we see that $F'(\pi_0)  \circ \varphi_{\mathbb N^k} = F'(\pi_0) \circ \varphi'_{\mathbb N^k}$ and $F'(\pi_1) \circ \varphi_{\mathbb N^k}  = F'(\pi_1) \circ \varphi'_{\mathbb N^k} $.

Since $F'$ preserves pullbacks over the terminal object, $F'(\mathbb N^k)$ is a pullback and thus in
\[\begin{tikzcd}[column sep=huge]
	{F(\mathbb N^k)} \\
	& {F'(\mathbb N^k)} & {F'(\mathbb N^1)} \\
	& {F'(\mathbb N^{k-1})} & {F'(\mathbb N^0)}
	\arrow["{\varphi_{\mathbb N^k}}", shift left, from=1-1, to=2-2]
	\arrow["{\varphi'_{\mathbb N^k}}"', shift right, from=1-1, to=2-2]
	\arrow["{\varphi_{\mathbb N^1}}", curve={height=-12pt}, from=1-1, to=2-3]
	\arrow["{\varphi_{\mathbb N^{k-1}}}"', curve={height=12pt}, from=1-1, to=3-2]
	\arrow["{F(\pi_0)}", from=2-2, to=2-3]
	\arrow["{F(\pi_1)}"', from=2-2, to=3-2]
	\arrow["\lrcorner"{anchor=center, pos=0.125}, draw=none, from=2-2, to=3-3]
	\arrow[from=2-3, to=3-3]
	\arrow[from=3-2, to=3-3]
\end{tikzcd}\]
 the morphisms $\varphi_{\mathbb N^k}$ and $\varphi'_{\mathbb N^k}$ are determined by their compositions with $F'(\pi_0)$ and $F'(\pi_1)$. As these compositions coincide, the morphisms $\varphi_{\mathbb N^k}$ and $\varphi'_{\mathbb N^k}$ are equal. 
\end{proof}

\section{Establishing the equivalence}\label{sec:equivalence}

In the previous section we mentioned that \Induce{} is the inverse of $Eval_\mathbb N$. In this section we will prove how exactly they are inverse to each other. This will lead to an equivalence between categories of differential bundles and categories of lax tangent functors.

\begin{lemma}\label{lemma_composition_e_i}
Let $\mathbb X$ be a tangent category.
\begin{enumerate}[label = (\alph*)]
\item The composition of \Eval{} and \Induce{} is the identity functor: 
$$
\Eval \circ \Ind = 1_{\DBun(\mathbb X)}
$$
\item \Eval{} is essentially surjective.
\item \Eval{} is full.
\end{enumerate}
\end{lemma}
\begin{proof}
\begin{enumerate}[label = (\alph*)]
\item We will check the equality of functors on objects and morphisms.
Let $(E,M,q,\zeta,\sigma,\lambda)$ be an object, i.e. a differential bundle. Then $\Induce(E)$ is a functor that sends $\mathbb N^0$ to $M$, $\mathbb N^1$ to $E$, $q_\mathbb N$ to $q$, $\zeta_\mathbb N$ to $\zeta$, $\sigma_\mathbb N$ to $\sigma$ and $\lambda_\mathbb N$ to $\lambda$. Thus applying the functor $\Eval$ which is evaluation on $(\mathbb N, \mathbb N^0, q_\mathbb N , \zeta_\mathbb N, \sigma_\mathbb N, \lambda_\mathbb N)$ returns $(E,M,q,\zeta,\sigma, \lambda)$.

For an additive morphism $(f,g)$ of differential bundles, $\Induce(f,g)$ is a natural transformation $\varphi$ such that $\varphi_\mathbb N =f$ and $\varphi_{\mathbb N^0}=g$. Now \Eval{} is the functor that evaluates on the $\mathbb N$ and $\mathbb N^0$ components, thus the result is again $(f,g)$.

\item and (c) follow from (a). 

\end{enumerate}
\end{proof}
\begin{lemma}\label{lemma:faithful}
\Eval{} is faithful.
\end{lemma}
\begin{proof}
Let $\varphi$ and $\varphi'$ be natural transformations between \nice{} functors such that $\Eval (\varphi) = \Eval (\varphi')$. This means that $\varphi_\mathbb N = \varphi'_\mathbb N$ and $\varphi_{\mathbb N^0} = \varphi_{\mathbb N^0}$. Now it follows from Proposition \ref{prop:nat_trafos_generated_on} that $\varphi = \varphi'$ which means that \Eval{} is faithful.
\end{proof}
\begin{theorem}\label{thm:equivalence_bundle_categories}
The functors \Eval{} and \Induce{} form equivalences
$$
\DBun_\add( \mathbb X) \cong \DiffFun(\mathbb N^\bullet, \mathbb X) \qquad \DBun_\lin( \mathbb X) \cong \DiffFun_\lin(\mathbb N^\bullet,\mathbb X)
$$
\end{theorem}
\begin{proof}
Lemma \ref{lemma_composition_e_i} and Lemma \ref{lemma:faithful} together show that Eval is essentially surjective, full and faithful. This shows there is an equivalence. 
Let $\Eval^{-1}$ be a pseudo-inverse of $\Eval$ (since \Eval{} is an equivalence, it exists).
Then
$$
\Induce \circ \Eval \cong \Eval^{-1} \circ \Eval \circ \Induce \circ \Eval = \Eval^{-1} \circ \Eval \cong 1_{\DiffFun(\mathbb N^\bullet, \mathbb X)}
$$
It follows that \Eval{} and \Induce{} are pseudo-inverses to each other.
\end{proof}
This finally is the characterization of Differential bundles as functors from $\mathbb N^\bullet$.
\begin{corollary}\label{cor:eval_ind_add_lin}
\Eval{} and \Induce{} also form equivalences
$$
\DObj_\add(\mathbb X) \cong \DiffFun^\strong(\mathbb N^\bullet, \mathbb X) \qquad \DObj_\lin(\mathbb X) \cong \DiffFun^\strong_\lin(\mathbb N^\bullet, \mathbb X)
$$
\end{corollary}
\begin{proof}
As $\DiffFun^\strong(\mathbb N^\bullet , \mathbb X)$ is the full subcategory of functors that preserve the terminal object, it is equivalent to the full subcategory of differential bundles over the terminal object, i.e. differential objects.
\end{proof}

This Corollary generalizes Proposition 5.9 from \cite{BauerBurkeChing}, which states that the category of differential objects with linear morphisms is equivalent to the category of \nice{} functors and linear natural transformations.  Corollary \ref{cor:eval_ind_add_lin} generalizes the correspondence to the case of additive morphisms as well, and Theorem \ref{thm:equivalence_bundle_categories} generalizes the correspondence to differential bundles.

\section{Concluding remarks}\label{sec:concluding_remarks}

Differential bundles appear in the tangent category literature in a number of places.  In some cases, the approach taken by other authors is different than the approach taken here.  We review these contributions, highlighting the intuition about how these different approaches reference the same structure.

\subsection{Relation to Cockett-Cruttwell differential bundles} In \cite{cockett2016diffbundles}, Cockett and Cruttwell characterize differential bundles as fibrations. Concretely, \cite[Proposition 5.12]{cockett2016diffbundles} states that a differential bundle $E \to M$ can be seen as a differential object in the fiber $\mathrm{bun}_D(\mathbb X)[M]$ over $M$ of the fibration $\mathrm{bun}_D(\mathbb X) \to \mathbb X$ where $\mathrm{bun}_D(\mathbb X)$ is a full subcategory of the arrow category. For the definition of this subcategory a transerse system and a display system are chosen, which is necessary to control the existence of certain pullbacks which are preserved by the tangent structure.  These include the tangent pullbacks, but are more general. This approach was further developed in \cite{lanfranchi2025tangentdisplaymaps}.

From this perspective, a differential bundle is a differential object in (i.e. a strong \nice{} functor from $\mathbb N ^\bullet$ into) a certain full subcategory of the arrow category. The intricacies about the pullbacks are specified in the display system in the target of this functor.

The difference between the approach of \cite{cockett2016diffbundles, lanfranchi2025tangentdisplaymaps} and our current approach is that the equivalence of Theorem \ref{thm:equivalence_bundle_categories} uses general, not strong, \nice{} functors that do not preserve the terminal object.

The existence of any necessary pullbacks comes from their existence in \WEIL, as well as the fact that Cartesian lineators have certain pullbacks associated to them
Preservation of pullbacks are part of the data of \nice{} functors, not part of the target category. This is a marked difference in perspective from that of \cite{cockett2016diffbundles, lanfranchi2025tangentdisplaymaps}.

For future work we propose to investigate the relation of our characterization with $\nice{}$ functors to display systems. In particular, one could ask, which \nice{} functors into a tangent category with a display system correspond to display differential bundles.
The key intricacy is that not every morphism in the image of $F_E: \mathbb N^\bullet \to \mathbb X$ needs to lie in the display system, but certain morphisms need to.

\subsection{Relation to Ching classification}

There is another approach to differential bundles in tangent categories in \cite{ching2024characterizationdifferentialbundlestangent}.
In \cite[Corollary 9]{ching2024characterizationdifferentialbundlestangent}, Ching defines differential bundles as triples of morphisms which correspond to $(q: E \to M,\zeta: M \to E, \lambda: E \to T(E))$ from Definition \ref{def:differential bundle}.  These must fulfill certain properties.  The additive structure of Ching's differential bundles comes from the natural transformation $\sigma$ in the tangent structure.  Ching proves that this definition of differential bundles is equivalent to the one in \cite{cockett2016diffbundles}, therefore ours is also equivalent to Ching's.  

The approach in \cite {ching2024characterizationdifferentialbundlestangent} is very axiomatic in nature, and does not use the structure of $\mathbb N^\bullet$ specifically. However, 
  the differential structure in $\mathbb N^\bullet$ is already given by the tangent structure together with product projections, as in Lemma \ref{lemma:morph_in_N_bullet}. Ching's approach to differential bundles is to  generalizes this property from $\mathbb N^\bullet$ to any arbitrary differential bundle axiomatically in any tangent category, rather than transporting this structure from $\mathbb N^\bullet$ using \nice{} functors.

\subsection{Relation to tangent infinity categories}

In current work in progress Arro and Ching further generalize differential bundles to tangent infinity categories. Tangent infinity categories were defined in \cite{BauerBurkeChing} generalizing the classification of tangent categories by Leung and Garner. The definition of differential bundles in tangent infinity categories is a generalization of Theorem \ref{thm:equivalence_bundle_categories}.   This result characterizes differential bundles as functors from a structure category. The structure category $\mathbb N^\bullet$ is generalized with a certain structure category $E_\infty$ of spans of finite sets. One advantage  our \WEIL-actegory approach over the axiomatic tangent category approaches of \cite{cockett2016diffbundles, lanfranchi2025tangentdisplaymaps, ching2024characterizationdifferentialbundlestangent} is that our use of \WEIL-actegories compares nicely to the way tangent infinity categories are constructed. We could have characterized differential bundles using \nice{} functors in the axiomatic tangent category setting rather than the \WEIL-actegory setting, but chose to use the Garner-Leung setting for tangent actegories exactly for this reason.

\subsection{Future work in connections}

In \cite{cockett2017connections} differential bundles were further used to describe connections, a key feature of geometry and mathematical physics.  This led to one of the approaches to describe dynamical systems in \cite{cockett2019diffeq} which rely on differential bundles in other ways. 

Classically, there are two different notions of connections which, under certain circumstances, are equivalent. A Koszul connection (also known as covariant derivative) is a differential operator sending a section of a vector bundle and tangent vector to the directional derivative of the section along the tangent vector. An Ehresmann connection is a horizontal sub-bundle of a fiber bundle. In general, in the definition of an Ehresmann connection, the fibre bundle is not required to be a vector bundle. In the special case that it is a vector bundle, there is an equivalence between Koszul connections and Ehresmann connections.

Both Koszul and Ehresmann connections on vector bundles were generalized in \cite{cockett2017connections} as vertical and horizontal connections on differential bundles which are the categorical generalization of vector bundles. Under certain circumstances they are equivalent, in general they are not though. We would like to investigate connections via the characterization of differential bundles in this paper, and hope that this might advance this line of inquiry by a functorial characterization of bundles with connections. This will allow us to understand the behavior of connections under many categorical constructions, like opposites and Kleisli categories.

\subsection{Future work in dynamical systems}
Another construction this characterization of differential bundles can be helpful with, are dynamical systems.

Classically, dynamical systems are defined in \cite{jazmeyers_dynamical} to consists of a notion of how things can be, called the states, and a notion of how things will change given how they are, called the dynamics. Often the dynamics is described by a set of differential equations.

In \cite{cockett2019diffeq}, Cockett, Cruttwell and Lemay define a dynamical system in a tangent category, generalizing systems of differential equations.

While, in general the solution of a dynamical system in \cite{cockett2019diffeq} only requires a curve object, they obtain more results (like the existence of an exponential function) when assuming a differential curve object. 

We would like to investigate differential curve objects and dynamical systems in general via the characterization of differential bundles in this paper. This would allow us to describe the behavior of dynamical systems under constructions like opposites and Kleisli categories. In particular we hope this will be useful to include categorical differentiation into the categorically combinatorial descriptions of large dynamical systems as in \cite{jazmeyers_dynamical} and thereby simplify the description of real world systems.

\appendix
\section{Naturality diagrams for \Induce(E)}\label{Appendix:naturality}
In this appendix we will check that for the functor $\Induce(E): \mathbb N^\bullet$ to $\mathbb X$ defined  in Proposition \ref{prop:induce_on_objects} with components $(F_E,\hat \alpha)$, the $\hat \alpha: F_E \circ \T_\mathbb {N^\bullet} \to \Te \circ (1  \times F_E)$ is natural. The functors $F_R \circ \T_\mathbb {N^\bullet}$ and  $\Te \circ (1 \times F_E)$ both go from $\weil \times \mathbb N^\bullet$ to $\mathbb X$.

The following Lemma is a summary of some of Leung's observations about generators of $\weil$ in \cite{Leung2017}.
\begin{lemma}\label{lemma:Leung_on_weil-morphs}\cite[Proposition 9.5]{Leung2017}
%\begin{enumerate}
%    \item 
In the category $\Weil$ every map is inductively generated by $p$, $0$, $+$, $\ell$, $c$ and the projections $\pi_i$ using compositions, coproducts and foundational pullbacks which are pullbacks of the form 
\[\begin{tikzcd}
	{A \otimes (B \times C)} & {A \otimes B} \\
	{A \otimes C} & {A.}
	\arrow[from=1-1, to=1-2]
	\arrow[from=1-1, to=2-1]
	\arrow["\lrcorner"{anchor=center, pos=0.125}, draw=none, from=1-1, to=2-2]
	\arrow[from=1-2, to=2-2]
	\arrow[from=2-1, to=2-2]
\end{tikzcd}\]
\end{lemma}

Recall from Definition \ref{def:alpha_hat_for_induce} that the transformation $\hat \alpha: F_E\circ T_{\mathbb N^\bullet} \to T \circ F_E$ is defined inductively from the $W,\mathbb N^1$ component:
$$
\hat \alpha_{W  , \mathbb N^1} := T(\sigma) \circ \langle 0 \circ \pi_0 , \lambda \circ \pi_1 \rangle : E_2 \to TE  \qquad \hat \alpha_{W, \mathbb N^k} := (\alpha_\mathbb N)_k: E_{2k} \to T(E)_k.
$$
This is the map in the diagram for the universality of the vertical lift in Definition \ref{def:differential bundle}. 
In particular for the zero-th power $\mathbb N^0$ this definition means $\hat \alpha_{W,\mathbb N^0} = 0_M : M \to T(M)$.
For a product $W^n$ of $W$s there, we define
$$
\hat \alpha_{W^n, \mathbb N^1} := (\hat \alpha_{W,\mathbb N^1})_n : (E_2)_n \to T_n(E) \qquad \hat \alpha_{W^n, \mathbb N^k} := (\hat \alpha_{W^n, \mathbb N^1})_k: (E_{2n})_k \to (T_n(E))_k
$$
In particular for the empty product $\hat \alpha_{\mathbb N, \mathbb N^k} := 1_{E_k}$. Now, for coproducts $A \otimes B \in \mathrm{Obj}(\weil)$, we define $\hat \alpha_{A \otimes B, -}$ through the commuting diagrams in definition \ref{def:lax_tangent_functor}, concretely through
$$
\hat \alpha_{A \otimes B , X} := \T_A \hat  \alpha_{B,X} \circ \alpha_{A  , \T_B X} : F \circ \T_A \circ \T_B (X) \to  \T_A \circ \T_B \circ F (X)
$$
and 
$$
\alpha_{\mathbb N, X} = 1_X : T^0(X) = X \to X = T^0(X).
$$
\begin{proposition}
This transformation $\alpha: F_E \circ \T_\mathbb {N^\bullet} \to \Te \circ (1  \times F_E)$ is a natural transformation between functors $\weil \times \mathbb N^\bullet \to \mathbb X$
\end{proposition}

Recall from Theorem \ref{thm:Leung} that the morphisms of \WEIL{} produce the tangent structures in $\mathbb X$ via the functor $T_\bullet$ and the equality $(\mathbb X , T, p, 0, +, \ell, c) = (\mathbb X, T_W, T_p, T_0, T_+, T_\ell, T_c)$. Thus the symbols $(p, 0, +, \ell, c)$ are used both for the morphisms in $\weil$ and the morphisms in $\mathbb X$ that come from the tangent structure.

Analogously $F_E$ sends the morphisms $(\zeta: \mathbb N^0 \to \mathbb N^1, \sigma: \mathbb N^2 \to \mathbb N^1)$ of $\mathbb N^\bullet$ to the additive bundle structure $(\zeta: M \to E, \sigma : E_2 \to E)$ underlying the differential bundle $(E,M,q, \zeta, \sigma , \lambda)$. Thus the symbols $\sigma$ and $\zeta$ are used both for the maps in $\mathbb N^\bullet$ and their image in $\mathbb X$.

\begin{proof}
We need to show that it is natural in $\mathbb N^\bullet$ and \weil. We will accomplish this by proving that it is natural in the morphisms generating $\mathbb N^\bullet$ and \weil.

First we show it is natural in $\mathbb N^\bullet$.

As $\alpha_{A,-}$ is built as a composition of pullbacks of $\alpha_{W,-}$ it suffices to check naturality of $\alpha_{W,-}$ in $\mathbb N^\bullet$.
Let $f: \mathbb N^k \to \mathbb N^n$ be a morphism in $\mathbb N^\bullet$. As shown in Lemma \ref{lemma:morph_in_N_bullet}, its components are given by 
$$
f_n = \sigma_{f_n^1 + ... + f_n^k} \circ (\Delta_{f_n^1} \times_M ... \times_M \Delta_{f_n^k}).
$$
Unpacking the definition of the diagonal as the induced map by the identity this can be rewritten into
$$
f_n = \sigma_{N_n} \circ \langle \pi_1 , ... , \pi_2 , ... , ..., \pi_k , ... \rangle
$$
where $N_n= f_n^1 + ... + f_n^k$ and the projection $\pi_i$ appears $f_n^i$ times.
For the zero-fold case, the zero-fold sum is the zero map $\sigma_0 = \zeta: \mathbb N^0 \to \mathbb N^1$ and the zero-fold diagonal is the unique map to the terminal object $\Delta_0 =!: \mathbb N^1 \to \mathbb N^0$ .

The map $f$ is given by a compositions of pullbacks of the addition and the projection over the terminal object $\mathbb N^\bullet$ and the functors $F$ and $\Te$ preserve these pullbacks. Thus we only need to show that $\alpha$ is natural with respect to the addition $\sigma: \mathbb N^2 \to \mathbb N^1$, the zero $\zeta: \mathbb N^0 \to \mathbb N^1$ (zero-fold addition), the pullback-projections $\pi_i: \mathbb N^k \to \mathbb N^1$ and the unique map $!: \mathbb N^1 \to \mathbb N^0$ into the terminal object (zero-fold diagonal).

Naturality for the addition $\sigma: \mathbb N^2 \to \mathbb N$ is given by the following square:
\[\begin{tikzcd}
	{F_E \circ D_W (\mathbb N^2)} && {F_E \circ D_W(\mathbb N)} \\
	{\T_W(F_E(\mathbb N^2))} && {\T_W(F_E(\mathbb N))}
	\arrow["{F_E \circ D_W(\sigma)}", from=1-1, to=1-3]
	\arrow["{\hat \alpha_{W,\mathbb N^2}}"', from=1-1, to=2-1]
	\arrow["{\hat \alpha_{W,\mathbb N}}", from=1-3, to=2-3]
	\arrow["{\T_W(F_E(\sigma))}"', from=2-1, to=2-3]
\end{tikzcd}\]
The composition along the left path through the diagram amounts to
$$
T(\sigma) \circ (T(\sigma) \circ (0 \times_{TM} \lambda) \times_{TM} T(\sigma) \circ (0 \times_{TM} \lambda ) )
$$
As $F \circ T_W(\sigma)$ is given by the pullback $\sigma \times_E \sigma = \langle \sigma \circ \langle \pi_0, \pi_2 \rangle , \sigma \circ \langle \pi_1 , \pi_2 \rangle \rangle$, the right path through the square amounts to
$$
T(\sigma ) \circ (0_E \times_{TM} \lambda) \circ \langle \sigma \circ \langle \pi_0 , \pi_2 \rangle , \sigma \circ \langle \pi_1 , \pi_3 \rangle \rangle.
$$
Using the naturality of $0$ and that $(\lambda , 0)$ is an additive bundle morphism this equals to
$$
\qquad =T(\sigma) \circ \langle T(\sigma ) \circ \langle 0 \circ \pi_0 , 0 \circ \pi_2 \rangle , T(\sigma) \circ \langle \lambda \circ \pi_1 , \lambda \circ \pi_3 \rangle \rangle .
$$
By commutativity and associativity of $\sigma$ this equals to the left path.

Naturality for the zero map $\zeta: \mathbb N^0 \to \mathbb N^1$ is given by the following square:
\[\begin{tikzcd}
	{F_E \circ D_{\bullet}(W,\mathbb N^0)} && {F_E \circ  D_\bullet(W, \mathbb N^1)} \\
	{T_\bullet(W,F_E(\mathbb N^0))} && {T_\bullet(W,F_E(\mathbb N^1))}
	\arrow["{F_E \circ \hat D_\bullet(1_W, \zeta)}", from=1-1, to=1-3]
	\arrow["{\hat \alpha_{W,\mathbb N^0}}"', from=1-1, to=2-1]
	\arrow["{\hat \alpha_{W,\mathbb N^1}}", from=1-3, to=2-3]
	\arrow["{\hat T(1_W,F_E(\zeta))}"', from=2-1, to=2-3]
\end{tikzcd}\]
Since $\alpha_{W, \mathbb N^0} = 0: M \to T(M)$, the composition along the left path through the diagram amounts to 
$
T(\zeta) \circ 0
$.
The composition along the right path is 
$$
T(\sigma) \circ (0 \times_{T(E)} \lambda) \circ (\zeta \times_E \zeta).
$$  
Since, $0 \circ \zeta = \lambda \circ \zeta$, this equals
$$
T(\sigma) \circ (0 \circ \zeta \times_{T(E)} 0 \circ \zeta)
$$
and since
$0: 1_\mathbb X \to T$ is a natural transformation, this equals
$$
T(\sigma) \circ (T(\zeta) \circ 0 \times_{T(E)} T(\zeta) \circ 0)
$$
which equals $T(\zeta) \circ 0$ because $\zeta$ is the zero for the addition $\sigma$.

Naturality for the projections $\pi_i: \mathbb N^k \to \mathbb N$ is given by the following square:
\[\begin{tikzcd}
	{F_E \circ D_W(\mathbb N^k)} && {F_E \circ D_W(\mathbb N)} \\
	{\T_W(F(\mathbb N^k))} && {\T_W(F(\mathbb N))}
	\arrow["{F_E\circ D_W( \pi_i)}", from=1-1, to=1-3]
	\arrow["{\hat \alpha_{W,\mathbb N^k}}"', from=1-1, to=2-1]
	\arrow["{\hat \alpha_{W,\mathbb N}}", from=1-3, to=2-3]
	\arrow["{\T_W(F(\pi_i))}"', from=2-1, to=2-3]
\end{tikzcd}\]
The composition along the left path is 
$
\pi_i \circ (T(\sigma) \circ (0 \times_{TM} \lambda))_k,
$
the composition along the right path is
$
 T(\sigma) \circ (0 \times \lambda) \circ \langle \pi_{2i} , \pi_{2i+1} \rangle
$
and they are equal as projecting to the sum of the $i$-th pair is the same as taking the sum of the pair in indices $2i$ and $2i+1$.

Naturality for the unique morphism to the terminal object, $!: \mathbb N^1 \to \mathbb N^0$ is given by the following square:
\[\begin{tikzcd}
	{F_E \circ D_{\bullet}(W,\mathbb N^1)} && {F_E \circ  D_\bullet(W, \mathbb N^0)} \\
	{T_\bullet(W,F_E(\mathbb N^1))} && {T_\bullet(W,F_E(\mathbb N^0))}
	\arrow["{F_E \circ \hat D_\bullet(1_W, !)}", from=1-1, to=1-3]
	\arrow["{\hat \alpha_{W,\mathbb N^1}}"', from=1-1, to=2-1]
	\arrow["{\hat \alpha_{W,\mathbb N^0}}", from=1-3, to=2-3]
	\arrow["{\hat T(1_W,F_E(!))}"', from=2-1, to=2-3]
\end{tikzcd}\]
The composition along the left path in the diagram gives
$$
T(q) \circ T(\sigma) \circ (0 \times_{T(E)} \lambda) 
$$
which equals $q \circ \pi_0 = q \circ \pi_1$ by the commutativity of the diagram for the universality of the vertical lift.
The composition along the right path in the diagram is
$$
T(\zeta) \circ 0 \circ \langle q , q \rangle
$$
where $\langle q , q \rangle =  q \circ \pi_0 = q \circ \pi_1$ is the morphism $E \times_M E \to M \times_M M = M$.

Since every morphism in $\mathbb N^\bullet$ is generated by $\sigma, \zeta, \pi_i$ and $!$, this concludes the proof that $\hat \alpha$ is natural in $\mathbb N^\bullet$.

Now we show $\hat \alpha$ is natural in \weil.
Since $\hat \alpha_{- , \mathbb N^k }$ is a pullback power of $\hat \alpha_{-,\mathbb N^1}$ we only need to check the naturality of $\hat \alpha_{-,\mathbb N^1}$.

By Lemma \ref{lemma:Leung_on_weil-morphs} (Proposition 9.5 in \cite{Leung2017}) every morphism in \weil{} can be constructed from $0,p, \ell , + , c$ using composition, tensors and foundational pullbacks.

We first prove that $\hat \alpha$ is natural with respect to the generating morphisms $0,p, \ell , + , c$ and then prove that the naturality diagrams are preserved under tensors and foundational pullbacks.

Naturality with respect to $0: \mathbb N \to \mathbb W$ amounts to the following square:
\[\begin{tikzcd}
	{F_E \circ D_\bullet(\mathbb N,\mathbb N)} && {F \circ D_\bullet(W, \mathbb N)} \\
	{\T_\bullet(\mathbb N,F_E(\mathbb N))} && {\T_\bullet(W,F(\mathbb N))}
	\arrow["{F_E\circ D_\bullet(0, 1_\mathbb N)}", from=1-1, to=1-3]
	\arrow["{\hat \alpha_{\mathbb N,\mathbb N}}"', from=1-1, to=2-1]
	\arrow["{\hat \alpha_{W,\mathbb N}}", from=1-3, to=2-3]
	\arrow["{\Te(0,1_{F_E(\mathbb N)})}"', from=2-1, to=2-3]
\end{tikzcd}\]
Since $\hat \alpha_{\mathbb N, \mathbb N}$ is the identity map (the empty pullback), the left path is just the zero-map $0: E \to T(E)$
The composition along the right path of the diagram is 
$$
T(\sigma)\circ (0 \times_{T(M)} \lambda)\circ \langle 1, \zeta \circ q \rangle.
$$
Since $(\lambda,0)$ is an additive bundle morphism from $(E,M)$ to $(T(E), T(M))$, this equals to
$$
T( \sigma ) \circ \langle 0 , T(\zeta) \circ 0 \circ q \rangle = 0
$$
which is the zero-map because $\zeta$ is the neutral element with respect to the addition $\sigma$.

Naturality with respect to $p:W \to \mathbb N$ amounts to the following square:
\[\begin{tikzcd}
	{F_E \circ D_\bullet(W,\mathbb N)} && {F_E \circ D_\bullet(\mathbb N, \mathbb N)} \\
	{\T_{\mathbb X}(W,F_E(\mathbb N))} && {\T_{\mathbb X}(\mathbb N,F_E(\mathbb N))}
	\arrow["{F_E\circ D_\bullet(p, 1_\mathbb N)}", from=1-1, to=1-3]
	\arrow["{\hat \alpha_{W,\mathbb N}}"', from=1-1, to=2-1]
	\arrow["{\hat \alpha_{\mathbb N,\mathbb N}=1_{F(\mathbb N)}}", from=1-3, to=2-3]
	\arrow["{\Te(p,1_{F_E(\mathbb N)})}"', from=2-1, to=2-3]
\end{tikzcd}\]
It commutes since the composition of the left path is $p \circ T(\sigma) \circ \circ (0 \times \lambda)$ which by naturality of $p$ euqals to $\sigma \circ (p \circ 0 \times p \circ \lambda)$. Since $(\lambda, \zeta)$ is an additive bundle morphism this equals to $\sigma \circ (1 \times \zeta \circ q)$. This equals $\pi_0$ because $\zeta$ is zero with respect to the addition $\sigma$.

Naturality with respect to $+:W^2 \to W$ amounts to the following square:
\[\begin{tikzcd}
	{F_E \circ D_\bullet(W^2,\mathbb N)} && {F_E \circ D_\bullet(W, \mathbb N)} \\
	{\T_{\mathbb X}(W^2,F_E(\mathbb N))} && {\T_{\mathbb X}(W,F_E(\mathbb N))}
	\arrow["{F_E\circ D_\bullet(+, 1_\mathbb N)}", from=1-1, to=1-3]
	\arrow["{(\hat \alpha_{W,\mathbb N})_2}"', from=1-1, to=2-1]
	\arrow["{\hat \alpha_{W,\mathbb N}}", from=1-3, to=2-3]
	\arrow["{\Te(+,1_{F_E(\mathbb N)})}"', from=2-1, to=2-3]
\end{tikzcd}\]
Because $\hat \alpha_{W,\mathbb N} = T(\sigma) \circ (0 \times_{TM} \lambda)$, the map $(\hat \alpha_{W,\mathbb N})_2 : E_3 \to T_2E$ is given by
$$
(\hat \alpha_{W,\mathbb N})_2 = \langle T(\sigma) \circ \langle 0 \circ \pi_0 , \lambda \circ \pi_1 \rangle , T(\sigma) \circ \langle 0 \circ \pi_0 , \lambda \circ \pi_2 \rangle \rangle .
$$
The composition along the left path gives
$$
+ \circ \langle T(\sigma ) \circ \langle 0 \circ \pi_0 , \lambda \circ \pi_1 \rangle , T( \sigma) \circ \langle 0 \circ \pi_0 , \lambda \circ \pi_2 \rangle \rangle.
$$
By naturality of $+$ this equals
\begin{align*}
&T(\sigma) \circ +_{E_2} \circ \langle \langle 0 \circ \pi_0 , \lambda \circ \pi_1 \rangle , \langle 0 \circ \pi_0 , \lambda \circ \pi_2 \rangle \rangle
\\=&
T(\sigma) \circ \langle + \circ \langle 0 \circ \pi_0 , 0 \circ \pi_0 \rangle , + \circ \langle \lambda \circ \pi_1, \lambda \circ \pi_2 \rangle \rangle.
\end{align*}
Since $0$ is the zero with respect to the addition $+$ and $(\lambda, \zeta)$ is an additive bundle morphism, this equals
$$
T(\sigma) \circ \langle 0 \circ \pi_0 , \lambda \circ \sigma \circ \langle \pi_1 , \pi_2 \rangle \rangle = T(\sigma) \circ (0 \times_{TM} \lambda) \circ (1_E \times_M \sigma)
$$
which is the composition along the right path in the square.

Naturality with respect to $\ell: W \to W \otimes W $ amounts to the following square:
\[\begin{tikzcd}
	{F_E \circ D_\bullet(W,\mathbb N)} && {F_E \circ D_\bullet(W \otimes W, \mathbb N)} \\
	{\T_{\mathbb X}(W,F_E(\mathbb N))} && {\T_{\mathbb X}(W \otimes W,F_E(\mathbb N))}
	\arrow["{F_E\circ D_\bullet(\ell, 1_\mathbb N)}", from=1-1, to=1-3]
	\arrow["{\hat \alpha_{W,\mathbb N}}"', from=1-1, to=2-1]
	\arrow["{\Te (1_W,\hat \alpha_{W,\mathbb N}) \circ \hat \alpha_{W, \mathbb N^2}}", from=1-3, to=2-3]
	\arrow["{\Te(\ell,1_{F_E(\mathbb N)})}"', from=2-1, to=2-3]
\end{tikzcd}\]
The composition along the right path gives
$$
T^2(\sigma) \circ (T(0) \times_{T^2M} T(\lambda)) \circ (T(\sigma) \circ (0 \times_{TM} \lambda) \times_{TM} T(\sigma) \circ (0 \times_{TM} \lambda)) \circ \langle \pi_0 , \zeta \circ q \circ \pi_0 , \zeta \circ q \circ \pi_1 , \pi_1 \rangle 
$$
which equals
$$
T^2( \sigma) \circ \langle T(0) \circ T(\sigma) \circ \langle 0 \circ \pi_0 , \lambda \circ \zeta \circ q \circ \pi_0 \rangle , T(\lambda) \circ T(\sigma) \circ \langle 0 \circ \zeta \circ q \circ \pi_1, \lambda \circ \pi_1 \rangle \rangle.
$$
In the last part of the composition we used the fact that $$
\zeta \circ q \circ \pi_1 = \zeta \circ q \circ \pi_0: E_2 \to E.
$$
Due to the additive bundle morphism properties $\lambda \circ \zeta = 0 \circ \zeta = T(\zeta) \circ 0$ and therefore the right path gives
$$
T^2(\sigma) \circ \langle T(0) \circ T(\sigma) \circ \langle 0 \circ \pi_0 , T(\zeta) \circ 0 \circ q \circ \pi_0 \rangle , T( \lambda) \circ T( \sigma) \circ \langle T( \zeta) \circ 0 \circ \pi_1 , \lambda \circ \pi_1 \rangle \rangle .
$$
Since $\zeta$ is the zero of the addition $\sigma$ this equals
$$
T^2(\sigma) \circ \langle T(0) \circ 0 \circ \pi_0 , T( \lambda) \circ \lambda \circ \pi_1 \rangle = \ell \circ T(\sigma) \circ \langle 0 \circ \pi_0 , \lambda \circ \pi_1 \rangle 
$$
which is the composition along the left path.

Naturality with respect to $c: W \otimes W \to W \otimes W$ amounts to the following square:
\[\begin{tikzcd}
	{F_E \circ D_\bullet(W \otimes W,\mathbb N)} && {F_E \circ D_\bullet(W \otimes W, \mathbb N)} \\
	{\T_{\mathbb X}(W \otimes W,F_E(\mathbb N))} && {\T_{\mathbb X}(W \otimes W,F_E(\mathbb N))}
	\arrow["{F_E\circ D_\bullet(c, 1_\mathbb N)}", from=1-1, to=1-3]
	\arrow["{\Te (1_W,\hat \alpha_{W,\mathbb N}) \circ \hat \alpha_{W, \mathbb N^2}}"', from=1-1, to=2-1]
	\arrow["{\Te (1_W,\hat \alpha_{W,\mathbb N}) \circ \hat \alpha_{W, \mathbb N^2}}", from=1-3, to=2-3]
	\arrow["{\Te(c,1_{F_E(\mathbb N)})}"', from=2-1, to=2-3]
\end{tikzcd}\]
The composition along the left path gives
\begin{align*}
&c \circ T^2(\sigma) \circ (T(0) \times_{T^2M} T(\lambda)) \circ (T(\sigma) \circ (0 \times_{TM} \lambda) \times_{TM} T(\sigma) \circ (0 \times_{TM} \lambda))
\\=&
c \circ T^2( \sigma) \circ \langle T(0) \circ T(\sigma) \circ \langle 0 \circ \pi_0 , \lambda \circ \pi_1 \rangle , T(\lambda) \circ T( \sigma ) \circ \langle 0 \circ \pi_2 , \lambda \circ \pi_3 \rangle \rangle .
\end{align*}
Due to naturality of the canonical flip of the tangent structure $c: T^2 \Rightarrow T^2$, since $c \circ T(0_E) = 0_{TE}$ and since $(\lambda, \zeta)$ is an additive bundle morphism this equals to
$$
T^2(\sigma) \circ \langle 0_{TE} \circ T(\sigma) \circ \langle 0 \circ \pi_0 , \lambda \circ \pi_1 \rangle , c \circ T^2(\sigma) \circ \langle T(\lambda) \circ 0 \circ \pi_2 , T(\lambda) \circ \lambda \circ \pi_3 \rangle \rangle .
$$
By naturality of $c$ and since $c \circ T(\lambda ) \circ \lambda = c \circ \ell \circ \lambda = \ell \circ \lambda = T(\lambda) \circ \lambda$ this equals to
$$
T^2(\sigma) \circ \langle T^2(\sigma) \circ \langle 0_{TE} \circ 0 \circ \pi_0 , 0_{TE} \circ \lambda \circ \pi_1 \rangle , T^2( \sigma) \circ \langle c \circ T(\lambda) \circ 0 \circ \pi_2 , T( \lambda) \circ \lambda \circ \pi_3 \rangle \rangle.
$$
By naturality of zero and since $c\circ 0_{TE} = T(0_E)$ this equals to
$$
T^2(\sigma) \circ \langle T^2(\sigma) \circ \langle T(0) \circ 0 \circ \pi_0, T(\lambda) \circ 0 \circ \pi_1 \rangle , T^2(\sigma) \circ \langle
T(0) \circ \lambda \circ \pi_2 , T(\lambda) \circ \lambda \circ \pi_3
\rangle \rangle.
$$
Now using associativity and commutativity of $\sigma$ we can rearrange the terms to obtain
$$
T^2(\sigma) \circ \langle T^2(\sigma) \circ \langle T(0) \circ 0 \circ \pi_0 , T(0) \circ \lambda \circ \pi_2 \rangle , T^2(\sigma) \circ \langle T(\lambda) \circ 0 \circ \pi_1 , T(\lambda) \circ \lambda \circ \pi_3 \rangle \rangle .
$$
By naturality of 0 and since $(\lambda, 0)$ is an additive bundle morphism, this equals to
\begin{align*}
&T^2(\sigma) \circ \langle T(0) \circ T(\sigma) \circ \langle 0 \circ \pi_0 , \lambda \circ \pi_2 \rangle , T(\lambda) \circ T(\sigma) \circ \langle 0 \circ \pi_1 , \lambda \circ \pi_3 \rangle \rangle
\\=&
T^2(\sigma) \circ ( T(0) \times_{T^2M} T(\lambda)) \circ 
(
T(\sigma) \circ (0 \times_{TM} \lambda) \times_{TM} T(\sigma) \circ (0 \times_{TM} \lambda)
)
\circ \langle \pi_0, \pi_2 , \pi_1 , \pi_3 \rangle
\end{align*}
which is the composition along the right path in the square.

This shows that $\hat \alpha$ is natural with respect to the building blocks $p,+,0, \ell, c$ and projections out of products. In order to show that $\hat \alpha$ is natural with respect to all morphisms in $\weil$, we show next that $\hat \alpha$ is compatible with coproducts and pullbacks of the form $A \otimes (B \times C)$ from Lemma \ref{lemma:Leung_on_weil-morphs} (which include products).

For tensor products $f_1 \otimes f_2: A \otimes B \to A' \otimes B'$ in \weil, $\hat \alpha$ is natural since $\hat \alpha_{A \otimes B}$ is defined as $\T_A (\hat \alpha_{B}) \circ \hat \alpha_{A}$.

For pullbacks of the form $A \otimes (B\times C)$, suppose $\hat \alpha_{-,\mathbb N^k}$ is natural with respect to the $\weil$-morphisms $f_1: K \to A \otimes B$ and $f_2: K \to A \otimes C$. 

From the tangent category axioms, the following diagram is a pullback in the tangent category $(\mathbb X,\Te)$:
\[\begin{tikzcd}
	{\T_A \circ \T_{B \times C} \circ F_E} && {\T_A \circ \T_{C} \circ F_E} \\
	{\T_A \circ \T_{B} \circ F_E} && {\T_A \circ F_E}
	\arrow["{\T_A \circ \T_{\pi_1} \circ F_E }", from=1-1, to=1-3]
	\arrow["{\T_A \circ \T_{\pi_0} \circ F_E }"', from=1-1, to=2-1]
	\arrow["{\T_A \circ \T_! \circ F_E }", from=1-3, to=2-3]
	\arrow[""{name=0, anchor=center, inner sep=0}, "{\T_A \hat \circ \T_! \circ F }"', from=2-1, to=2-3]
	\arrow["\lrcorner"{anchor=center, pos=0.125}, draw=none, from=1-1, to=0]
\end{tikzcd}\]
Therefore, due to the universal property of the pullback, the naturality diagram 
\[\begin{tikzcd}[column sep=large]
	{F_E \circ D_K} &&& {F_E \circ D_{A \otimes (B \times C)}} \\
	{T_K \circ F_E} &&& {T_{A \otimes (B \times C)} \circ F_E}
	\arrow["{F_E \circ D_{\langle f_1, f_2\rangle}=F_E \circ \langle D_{f_1} , D_{f_2}\rangle}", from=1-1, to=1-4]
	\arrow["{\hat \alpha_K}"', from=1-1, to=2-1]
	\arrow["{\hat \alpha_{A \otimes (B \times C)}}", from=1-4, to=2-4]
	\arrow["{T_{\langle f_1, f_2\rangle}\circ F_E = \langle T_{f_1}, T_{f_2}\rangle \circ F_E}"', from=2-1, to=2-4]
\end{tikzcd}\]
commutes if and only if the naturality diagrams
\[\begin{tikzcd}
	{F_E \circ D_K} && {F_E \circ D_{A \otimes B}} && {F_E \circ D_K} && {F_E \circ D_{A \otimes C}} \\
	{T_K \circ F_E} && {T_{A \otimes B} \circ F_E} && {T_K \circ F_E} && {T_{A \otimes C} \circ F_E}
	\arrow["{F_E \circ D_{ f_1}}", from=1-1, to=1-3]
	\arrow["{\hat \alpha_K}"', from=1-1, to=2-1]
	\arrow["{\hat \alpha_{A \otimes B}}", from=1-3, to=2-3]
	\arrow["{F_E \circ D_{ f_2}}"', from=1-5, to=1-7]
	\arrow["{\hat \alpha_K}"', from=1-5, to=2-5]
	\arrow["{\hat \alpha_{A \otimes C}}", from=1-7, to=2-7]
	\arrow["{T_{f_1}\circ F_E }"', from=2-1, to=2-3]
	\arrow["{T_{f_2}\circ F_E }", from=2-5, to=2-7]
\end{tikzcd}\]
for the components $f_1$ and $f_2$ commute (which we assumed).

This concludes the proof that $\hat \alpha$ is natural in $\weil$.
Thus $\hat \alpha$ is natural in $\mathbb N^\bullet$ and in \weil{} and therefore natural in $\mathbb N^\bullet \times \weil$.
\end{proof}

\bibliography{sample}

\end{document}